\RequirePackage{fix-cm}
\documentclass{svjour3}                     
\smartqed  
\usepackage{graphicx}
\usepackage{amsmath}
\usepackage{amssymb}
\usepackage[english]{babel}
\usepackage[all]{xy}
\usepackage{url}
\usepackage{xcolor}
\journalname{JOTA}
\begin{document}

\title{An optimization problem related to water artificial recirculation for controlling eutrophication\thanks{Research partially funded
by Xunta de Galicia (Spain) under project ED431C 2018/50 {\color{red} and ED431C 2019/02.}}}

\titlerunning{Optimizing water recirculation for eutrophication control}        

\author{Francisco~J.~Fern\'andez      \and
       Aurea~Mart\'inez    \and 
       Lino~J.~Alvarez-V\'azquez 
}

\authorrunning{F.~J. Fern\'andez et al.} 

\institute{F.~J.~Fern\'andez  \at
              Universidade de Santiago de Compostela, 15782 Santiago, Spain \\
              \email{fjavier.fernandez@usc.es}           
           \and
              A.~Mart\'inez and L.~J.~Alvarez-V\'azquez \at
              Universidade de Vigo, 36310 Vigo, Spain \\
              \email{\{aurea, lino\}@dma.uvigo.es}      
}

\date{Received: date / Accepted: date}

\maketitle

\begin{abstract}
In this work, the artificial recirculation of water is presented and analyzed, from the perspective of the optimal control of partial differential equations,
as a tool to prevent eutrophication effects in large waterbodies. A novel formulation of the environmental problem, based on the coupling of nonlinear models
for hydrodynamics, water temperature and concentrations of the different species involved in the eutrophication processes, is introduced. After a complete and 
rigorous analysis of the existence of optimal solutions, a full numerical algorithm for their computation is proposed. Finally, some numerical results for a realistic 
scenario are shown, in order to prove the efficiency of our approach.
\end{abstract}

\section{Introduction: The environmental problem}

Eutrophication is one of the most important problems of large masses of water (estuaries, lakes, reservoirs, etc.) 
and it is caused by undue high levels of nutrients (usually nitrogen and phosphorus) reaching the water. These nutrients 
mainly come from human activities (resulting in the discharge of sewage, detergents, fertilizers and so on, 
very rich in phosphate or nitrate), and can cause an excessive phytoplankton 
growth that lead to undesirable effects like algal blooms. This abnormal growth of algae directly affects 
the concentration of dissolved oxygen, mainly in the deeper layers, since the processes of remineralization of 
organic detritus (accumulated in the bottom due to the effects of sedimentation) consumes 
oxygen, which can lead to oxygen depletion of the body of water \cite{intro1}. 
In Figure~\ref{figure1} (left) we can find a schematic representation of the problem and its consequences.

\begin{figure}[!ht]
\centering
\includegraphics[width=0.45\textwidth]{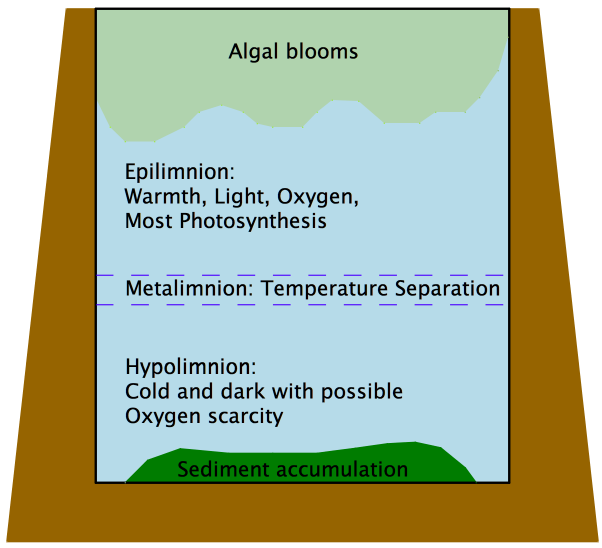} \qquad
\includegraphics[width=0.45\textwidth]{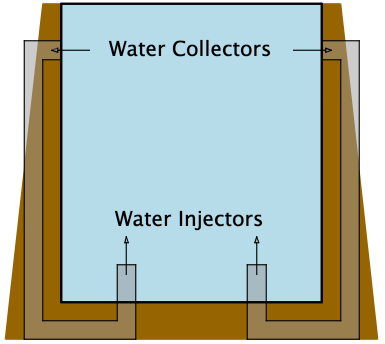}
\caption{On the left side, a diagram representing algal blooms caused by eutrophication and its consequences.
On the right side, a basic scheme depicting the water artificial circulation process.}
\label{figure1}
\end{figure}

Artificial circulation is a management technique for oxygenating eutrophic water bodies subject to 
quality problems, such as loss of oxygen, sediment accumulation and algal blooms. It disrupts stratification 
and minimizes the development of stagnant zones that may be subject to above commented water quality problems. 
In our particular case we are only interested in increasing the dissolved oxygen concentration in the bottom layers
(but our methodology could be extended in a straightforward way to any phenomenon and any region). 
In the process of artificial recirculation, a set of flow pumps takes water from the well aerated upper layers by means of a collector 
and injects it into the poorly oxygenated bottom layers, through a pipeline, setting up a circulation pattern that prevents stratification. 
Then, oxygen-poor water from the bottom is circulated to the surface, where oxygenation from the atmosphere and 
photosynthesis can naturally occur \cite{intro2}. 
In Figure~\ref{figure1} (right) we can find a representation of the main idea of water artificial circulation.

Although eutrophication has received some attention from the mathematical viewpoint in last decade 
(see, for instance, the recent publications \cite{intro3,intro4,intro5} and the references therein), 
the study of artificial circulation as a eutrophication control tool has remained unaddressed in the mathematical literature
up to now, as far as we know (we can only mention a recent paper of the authors \cite{intro2}, where a simplified 
preliminary formulation of the problem is posed and briefly analyzed). 
Thus, in next section we present a detailed mathematical formulation of the physical problem as a control/state constrained 
optimal control problem of nonlinear partial differential equations. Then, in the central part of the paper, we analyze 
the wellposedness of the corresponding state system, in order to demonstrate in a rigorous way the existence of optimal solutions. 
Finally, we present the numerical resolution of the problem, introducing a full computational algorithm and a realistic numerical example,
showing the efficiency of our approach.

\section{Mathematical formulation of the control problem}

In this section we will formulate the environmental problem in the 
framework of optimal control of partial differential equations. For a
better understanding of this novel mathematical formulation, we will 
divide this section into five subsections: in the first subsection we 
will introduce and describe the physical domain; in the second one, 
the control variables (in our case, the volumetric flow 
rate for each pump); in the third subsection we will establish the 
mathematical formulation for the thermo-hydrodynamic model; 
in the fourth one we will present the eutrophication model that 
will be used (the core of our model) and, finally, in the fifth subsection we 
will formulate the optimal control problem.

\subsection{The physical domain}

We consider a domain $\Omega \subset \mathbb{R}^3$ corresponding, for instance, to a reservoir. 
In order to promote the artificial circulation of water inside the domain $\Omega$, we suppose the existence of a set of 
$N_{CT}$ pairs collector-injector $\{(C^k,T^k)\}_{k=1}^{N_{CT}}\subset \partial \Omega$ in such a way that each water collector 
is connected to its corresponding injector by a pipe with a pumping group. We also assume a smooth enough 
boundary $\partial \Omega$, such that it can be split into four disjoint subsets
$\partial \Omega = \Gamma_S \cup \Gamma_C \cup \Gamma_T \cup \Gamma_N$,
where $\displaystyle \Gamma_C$ corresponds to the part of the boundary where the water 
collectors are located ($\Gamma_C=\cup_{k=1}^{N_{CT}} C^k$),
$\displaystyle \Gamma_T$ corresponds to the part of the boundary where the water 
injectors are located ($\Gamma_T=\cup_{k=1}^{N_{CT}}T^k$), $\Gamma_S$ is the top part of the boundary 
in contact with air, and $\displaystyle \Gamma_N=\partial \Omega \setminus \left( \Gamma_S \cup \Gamma_C 
\cup \Gamma_T \right)$ corresponds to the rest of the boundary. 
In particular, we suppose the boundary $\partial \Omega$ regular enough to assure the existence of elements 
$\varphi^k,\, \widetilde{\varphi}^k \in H^{3/2}(\partial \Omega)$, for
$k=1,\ldots,N_{CT}$, satisfying the following assumptions (mainly corresponding 
to suitable regularizations of the indicator functions of $T^k$ and $C^k$, respectively):
\begin{itemize}
\item $\varphi^k(\mathbf{x}),\,\widetilde{\varphi}^k(\mathbf{x})\geq 0$, a.e. $\mathbf{x} \in \partial \Omega$, 
\item $\varphi^k(\mathbf{x})=0$, a.e. $\mathbf{x} \in \partial \Omega \setminus T^k$, 
and $ \int_{T^k} \varphi^k \, d \gamma =\mu(T^k)$,
\item $\widetilde{\varphi}^k(\mathbf{x})=0$, a.e. $\mathbf{x} \in \partial \Omega \setminus C^k$, 
and $ \int_{C^k} \widetilde{\varphi}^k \, d \gamma = \mu(C^k)$,
\end{itemize}
where $\mu(S)$ {\color{red} represents the $n-1$ dimensional measure of a generic set $S$} 
and $\beta_0: u \in H^{3/2}(\partial \Omega) \rightarrow \beta_0(u) \in H^2(\Omega)$ denotes the right inverse of the classical 
trace operator $\gamma_0$, i.e., $(\gamma_0 \circ \beta_0)(u) =u$ (cf. Theorem 8.3. in Chapter 1 of \cite{magenes1}). 
Finally, we also consider a subdomain $\Omega_C \subset \Omega$, corresponding 
to the part of the domain where we want to increase the dissolved oxygen concentration (denoted as control domain in Figure~\ref{figure2}). 

\begin{figure}[!ht]
\centering
\includegraphics[width=0.5\textwidth]{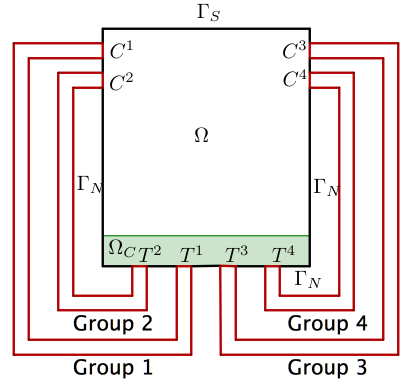}
\caption{Geometrical configuration of an example domain $\Omega$ with $N_{CT}=4$ collector/injector pairs,
showing the different boundary sections: $\Gamma_S$,  $\Gamma_C =\cup_{k=1}^{4} C^k$,  
$\Gamma_T=\cup_{k=1}^{4} T^k$ and $\Gamma_N$, and also the control domain $\Omega_C$.}
\label{figure2}
\end{figure}

\subsection{The control variable}

As above commented, our control will be the volumetric flow rate $({\rm m}^3 \, {\rm s}^{-1})$ by pump 
$k$ at each time $t$, $g^k(t) \in H^1(0,T)$, for $k=1,\ldots, N_{CT}$, where 
$T$ ($\rm{s})$ denotes the length of the time interval.  
We will suppose that the control acts over the system through a 
Dirichlet boundary condition on the hydrodynamic model:
\begin{equation}
\mathbf{v}=\boldsymbol{\phi}_{\mathbf{g}}\quad \mbox{on} \ \partial \Omega \times (0,T),
\end{equation}
where $\mathbf{v}(\mathbf{x},t)$ will denote the water velocity, and where: 
\begin{equation} \label{eq:phig}
\boldsymbol{\phi}_{\mathbf{g}}(\mathbf{x},t)
=\sum_{k=1}^{N_{CT}}g^k(t) \left[\frac{\varphi^k(\mathbf{x})}{\mu(T^k)}
 -\frac{\widetilde{\varphi}^k(\mathbf{x})}{\mu(C^k)} 
\right] \mathbf{n}(\mathbf{x})
\end{equation}
represents the given Dirichlet condition for the hydrodynamic system. 
It is immediate that, thanks to the regularity of the control $\mathbf{g}$ and of the functions 
$\{(\varphi^k,\widetilde{\varphi}^k)\}_{k=1}^{N_{CT}}$, we have that 
$\boldsymbol{\phi}_{\mathbf{g}} \in W^{1,2,2}(0,T;H^{3/2}(\partial \Omega),H^{3/2}(\partial \Omega))$
(cf. expression (\ref{defespacio}) below for a detailed definition of this Sobolev-Bochner space), and also that
\begin{equation} \nonumber
\int_{\partial \Omega} \boldsymbol{\phi}_{\mathbf{g}} \cdot \mathbf{n} \, d \gamma =0.
\end{equation}

\subsection{The thermo-hydrodynamic model}

We denote by $\mathbf{v}(\mathbf{x},t) \ ({\rm m} \, {\rm s}^{-1})$ the solution of the 
following modified Navier-Stokes system with a Smagorinsky model of turbulence:
\begin{equation} \label{eq:system1}
\left\{\begin{array}{l}
\displaystyle
\frac{\partial \mathbf{v}}{\partial t} + \nabla \mathbf{v} \, \mathbf{v} - \text{div} ( \Xi(\mathbf{v}))
+\nabla p = \alpha^0 (\theta-\theta^0) \, \mathbf{a}_g \quad \mbox{in} \; \Omega \times (0,T),  \\
\displaystyle
\nabla \cdot \mathbf{v}=0 \quad \mbox{in} \; \Omega \times (0,T), \\
\displaystyle
\mathbf{v}=\boldsymbol{\phi}_{\mathbf{g}} \quad  \mbox{on}\; \partial \Omega\times (0,T),  \\
\displaystyle
\mathbf{v}(0)=\mathbf{v}^0 \quad \mbox{in}\; \Omega,
\end{array}\right.
\end{equation}
where $\mathbf{a}_g$ $({\rm m} \, {\rm s}^{-2})$ is the gravity acceleration,  
$\alpha^0=-\frac{1}{\rho} \frac{\partial \rho}{\partial \theta}$ $({\rm K}^{-1})$ is the 
thermic expansion coefficient, $\rho$ is the density ,$\mathbf{v}^0$ is the initial 
velocity, and the boundary field $\boldsymbol{\phi}_{\mathbf{g}}$ is the element given by 
(\ref{eq:phig}). The diffusion term $\Xi(\mathbf{v})$ is given by:
\begin{equation}
\Xi(\mathbf{v})=\left. \frac{\partial D(e)}{\partial e} \right|_{e=e(\mathbf{v})},
\ \mbox{ with } e(\mathbf{v})=\frac{1}{2} \left( \nabla \mathbf{v} + \nabla \mathbf{v}^t\right),
\end{equation}
where $D$ is a potential function (for instance, in the standard case of the classical Navier-Stokes equations,
$D(e)=\nu \left[e:e\right]$, with $\nu$ $({\rm m}^2 \, {\rm s}^{-1})$ the kinematic viscosity of the water, and, consequently,
$\Xi(\mathbf{v}) = 2 \nu\, e(\mathbf{v})$).
However, in our case, the Smagorinsky model, the potential function is defined as in \cite{lady1}:
\begin{equation} \label{eq:doperador}
D(e)=\nu \left[e:e\right] + \frac{2}{3} \nu_{tur} \left[e:e\right]^{3/2},
\end{equation}
where $\nu_{tur}$ $({\rm m}^2)$ is the turbulent viscosity. 

Regarding thermic effects, water temperature $\theta(\mathbf{x},t)$ ($\rm{K}$) is the solution of the following convection-diffusion partial 
differential equation with nonhomogeneous, nonlinear, mixed boundary conditions:
\begin{equation} \label{eq:system2}
\left\{\begin{array}{l}
\displaystyle 
\frac{\partial \theta}{\partial t}+ \mathbf{v} \cdot \nabla \theta - 
\nabla \cdot (K \nabla \theta) =0 \quad \mbox{in}\; \Omega \times (0,T), \\ 
\displaystyle
\theta =  \phi_{\theta} \quad \mbox{on} \; \Gamma_T\times (0,T), \\ 
\displaystyle
K \frac{\partial \theta}{\partial \mathbf{n}} = 0 \quad \mbox{on} \; \Gamma_C \times (0,T), \vspace{0.1cm}\\ 
\displaystyle
K \frac{\partial \theta}{\partial \mathbf{n}}= b_1^N (\theta_N-\theta) \quad \mbox{on} \; \Gamma_N \times (0,T), \vspace{0.1cm}\\ 
\displaystyle
K \frac{\partial \theta}{\partial \mathbf{n}} = b_1^S (\theta_S-\theta) +
b_2^S (T_r^4-|\theta|^3\theta) \quad \mbox{on}\; \Gamma_S \times (0,T),  \\ 
\theta(0)=\theta^0  \quad \mbox{in} \; \Omega,
\end{array}\right.
\end{equation}
where Dirichlet boundary condition $\phi_{\theta}$ is given by expression:
\begin{equation} \label{eq:defphi}
\phi_{\theta}(\mathbf{x},t)=\sum_{k=1}^{N_{CT}} \varphi^k(\mathbf{x})
\int_{-T}^T \rho_{\epsilon}(t-\epsilon-s) \, \gamma_{\theta}^k(s) \,  ds
\end{equation}
with, for each $k=1,\ldots,N_{CT}$, 
\begin{equation} \label{eq:gammatheta}
\gamma_{\theta}^k(s)=\left\{
\begin{array}{lcl}
\displaystyle \frac{1}{\mu(C^k)} \int_{C^k} \theta^0 \, d \gamma & \mbox{ if } & s \leq 0,  \\
\displaystyle \frac{1}{\mu(C^k)} \int_{C^k} \theta(s) \, d \gamma & \mbox{ if } & s > 0,
\displaystyle
\end{array}\right.
\end{equation}
representing the mean temperature of water in the collector $C_k$, and with the weight function 
$\rho_{\epsilon}$ defined by:
\begin{equation} \label{eq:rhoeps}
\rho_{\epsilon}(t)=\left\{\begin{array}{lcl}
\displaystyle \frac{c}{\epsilon} \, \exp\left(\frac{t^2}{t^2-\epsilon^2}\right) & \mbox{ if } &
|t| < \epsilon,  \\ 
\displaystyle 0 & \mbox{ if } & |t| \geq \epsilon,
\end{array}\right.
\end{equation}
for $c \in \mathbb{R}$ the positive constant satisfying the unitary condition:
\begin{equation} \nonumber
\displaystyle \int_{\mathbb{R}} \rho_1(t) \, dt =1.
\end{equation}
In other words, we are assuming that the mean temperature of water at each injector $T_k$ is a 
weighted average in time of the mean temperatures of water at its corresponding collector $C_k$. 
In order to obtain the mean temperature at each injector, we convolute the 
mean temperature at the collector with a smooth function with support 
in $(t-2\epsilon,t)$. In this way, we have that the temperature in the injector 
only depends on the mean temperatures in the collector along the time interval 
$(t-2\epsilon,t)$. Parameter $0<\epsilon<T$ represents, in a certain 
sense, the technical characteristics of the pipeline that define the stay time of water in the pipe. 
We also suppose that there is not heat transfer through the walls of the pipelines (that is, they are isolated). 

Moreover, for the other terms appearing in the formulation of problem (\ref{eq:system2}) we have that:
\begin{itemize}
\item $\mathbf{n}$ is the unit outward normal vector to the boundary $\partial \Omega$.
\item $K>0 \ ({\rm m}^2 \, {\rm s}^{-1})$ is the thermal diffusivity 
of the fluid, that is, $K=\frac{\alpha}{\rho\, c_p}$, where 
$\alpha\ ({\rm W} \, {\rm m}^{-1} \, {\rm K}^{-1})$ is the thermal conductivity, 
$\rho\ ({\rm g} \, {\rm m}^{-3})$ is the density, and 
$c_p\ ({\rm W} \, {\rm s} \, {\rm g}^{-1} \, {\rm K}^{-1})$ is the specific heat 
capacity of water.
\item $b_1^K \geq 0 \ ({\rm m} \, {\rm s}^{-1})$, for $K\in\{N,S\}$, are the coefficients 
related to convective heat transfer through the boundaries $\Gamma_N$ 
and $\Gamma_S$, obtained from the relation $\rho \, c_p\, b_1^K=h^K$, 
where $h^K\geq 0 \ ({\rm W} \, {\rm m}^{-2} \, {\rm K}^{-1})$ 
are the convective heat transfer coefficients on each surface.
\item $b_2^S>0\ ({\rm m} \, {\rm s} \, {\rm K}^{-3})$ is the coefficient 
related to radiative heat transfer through the boundary $\Gamma_S$, given by
$b_2^S=\frac{\sigma_B\, \varepsilon}{\rho\, c_p}$, where 
$\sigma_B\ ({\rm W} \, {\rm m}^{-2} \, {\rm K}^{-4})$ 
is the Stefan-Boltzmann constant and $\varepsilon$ is the emissivity.
\item $\theta^0\geq 0 \ ({\rm K})$ represents the initial temperature. 
\item $\theta_N,\, \theta_S \geq 0 \ ({\rm K})$ are the temperatures 
related to convection heat transfer on the surfaces $\Gamma_N$ 
and $\Gamma_S$. 
\item $T_r \geq 0 \ ({\rm K})$ is the radiation temperature on the 
surface $\Gamma_S$, derived from the expression
$\sigma_B \, \varepsilon\, T_r^4 = (1-a) R_{sw,net}+R_{lw,down}$,
where $a$ is the albedo, $R_{sw,net}\ ({\rm W} \, {\rm m}^{-2})$ 
denotes the net incident shortwave radiation on the surface $\Gamma_R$,
and $R_{lw,down} \ ({\rm W} \, {\rm m}^{-2})$ denotes 
the downwelling longwave radiation.
\end{itemize}

\subsection{The eutrophication model}

We consider the following system for modelling the eutrophication processes, based in Michaelis-Menten kinetics 
(further details can be found, for instance, in \cite{fran1,DRAGO200117} and the references therein), where we consider
the concentrations of five different species: $u^1(\mathbf{x},t) \ ({\rm mg} \, {\rm l}^{-1})$ stands for the nutrient (nitrogen in this case),
$u^2(\mathbf{x},t) \ ({\rm mgC} \, {\rm l}^{-1})$ for the phytoplankton, $u^3(\mathbf{x},t) \ ({\rm mgC} \, {\rm l}^{-1})$ for the zooplankton,
$u^4(\mathbf{x},t) \ ({\rm mgC} \, {\rm l}^{-1})$ for the organic detritus, and $u^5(\mathbf{x},t) \ ({\rm mg} \, {\rm l}^{-1})$ for the dissolved oxygen:   
\begin{equation} \label{eq:system3}
\left\{\begin{array}{l}
\displaystyle \frac{\partial u^i}{\partial t} + 
\mathbf{v} \cdot \nabla u^i - 
\nabla \cdot (\mu^i \nabla u^i) = 
A^i(\mathbf{x},t,\theta,\mathbf{u}) \quad \mbox{in}\; \Omega \times (0,T),  \\ 
\displaystyle u^i= \phi_{u^i} \quad \mbox{on}\; \Gamma_T \times (0,T),  \\
\displaystyle 
\mu^i \frac{\partial u^i}{\partial \mathbf{n}}=0 \quad \mbox{on}\; 
(\Gamma_S \cup \Gamma_N \cup \Gamma_C )\times(0,T),  \\ 
\displaystyle 
u^i(0)=u^{0,i} \quad \mbox{in}\; \Omega,  \qquad i=1,\ldots,5,
\end{array}\right.
\end{equation}
where, for $i=1,\ldots,5$,
\begin{equation}
\phi_{u^i}(\mathbf{x},t)=\sum_{k=1}^{N_{CT}} 
\varphi^k(\mathbf{x})
\int_{-T}^T 
\rho_{\epsilon}(t-\epsilon-s) \gamma_{u^i}^k(s) 
\, d s, 
\end{equation}
and, for $k=1,\ldots,N_{CT}$, and $i=1,\ldots,5$,  
\begin{equation}\label{eq:gammaeutro}
\gamma_{u^i}^k(s)=\left\{
\begin{array}{lcl}
\displaystyle \frac{1}{\mu(C^k)} \int_{C^k} u^{0,i} \, d \gamma, 
& \mbox{if} & s \leq 0,  \vspace{0.1cm} \\
\displaystyle \frac{1}{\mu(C^k)} \int_{C^k} u^i(s) \, d \gamma, 
& \mbox{if} & s > 0.
\displaystyle
\end{array}\right.
\end{equation}
Finally, the reaction term 
$\mathbf{A}=(A^i):\Omega \times (0,T) \times \mathbb{R}^6
\rightarrow \mathbb{R}^5$ is defined by the following expression:
\begin{equation} \label{eq:defA}
\mathbf{A} (\mathbf{x},t,\theta,\mathbf{u})= \left[\begin{array}{l}
\displaystyle
- \frac{C_{nc} L(\mathbf{x},t,\theta) }{K_N+|u^1|}u^1 u^2
+C_{nc} K_r u^2   +
C_{nc} K_{rd} D(\theta) u^4 \vspace{0.1cm} \\
\displaystyle
 \frac{L(\mathbf{x},t,\theta)}{K_N+|u^1|}u^1u^2-K_r u^2   -
K_{mf}u^2- \frac{K_z}{K_F+|u^2|}u^2 u^3 \vspace{0.1cm} \\
\displaystyle
 \frac{C_{fz}K_z}{K_F+|u^2|}u^2 u^3-K_{mz}u^3 \vspace{0.1cm}  \\
\displaystyle
K_{mf}u^2+K_{mz}u^3-K_{rd}D(\theta)u^4  \vspace{0.1cm} \\
\displaystyle
\frac{ C_{oc}  L(\mathbf{x},t,\theta)}{K_N+|u^1|}u^1 u^2-C_{oc} K_r u^2 -
C_{oc}K_{rd}D(\theta)u^4
\end{array} \right]
\end{equation}
where:
\begin{itemize}
\item $C_{oc} \geq 0 \ ({\rm mg} \, {\rm mgC}^{-1})$ is the oxygen-carbon stoichiometric relation,
\item $C_{nc} \geq 0 \ ({\rm mg} \, {\rm mgC}^{-1})$ is the nitrogen-carbon stoichiometric relation,
\item $C_{fz} \geq 0$ is the zooplankton grazing efficiency factor,
\item $K_{rd} \geq 0 \ ({\rm s}^{-1})$ is the detritus regeneration rate,
\item $K_r \geq 0 \ ({\rm s}^{-1})$ is the phytoplankton endogenous respiration rate,
\item $K_{mf} \geq 0 \ ({\rm s}^{-1})$ is the phytoplankton death rate,
\item $K_{mz} \geq 0 \ ({\rm s}^{-1})$ is the zooplankton death rate (including predation),
\item $K_z \geq 0 \ ({\rm s}^{-1})$ is the zooplankton predation (grazing),
\item $K_F > 0 \ ({\rm mgC} \, {\rm l}^{-1})$ is the phytoplankton half-saturation constant,
\item $K_N > 0 \ ({\rm mg} \, {\rm l}^{-1})$ is the nitrogen half-saturation constant,
\item $\mu^i \geq 0 \ ({\rm m}^2 \, {\rm s}^{-1})$, $i=1,\ldots,5$, are the diffusion coefficients of each species,
\item $D$ is the thermic regeneration function for the organic detritus, defined as:
\begin{equation}
D(\theta)=\Theta^{\theta-\theta^0},
\end{equation}
with $\log(\Theta) \ ({\rm K}^{-1})$ the thermic regeneration constant for the reference 
temperature $\theta^0$. In order to simplify the mathematical 
analysis of the state equations we will consider the following linear approximation:
\begin{equation}
D(\theta)=1+\log(\Theta)(\theta-\theta^0)
\end{equation}
if $\Theta>0$, and $D(\theta)=1$ if $\Theta=0$.
\item $L$ is the luminosity function, given by:
\begin{equation}
L(\mathbf{x},t,\theta)=\mu \, C_t^{\theta-\theta^0}\, \frac{I^0(t)}{I_s}\, e^{-\varphi_1 x_3}, 
\end{equation}
with $I^0 \ ({\rm W} \, {\rm m}^{-2})$ the incident light intensity, $I_s \ ({\rm W} \, {\rm m}^{-2})$ 
the light saturation, $\log(C_t) \ ({\rm K}^{-1})$ the phytoplankton growth thermic constant for the 
reference temperature $\theta^0$, $\varphi_1 \ ({\rm m}^{-1})$ the light attenuation due to depth,
and $\mu \ ({\rm s}^{-1})$ the maximum phytoplankton growth rate. Again, for the sake of simplicity, 
we will consider the following linear approximation:
\begin{equation}
L(\mathbf{x},t,\theta)=\mu \left(1+\log(C_t)(\theta-\theta^0) 
\right) \frac{I^0(t)}{I_s}\, e^{-\varphi_1 x_3}
\end{equation}
if $C_t>0$, and $L(\mathbf{x},t,\theta)=\mu \, \frac{I^0(t)}{I_s} \, e^{-\varphi_1 x_3}$ if $C_t=0$.
\end{itemize}

\subsection{The optimal control problem}

Our main objective is to {\color{red} ensure that the concentration of dissolved oxygen in the bottom layer is 
in an admissible range by means of an optimal artificial circulation of water from the well aerated upper layer}. 
So, we want to solve the following optimal control problem
\begin{eqnarray*} 
(\mathcal{P}) \qquad 
\min\{J(\mathbf{g}) \, : \, \mathbf{g} \in \mathcal{U}_{ad},\, 
\frac{1}{\mu(\Omega_C)} \int_{\Omega_C} u^5(t) \, d \mathbf{x} \in [\lambda^m,\lambda^M] \},
\end{eqnarray*}
{\color{red} where 
\begin{equation}\label{uad}
\begin{array}{l}
\mathcal{U}_{ad}= \{\mathbf{g} \in [H^1(0,T)]^{N_{CT}} \, : \, \mathbf{g}(0)=\boldsymbol{0},\, \text{and} \\ \displaystyle \qquad 
\|g^k\|_{H^1(0,T)} \leq c, \, \forall k=1,\ldots,N_{CT} \}
\end{array}
\end{equation}}
is the admissible set, $c>0$ is a constant related to technological limitations of the pumps, $J(\mathbf{g})$ is the cost function:
\begin{equation}
J(\mathbf{g})=\frac{1}{2} \sum_{k=1}^{N_{CT}} \int_0^T g^k(t)^2 \, dt
+ \frac{1}{2}\sum_{k=1}^{N_{CT}} \int_0^T \frac{d g^k}{dt}(t)^2 \, dt,
\end{equation}
and $\lambda^m,\lambda^M>0$ represent, respectively, minimum and maximum permissible 
concentrations in the control domain $\Omega_C$. Finally,  
$(\mathbf{v},\theta,\mathbf{u})$ are the solutions of the coupled state systems 
(\ref{eq:system1}), (\ref{eq:system2}) and (\ref{eq:system3}).

\section{Mathematical analysis of the state equations}

In order to establish the appropriate framework for mathematically 
analyzing the coupled state systems (\ref{eq:system1}), (\ref{eq:system2}) and (\ref{eq:system3}), 
we consider, for a Banach space $V_1$ and a locally convex space $V_2$ such that $V_1\subset V_2$, 
and for $1\leq p,q \leq \infty$, the following Sobolev-Bochner space (cf. Chapter 7 of \cite{Roubicek1} for further details):
\begin{equation} \label{defespacio}
W^{1,p,q}(0,T;V_1,V_2)=\left\{u \in L^p(0,T;V_1):\; \frac{du}{dt} \in L^q(0,T;V_2)\right\},
\end{equation}
where $\frac{du}{dt} $ denotes the derivative of $u$ in the sense of distributions. It is well known that, if both $V_1$ and $V_2$ are Banach 
spaces, then $W^{1,p,q}(0,T;V_1,V_2)$ is also a Banach space endowed 
with the norm $\|u\|_{W^{1,p,q}(0,T;V_1,V_2)}=\|u\|_{L^p(0,T;V_1)}+
\left\| \frac{du}{dt}  
\right\|_{L^{q}(0,T;V_2)}$. 

So, for the modified Navier-Stokes system (\ref{eq:system1}) 
we consider the following spaces:
\begin{equation}
\begin{array}{rcl}
\displaystyle
\mathbf{X}_1 &=& \left\{\mathbf{v}\in [W^{1,3}(\Omega)]^3 \, : \, \nabla \cdot \mathbf{v}=0, \, 
\mathbf{v}_{|_{\partial\Omega \setminus (\Gamma_C \cup \Gamma_T)}}= \boldsymbol{0} \right\} , \\
\displaystyle
\widetilde{\mathbf{X}}_1 &=& \left\{\mathbf{v}\in [W^{1,3}(\Omega)]^3 \, : \, 
\nabla \cdot \mathbf{v}=0,  \, \mathbf{v}_{|_{\partial\Omega}}= \boldsymbol{0} \right\} .
\end{array}
\end{equation}
Then, associated to the previous spaces, we define:
\begin{equation}
\begin{array}{rcl}
\displaystyle
\mathbf{W}_1&=&
\displaystyle
W^{1,\infty,2}(0,T;\mathbf{X}_1,[L^2(\Omega)]^3) \cap 
\mathcal{C}([0,T];\mathbf{X}_1), \\
\displaystyle
\widetilde{\mathbf{W}}_1&=&
\displaystyle
W^{1,\infty,2}(0,T;\widetilde{\mathbf{X}}_1,[L^2(\Omega)]^3) \cap 
\mathcal{C}([0,T];\widetilde{\mathbf{X}}_1).
\end{array}
\end{equation}

Now, for the water temperature system (\ref{eq:system2}), we consider the 
following spaces:
\begin{equation} 
\begin{array}{rcl}
X_2&=&\{\theta \in H^1(\Omega) \, : \, \theta_{|_{\Gamma_S}}\in L^5(\Gamma_S)\}, \\ 
\widetilde{X}_2 &=& \{\theta \in X_2 \, : \, \theta_{|_{\Gamma_T}}=0\}.
\end{array}
\end{equation}
If we define the following norm associated to above space $X_2$:
\begin{equation} \nonumber
\|\theta\|_{X_2}=\|\theta\|_{H^1(\Omega)}+\|\theta\|_{L^5(\Gamma_S)},
\end{equation}
we have that $X_2$ is a reflexive separable Banach space 
(cf. Lemma 3.1 of \cite{Zolesio1}), and that $\widetilde{X}_2 \subset L^2(\Omega) 
\subset \widetilde{X}_2'$ is an evolution triple. So, we consider:
\begin{equation} \hspace{-.8cm}
\begin{array}{rcl}
\displaystyle
W_2&=&\{\theta \in W^{1,2,5/4}(0,T;X_2,X_2'): \theta_{|_{\Gamma_S}} \in 
L^5(0,T;L^5(\Gamma_S))\}\cap L^{\infty}(0,T;L^2(\Omega)), \\
\widetilde{W}_2 &=& \{\theta \in W^{1,2,5/4}(0,T;\widetilde{X}_2,\widetilde{X}_2'):
\theta_{|_{\Gamma_S}} \in 
L^5(0,T;L^5(\Gamma_S))\}\cap L^{\infty}(0,T;L^2(\Omega)). 
\end{array}
\end{equation}

Finally, for the eutrophication system (\ref{eq:system3}), we define:
\begin{equation}
\begin{array}{rcl}
\displaystyle \mathbf{X}_3 &=& [H^1(\Omega)]^5 , \\ 
\displaystyle \widetilde{\mathbf{X}}_3 &=& \{\mathbf{u} \in \mathbf{X}_3 \, : \, \mathbf{u}_{|_{\Gamma_T}}=\mathbf{0}\},
\end{array}
\end{equation}
and we consider the following spaces associated to them:
\begin{equation}
\begin{array}{rcl}
\displaystyle \mathbf{W}_3 &=& W^{1,2,2}(0,T;\mathbf{X}_3, \mathbf{X}_3'), \\
\displaystyle \widetilde{\mathbf{W}}_3 &=& W^{1,2,2}(0,T;\widetilde{\mathbf{X}}_3, \widetilde{\mathbf{X}}_3').
\end{array}
\end{equation}

From this section we will assume the following hypotheses for coefficients and data in the
analytical study of the problem:
\begin{itemize}
\item $g^k \in H^1(0,T)$, with $g^k(0)=0$, $\forall k=1,\ldots,N_{CT}$,
\item $\mathbf{v}^0 \in \left[H_{\sigma}^2(\Omega)\right]^3=\{\mathbf{v} \in [H^2(\Omega)]^3 \, : \, \nabla \cdot \mathbf{v}=0, \, 
\mathbf{v}_{|_{\partial \Omega}}=\boldsymbol{0}\}\subset \widetilde{\mathbf{X}}_1$,
\item $\theta_S \in L^2(0,T;L^2(\Gamma_S))$,
\item $\theta_N \in L^2(0,T;L^2(\Gamma_N))$,
\item $T_r \in L^5(0,T;L^5(\Gamma_S))$,
\item $\theta^0 \in X_2$,
\item $I_0 \in L^{\infty}(0,T)$,
\item $\mathbf{u}^0 \in \mathbf{X}_3$.
\end{itemize}

Under these hypotheses we will state now two lemmas (whose demonstrations can be found 
in \cite{fran7} and \cite{fran8}, respectively), which will allow 
us to reformulate the state systems (\ref{eq:system1}), (\ref{eq:system2}) 
and (\ref{eq:system3}) as homogeneous Dirichlet problems.

\begin{lemma} \label{traza1} There exists a linear continuous extension:
\begin{equation} 
\begin{array}{rcl}
R_{\mathbf{v}}: [H^1(0,T)]^{N_{CT}}& \rightarrow & W^{1,2,2}(0,T;[H_{\sigma}^2(\Omega)]^3,[H_{\sigma}^2(\Omega)]^3) \\
\mathbf{g} & \rightarrow &R_{\mathbf{v}}(\mathbf{g} )= \boldsymbol{\zeta}_{\mathbf{g}}
\end{array}	
\end{equation}
such that ${\boldsymbol{\zeta}_{\mathbf{g}}}_{|_{\partial \Omega}}=
\boldsymbol{\phi}_{\mathbf{g}}$, where $\boldsymbol{\phi}_{\mathbf{g}}$ 
is defined by (\ref{eq:phig}), and $H_{\sigma}^2(\Omega)=\{ \mathbf{u} 
\in [H^2(\Omega)]^3 \, : \, \nabla \cdot \mathbf{u}=0\}$.  \qed
\end{lemma}

\begin{remark} It is worthwhile emphasizing here that, thanks to the construction done in the proof of Lemma \ref{traza1}, we have 
\begin{equation} \nonumber
\nu \int_{\Omega} e( \boldsymbol{\zeta}_{\mathbf{g}} ): e( \boldsymbol{\eta})\, d \mathbf{x}=0,\quad
\forall  \boldsymbol{\eta} \in \widetilde{\mathbf{X}}_1,
\end{equation}
and, consequently, this term will not appear in the corresponding variational formulation. \qed
\end{remark}

\begin{lemma} \label{traza2} We have that the following operator is compact
\begin{equation}
\begin{array}{rcl}
R_{\mathbf{h}}:[L^2(0,T)]^{N_{CT}} & \rightarrow & W^{1,2,2}(0,T;H^{2}(\Omega),H^2(\Omega))\\ 
\mathbf{h} & \rightarrow & R_{\mathbf{h}}(\mathbf{h} )=\zeta_{\mathbf{h} },
\end{array}
\end{equation}
where:
\begin{equation}
\zeta_{\mathbf{h} }(\mathbf{x},t)=\sum_{k=1}^{N_{CT}} \beta_0(\varphi^k(\mathbf{x}) )
\int_{-T}^T \rho_{\epsilon} (t-\epsilon-s) 
\gamma^k_{\mathbf{h} }(s) \, ds,
\end{equation}
with $\gamma^k_{\mathbf{h}}(s) \in L^2(-T,T)$, for $k=1,\ldots,N_{CT}$, 
defined by:
\begin{equation}
\gamma^k_{\mathbf{h}}(s)=\left\{
\begin{array}{lcl}
\displaystyle \frac{1}{\mu(C^k)} \int_{C^k} \theta^0 \, d \gamma & \mbox{ if } & s \leq 0,   \\
\displaystyle{h^k}(s) & \mbox{ if } & s > 0,
\end{array}\right.
\end{equation}
and $\beta_0: u \in H^{3/2}(\partial \Omega) \rightarrow \beta_0(u) \in H^2(\Omega)$ the right inverse of the classical 
trace operator $\gamma_0$, i.e., such that $(\gamma_0 \circ \beta_0)(u) =u$ (cf. Theorem 8.3. of \cite{magenes1}). 

We also have the existence of a constant $C$, that depends continuously on the space-time 
computational domain and the initial temperature $\theta^0$, such that:
\begin{equation} \label{eq:acotapf0}
\|\zeta_{\mathbf{h}}\|_{W^{1,2,2}(0,T;H^2(\Omega),H^2(\Omega))} \leq 
C(\theta^0) \, (1+\|\mathbf{h}\|_{[L^2(0,T)]^{N_{CT}}}).
\end{equation}  \qed
\end{lemma}

Now, we will establish the following notations, in order to stablish the homogeneous Dirichlet systems. 
{\color{red} Given elements $(\mathbf{z},\xi,\mathbf{w}) \in \widetilde{\mathbf{W}}_1\times \widetilde{W}_2 \times 
\widetilde{\mathbf{W}}_3$, we define $(\mathbf{v},\theta,\mathbf{u})\in 
\mathbf{W}_1\times W_2 \times \mathbf{W}_3$ in the the following way:
\begin{itemize}
\item $\mathbf{v}=\mathbf{z} + \boldsymbol{\zeta}_{\mathbf{g}} \in \mathbf{W}_1$, 
with $\boldsymbol{\zeta}_{\mathbf{g}} \in W^{1,2,2}(0,T;[H_{\sigma}^2(\Omega)]^3,[H_{\sigma}^2(\Omega)]^3)$ 
the extension of control $\mathbf{g}$ given by Lemma \ref{traza1}.
\item $\theta=\xi+ \zeta_{\mathbf{h}_{\theta}}\in W_2$, 
with $\displaystyle \zeta_{\mathbf{h_{\theta}}} \in W^{1,2,2}(0,T;H^{2}(\Omega),H^{2}(\Omega))$ 
the extension of $\mathbf{h}_{\theta}$ obtained from Lemma \ref{traza2}, where:
\begin{equation} \label{eq:h1}
h_{\theta}^k(s)=\frac{1}{\mu(C^k)} \int_{C^k} \theta(s) \, d \gamma, \quad k=1,\ldots,N_{CT}.
\end{equation}
\item $u^i=w^i + \zeta_{\mathbf{h}^i_{\mathbf{u}}}$, with 
$\zeta_{\mathbf{h}^i_{\mathbf{u}}} \in W^{1,2,2}(0,T;H^2(\Omega),H^2(\Omega))$ 
the extension of $\mathbf{h}^i_{\mathbf{u}}$ obtained from Lemma \ref{traza2} with obvious modifications, where:
\begin{equation} \label{eq:h2}
h_{\mathbf{u}}^{i,k}(s)=\frac{1}{\mu(C^k)} \int_{C^k} u^i(s) \, d \gamma, \quad k=1,\ldots,N_{CT}, \quad i=1,\ldots,5.
\end{equation}
As it is immediate, $\mathbf{w}\in \mathbf{W}_3$.
\end{itemize}
}
Thus, using above notations, we can reformulate the state systems 
(\ref{eq:system1}), (\ref{eq:system2}) and (\ref{eq:system3}) in the following way:
\begin{equation} \label{eq:system1b}
\left\{
\begin{array}{l}
\displaystyle \frac{\partial \mathbf{z}}{\partial t} + \nabla (\boldsymbol{\zeta}_{\mathbf{g}} +
\mathbf{z}) \mathbf{z} +\nabla \mathbf{z} \boldsymbol{\zeta}_{\mathbf{g}} \\ 
\displaystyle  \quad -\text{div}  \left( 2 \nu e(\mathbf{z})+
2 \nu_{tur} \int_{\Omega} \left[e(\boldsymbol{\zeta}_{\mathbf{g}}+\mathbf{z}):
e(\boldsymbol{\zeta}_{\mathbf{g}}+\mathbf{z}) \right]^{1/2} e(\boldsymbol{\zeta}_{\mathbf{g}}+\mathbf{z}) \right) \\ 
\displaystyle \quad +\nabla p = \alpha_0(\theta-\theta^0)\, \mathbf{a}_{g}- \frac{\partial \boldsymbol{\zeta}_{\mathbf{g}}}{\partial t}
-\nabla \boldsymbol{\zeta}_{\mathbf{g}} \boldsymbol{\zeta}_{\mathbf{g}} +2 \nu \nabla \cdot 
e(\boldsymbol{\zeta}_{\mathbf{g}})\quad \mbox{in} \; \Omega \times (0,T), \\
\displaystyle \mathbf{z}  = \mathbf{0} \quad \mbox{on} \; \partial \Omega \times (0,T),  \\
\displaystyle \mathbf{z}(0)=\mathbf{v}^0 \quad \mbox{in}\; \Omega.
\end{array}
\right.
\end{equation}
\begin{equation}\label{eq:system2b}
\left\{\begin{array}{l}
\displaystyle \frac{\partial \xi}{\partial t}+\mathbf{v} \cdot \nabla \xi - \nabla \cdot (K \nabla \xi)  \\ 
\displaystyle \quad = -\frac{\partial \zeta_{\mathbf{h}_{\theta}}}{\partial t}-\mathbf{v} \cdot \nabla \zeta_{\mathbf{h}_{\theta}} + 
\nabla \cdot (K \nabla \zeta_{\mathbf{h}_{\theta}}) \quad \mbox{in} \; \Omega \times (0,T), \\ 
\displaystyle 
\xi=0 \quad \mbox{on} \; T^k \times (0,T),  \quad  \mbox{for} \; k=1,\ldots,N_{CT}, \\ 
\displaystyle K \frac{\partial \xi}{\partial \mathbf{n}}= -K  \frac{\partial \zeta_{\mathbf{h}_{\theta}}}{\partial \mathbf{n}} 
\quad \mbox{on} \; C^k \times (0,T), \quad \mbox{for} \; k=1,\ldots,N_{CT}, \\ 
\displaystyle K \frac{\partial \xi}{\partial \mathbf{n}}=b_1^N\big(
\theta_N-\zeta_{\mathbf{h}_{\theta}}-\frac{K}{b_1^N} \frac{\partial \zeta_{\mathbf{h}_{\theta}}}{\partial \mathbf{n}}-\xi
\big) \quad \mbox{on}\; \Gamma_N \times (0,T), \\ 
\displaystyle K \frac{\partial \xi}{\partial \mathbf{n}}=b_1^S\big(
\theta_S-\zeta_{\mathbf{h}_{\theta}}-\frac{K}{b_1^S} \frac{\partial \zeta_{\mathbf{h}_{\theta}}}{\partial \mathbf{n}}-\xi \big) \\
\displaystyle \quad +b_2^S \big(T_r^4-|\xi+\zeta_{\mathbf{h}_{\theta}}|^3(\xi+\zeta_{\mathbf{h}_{\theta}})\big) 
\quad \mbox{on}\; \Gamma_S \times (0,T), \\ 
\displaystyle \xi(0)=\theta^0-\zeta_{\mathbf{h}_{\theta}}(0) \quad \mbox{in} \; \Omega.
\end{array}\right.
\end{equation}
\begin{equation}\label{eq:system3b}
\left\{\begin{array}{l}
\displaystyle \frac{\partial w^i}{\partial t} + \mathbf{v} \cdot \nabla w^i - \nabla \cdot (\mu^i \nabla w^i) =
A^i(\mathbf{x},t,\theta,\boldsymbol{\zeta}_{\mathbf{h}_{\mathbf{u}}}+\mathbf{w}) \\
\displaystyle \quad -\frac{\partial \zeta_{\mathbf{h}_{\mathbf{u}}^i}}{\partial t}
- \mathbf{v} \cdot \nabla \zeta_{\mathbf{h}_{\mathbf{u}}^i}+\nabla \cdot (\mu^i \nabla 
\zeta_{\mathbf{h}_{\mathbf{u}}^i}) \quad \mbox{in} \; \Omega \times (0,T), \\
\displaystyle \frac{\partial w^i}{\partial \mathbf{n}}=
-\mu^i \frac{\partial \zeta_{\mathbf{h}_{\mathbf{u}}^i}}{\partial \mathbf{n}} \quad \mbox{on} \; 
(\Gamma_S \cup \Gamma_N \cup \Gamma_C )\times(0,T), \\ 
\displaystyle w^i=0 \quad \mbox{on}\; T^k \times (0,T), \quad \mbox{for}\; k=1,\ldots,N_{CT},  \\
\displaystyle w^i(0)=u^{0,i}-\zeta_{\mathbf{h}_{\mathbf{u}}^i}(0) \quad \mbox{in}\; \Omega,  \qquad  i=1,\ldots,5.
\end{array}\right.
\end{equation}

It is worthwhile noting here that all three previous systems show homogeneous Dirichlet 
boundary conditions and, consequently, we will be able to define the concept of solution of 
the original state systems (\ref{eq:system1}), (\ref{eq:system2}) and 
(\ref{eq:system3}) in terms of the modified state systems (\ref{eq:system1b}), 
(\ref{eq:system2b}) and (\ref{eq:system3b}). It should be also noted that, in the case of systems 
(\ref{eq:system2}) and (\ref{eq:system3}), the coupling terms in the Dirichlet boundary conditions are now 
transferred to the partial differential equations in systems (\ref{eq:system2b}) and (\ref{eq:system3b}).

\begin{definition}[The concept of solution for the state systems] \label{definition1} 
An element $(\mathbf{v},\theta,\mathbf{u}) \in \mathbf{W}_1 \times W_2 \times \mathbf{W}_3$ is a solution 
for the state systems (\ref{eq:system1}), (\ref{eq:system2}) and (\ref{eq:system3}), if there exists an element 
$(\mathbf{z},\xi,\mathbf{w})\in \widetilde{\mathbf{W}}_1\times \widetilde{W}_2 \times \widetilde{\mathbf{W}}_3$ such that:
\begin{itemize}
\item $\mathbf{v}=\mathbf{z} + \boldsymbol{\zeta}_{\mathbf{g}}$, 
with $\boldsymbol{\zeta}_{\mathbf{g}}\in W^{1,2,2}(0,T;[H_{\sigma}^2(\Omega)]^3,[H_{\sigma}^2(\Omega)]^3)$ 
as given by Lemma \ref{traza1}, 
$\mathbf{z}(0)=\mathbf{v}^0$, a.e. $\mathbf{x} \in \Omega$, and 
$\mathbf{z}\in \widetilde{\mathbf{W}}_1$ the solution of the following variational formulation:
\begin{equation} \label{eq:system1c}
\hspace{-.7cm} \begin{array}{r}
\displaystyle \int_{\Omega} \frac{\partial \mathbf{z}}{\partial t} \cdot \boldsymbol{\eta} \,d\mathbf{x} +
\int_{\Omega} \nabla (\boldsymbol{\zeta}_{\mathbf{g}}+
\mathbf{z}) \mathbf{z} \cdot \boldsymbol{\eta} \, d \mathbf{x} +
\int_{\Omega} \nabla \mathbf{z} \boldsymbol{\zeta}_{\mathbf{g}} \cdot \boldsymbol{\eta} \, d \mathbf{x}
+ 2 \nu \int_{\Omega} e(\mathbf{z}) : e(\boldsymbol{\eta}) \, d \mathbf{x} \\
\displaystyle + 2 \nu_{tur} \int_{\Omega} \left[e(\boldsymbol{\zeta}_{\mathbf{g}}+\mathbf{z}):
e(\boldsymbol{\zeta}_{\mathbf{g}}+\mathbf{z}) \right]^{1/2} e(\boldsymbol{\zeta}_{\mathbf{g}}+\mathbf{z}):
e(\boldsymbol{\eta}) \, d \mathbf{x} \\
\displaystyle=\int_{\Omega} \mathbf{H}_{\mathbf{g}} \cdot \boldsymbol{\eta} \, d \mathbf{x},
\quad \mbox{a.e.} \ t \in (0,T), \quad \forall \boldsymbol{\eta} \in \widetilde{\mathbf{X}}_1,
\end{array}
\end{equation}
where:
\begin{equation}
\mathbf{H}_{\mathbf{g}}= \alpha_0 (\theta-\theta^0) \, \mathbf{a}_g- \frac{\partial \boldsymbol{\zeta}_{\mathbf{g}}}{\partial t}
-\nabla \boldsymbol{\zeta}_{\mathbf{g}} \boldsymbol{\zeta}_{\mathbf{g}} \in L^2(0,T;[L^2(\Omega)]^3).
\end{equation}
\item $\theta=\xi + \zeta_{\mathbf{h}_{\theta}}$, with $\displaystyle \zeta_{\mathbf{h_{\theta}}}
\in W^{1,2,2}(0,T;H^{2}(\Omega),H^{2}(\Omega))$ obtained from Lemma \ref{traza2} for $\mathbf{h}_{\theta}\in 
[L^2(0,T)]^{N_{CT}}$ defined by (\ref{eq:h1}), 
$ \xi(0)=\theta^0-\zeta_{\mathbf{h}_{\theta}}(0)$, a.e. $\mathbf{x} \in \Omega$, 
and $\xi \in \widetilde{W}_2$ the solution of the following variational formulation:
\begin{equation} \label{eq:system2c}
\begin{array}{r}
\displaystyle \int_{\Omega} \frac{\partial \xi}{\partial t} \eta \, d \mathbf{x} +
\int_{\Omega} \mathbf{v} \cdot \nabla \xi \eta \, d \mathbf{x} + 
K \int_{\Omega} \nabla \xi \cdot \nabla \eta \, d \mathbf{x} +  
b_1^N \int_{\Gamma_N} \xi \eta \, d \gamma \\ \displaystyle
+b_1^S \int_{\Gamma_S} \xi \eta \, d \gamma +
b_2^S \int_{\Gamma_S} |\xi+\zeta_{\mathbf{h}_{\theta}}|^3
(\xi+\zeta_{\mathbf{h}_{\theta}}) \eta \, d \gamma = 
\int_{\Omega} H_{\mathbf{h}_{\theta}}  \eta \, d \mathbf{x} 
\\ \displaystyle
+\int_{\Gamma_C} g_{\mathbf{h}_{\theta}}^{C} \eta \, d \gamma  
+b_1^N \int_{\Gamma_N}   g_{\mathbf{h}_{\theta}}^{N} \eta \, d \gamma 
+b_1^S \int_{\Gamma_S}   g_{\mathbf{h}_{\theta}}^{S} \eta \, d \gamma \\
\displaystyle
+b_2^S \int_{\Gamma_S} T_r^4 \eta \, d \gamma, 
\quad \mbox{a.e.} \ t \in (0,T), \quad \forall \eta \in \widetilde{X}_2,
\end{array}
\end{equation}
where:
\begin{equation}
\begin{array}{rcl}
\displaystyle H_{\mathbf{h}_{\theta}}&=& 
\displaystyle \frac{\partial \zeta_{\mathbf{h}_{\theta}}}{\partial t}-\mathbf{v} \cdot \nabla \zeta_{\mathbf{h}_{\theta}} + 
\nabla \cdot (K \nabla \zeta_{\mathbf{h}_{\theta}}) \in L^2(0,T;L^2(\Omega)), \\
\displaystyle g^C_{\mathbf{h}_{\theta}} &=& \displaystyle -K  \frac{\partial \zeta_{\mathbf{h}_{\theta}}}{\partial \mathbf{n}} \in L^2(0,T;L^2(\Gamma_C)), \\ 
\displaystyle g^N_{\mathbf{h}_{\theta}}&=& \displaystyle \theta_N-
\zeta_{\mathbf{h}_{\theta}}-\frac{K}{b_1^N} \frac{\partial \zeta_{\mathbf{h}_{\theta}}}{\partial \mathbf{n}}
\in L^2(0,T;L^2(\Gamma_N)), \\ 
\displaystyle g^S_{\mathbf{h}_{\theta}}&=& \displaystyle \theta_S-\zeta_{\mathbf{h}_{\theta}}-\frac{K}{b_1^S} \frac{\partial \zeta_{\mathbf{h}_{\theta}}}{\partial \mathbf{n}}
\in L^2(0,T;L^2(\Gamma_S)).
\end{array}
\end{equation}
\item $u^i=w^i + \zeta_{\mathbf{h}^i_{\mathbf{u}}}$, with 
$\zeta_{\mathbf{h}^i_{\mathbf{u}}} \in W^{1,2,2}(0,T;H^2(\Omega), H^2(\Omega))$ obtained from
Lemma \ref{traza2} for $\mathbf{h}^i_{\mathbf{u}} \in [L^2(0,T)]^{N_{CT}}$ defined by (\ref{eq:h2}), 
$ \mathbf{w}(0)=\mathbf{u}_0 - \boldsymbol{\zeta}_{\mathbf{h}_\mathbf{u}}(0)$, a.e. $\mathbf{x} \in \Omega$, 
and $\mathbf{w} \in \widetilde{\mathbf{W}}_3$ the solution of the following variational formulation:
\begin{equation} \label{eq:system3c}
\begin{array}{r}
\displaystyle \int_{\Omega} \frac{\partial \mathbf{w}}{\partial t} \cdot \boldsymbol{\eta} \, d\mathbf{x} + 
\int_{\Omega} \nabla \mathbf{w} \mathbf{v} \cdot  \boldsymbol{\eta} \, d\mathbf{x} +
\Lambda_{\mu} \int_{\Omega}  \nabla \mathbf{w}:\nabla  \boldsymbol{\eta} \, d\mathbf{x} \\ 
\displaystyle = \int_{\Omega} \mathbf{A}(\theta,\boldsymbol{\zeta}_{\mathbf{h}_{\mathbf{u}}}+\mathbf{w}) \cdot \boldsymbol{\eta} \, d \mathbf{x} 
+\int_{\Omega} \mathbf{H}_{\mathbf{u}} \cdot  \boldsymbol{\eta} \, d \mathbf{x} \\ 
\displaystyle +\int_{\Gamma_S \cup \Gamma_N \cup \Gamma_C} \mathbf{g}_{\mathbf{u}} \cdot \boldsymbol{\eta} \, d \gamma, 
\quad \mbox{a.e.} \ t \in (0,T), \quad \forall \boldsymbol{\eta} \in \widetilde{\mathbf{X}}_3,
\end{array}
\end{equation}
where $ \Lambda_{\mu}={\rm diag}(\mu^1,\ldots,\mu^5) \in M_{5 \times 5}(\mathbb{R})$ is a diagonal matrix with diffusion coefficients, and:
\begin{equation}
\begin{array}{rcl}
\displaystyle H_{\mathbf{u}}^i &=& \displaystyle -\frac{\partial \zeta_{\mathbf{h}^i_{\mathbf{u}}}}{\partial t}
- \mathbf{v} \cdot \nabla \zeta_{\mathbf{h}^i_{\mathbf{u}}}+\nabla \cdot (\mu^i \nabla 
\zeta_{\mathbf{h}^i_{\mathbf{u}}}) \in L^2(0,T;L^2(\Omega)), \\
\displaystyle g_{\mathbf{u}}^i &=& \displaystyle -\mu^i \frac{\partial \zeta_{\mathbf{h}^i_{\mathbf{u}}}}{\partial \mathbf{n}} \in 
L^2(0,T;L^2(\Gamma_S \cup \Gamma_N \cup \Gamma_C)), \quad i=1,\ldots,5.
\end{array}
\end{equation} \qed
\end{itemize}
\end{definition}

\begin{remark} We have the following dependence scheme between the elements of state system:
{\color{red}\begin{equation} \nonumber 
\begin{array}{ccccccccccc}
\mathbf{g} & \ \rightarrow & \big( \mathbf{v} & & \longleftrightarrow & & \theta \big) & &\rightarrow && \mathbf{u}.
\end{array}
\end{equation}}
Therefore, we can separate the mathematical analysis of systems (\ref{eq:system1})-(\ref{eq:system2}) from system (\ref{eq:system3}). 
The coupled system (\ref{eq:system1})-(\ref{eq:system2}) has been fully analyzed by the authors in \cite{fran7} and \cite{fran8}. 
Thus, following the results there, we can assure that, for each control $\mathbf{g} \in [H^1(0,T)]^{N_{CT}}$, there exists a solution 
$(\mathbf{v},\theta) \in \widetilde{\mathbf{W}}_1\times \widetilde{W}_2$ of the thermo-hydrodynamic system (\ref{eq:system1})-(\ref{eq:system2}).
We must remark here that, due to the complexity of this nonlinear system, we cannot obtain a uniqueness result for the thermo-hydrodynamic solution
$(\mathbf{v},\theta)$ under our general hypotheses. However, this property will not be necessary in our approach, and previous existence result will be sufficient 
for our argumentation.
So, we can focus now all our attention in analyzing the solution $\mathbf{u}$ of the eutrophication system (\ref{eq:system3}) 
or, equivalently, in studying the solution $\mathbf{w}$ of the modified system (\ref{eq:system3b}). 
\qed
\end{remark}

Thus, in order to analyze the existence of a solution $\mathbf{u}$ by means of a fixed point technique, we consider the operator:
\begin{equation}\label{puntofijo1}
\begin{array}{l}
\displaystyle 
\mathbf{M}_{\mathbf{u}}: (\mathbf{u}^*,\mathbf{h}^*_{\mathbf{u}}) \in 
[L^{2}(0,T;L^2(\Omega))]^5\times [L^2(0,T)]^{5 \times N_{CT}} \longrightarrow  \\ 
\qquad \mathbf{M}_{\mathbf{u}}(\mathbf{u}^*,\mathbf{h}^*_{\mathbf{u}}) = 
(\mathbf{u},\mathbf{h}_{\mathbf{u}}) \in [L^{2}(0,T;L^2(\Omega))]^5\times [L^2(0,T)]^{5 \times N_{CT}},
\end{array}
\end{equation}
where $\mathbf{u}^*=({u^1}^*,\ldots,{u^5}^*)$, with ${u^i}^* \in L^2(0,T;L^2(\Omega))$, for $i=1,\ldots,5$, 
$\mathbf{h}^*_{\mathbf{u}}= ({\mathbf{h}^1_{\mathbf{u}}}^*,\ldots, {\mathbf{h}^5_{\mathbf{u}}}^*)$, with ${\mathbf{h}^i_{\mathbf{u}}}^* 
\in [L^2(0,T)]^{N_{CT}}$, for $i=1,\ldots,5$,  
$\mathbf{u}=(u^1,\ldots,u^5) \in \mathbf{W}_3$, 
$\mathbf{h}_{\mathbf{u}} =(\mathbf{h}^1_{\mathbf{u}},\ldots, \mathbf{h}^5_{\mathbf{u}}) \in [L^2(0,T)]^{5 \times N_{CT}}$, such that:
\begin{itemize}
\item $\zeta_{{\mathbf{h}^{i*}_{\mathbf{u}}}} \in W^{1,2,2}(0,T;H^{2}(\Omega),H^{2}(\Omega))$,
for $i=1,\ldots,5$, is defined by Lemma \ref{traza2}.
\item $\mathbf{u} \in {\mathbf{W}}_3$ is the solution, in the sense of Definition \ref{definition1} with the obvious 
modifications, of the following decoupled problem with resolution order $i=3 \rightarrow 2 \rightarrow 4 \rightarrow 1 \rightarrow 5$:
\begin{equation} \label{eq:system3d1}
\left\{\begin{array}{l}
\displaystyle \frac{\partial u^i}{\partial t} + \mathbf{v} \cdot \nabla u^i - \nabla \cdot (\mu^i \nabla u^i) = 
\widehat{A}^i(\mathbf{x},t,\theta,\mathbf{u}^*,\mathbf{u}) \quad \mbox{in} \; \Omega \times (0,T), \\ 
\displaystyle u^i= \zeta_{{\mathbf{h}^{i*}_{\mathbf{u}}}} \quad  \mbox{on}\; \Gamma_T \times (0,T),  \\
\displaystyle \mu^i \frac{\partial u^i}{\partial \mathbf{n}}=0 \quad \mbox{on}\; 
(\Gamma_S \cup \Gamma_N \cup \Gamma_C )\times(0,T),  \\ 
\displaystyle u^i(0)=u^{0,i} \quad \mbox{in}\; \Omega,  \qquad i=1,\ldots,5,
\end{array}\right.
\end{equation}
where the Caratheodory function $\widehat{\mathbf{A}}=(\widehat{A}^i):\Omega \times (0,T) \times \mathbb{R}^6 \times \mathbb{R}^6
\rightarrow \mathbb{R}^5$ is defined by:
\begin{eqnarray*} \label{eq:defA2} \hspace{-.7cm}
\widehat{\mathbf{A}} (\mathbf{x},t,\theta,\mathbf{u}^*,\mathbf{u})
= \left[\begin{array}{l}
\displaystyle
-C_{nc} L(\mathbf{x},t,\theta) \frac{ {u^1}^* }{K_N+|{u^1}^*|} u^2 +C_{nc} K_r u^2 + C_{nc} K_{rd} D(\theta) u^4 \\
\displaystyle 
L(\mathbf{x},t,\theta) \frac{ {u^1}^* }{K_N+|{u^1}^*|} u^2-K_r u^2 - K_{mf}u^2- K_z\frac{{u^2}^*}{K_F+|{u^2}^*|} u^3 \\
\displaystyle
C_{fz}K_z \frac{ {u^2}^*}{K_F+|{u^2}^*|} u^3-K_{mz}u^3  \\
\displaystyle
K_{mf}u^2+K_{mz}u^3-K_{rd}D(\theta)u^4  \\
\displaystyle 
C_{oc}  L(\mathbf{x},t,\theta) \frac{ {u^1}^*}{K_N+|{u^1}^*|} u^2-C_{oc} K_r u^2 - C_{oc}K_{rd}D(\theta)u^4
\end{array} \right]
\end{eqnarray*}
\item $\mathbf{h}^i_{\mathbf{u}} \in [L^2(0,T)]^{N_{CT}}$, for $i=1,\ldots,5$, is such that:
\begin{equation}
h^{i,k}_{\mathbf{u}}(s)=\frac{1}{\mu(C^k)} \int_{C^k} u^i(s) \, d \gamma, \quad k=1,2,\ldots,N_{CT}. 
\end{equation}
\end{itemize}

{\color{red} \begin{remark} \label{obs:pbtipo} All the five equations of the decoupled system 
(\ref{eq:system3d1}) can be expressed in the way of the following generic equation:
\begin{equation}\label{eq:tipo}
\left\{\begin{array}{l}
\displaystyle
\frac{\partial w}{\partial t} + \mathbf{v} \cdot \nabla w - \nabla \cdot (\mu \nabla w)=k_1 w+k_2 \quad \mbox{in}\; \Omega \times (0,T), \\ 
\displaystyle
\frac{\partial w}{\partial \mathbf{n}}=k_3 \quad \mbox{on}\; \Gamma_1 \times (0,T), \\
\displaystyle 
w=0 \quad \mbox{on}\; \Gamma_2 \times (0,T), \\ 
w(0)=w^0 \quad \mbox{in}\; \Omega,
\end{array}
\right.
\end{equation}
where $\Gamma_1=\Gamma_S \cup \Gamma_C \cup \Gamma_N$, $\Gamma_2=\Gamma_T$, $w^0 \in H^1(\Omega)$, 
$k_1 \in L^4(0,T;L^3(\Omega))$, $k_2 \in L^2(0,T;L^{3/2}(\Omega))$ and $k_3 \in L^2(0,T;L^2(\Gamma_1))$. 

In one hand, $k_3 \in L^2(0,T;L^2(\Gamma_1))$ thanks to Lemma \ref{traza2} and Theorem 3.1 of \cite{fursikov1} for 
the trace operator $\gamma_1= \partial^1 / \partial \mathbf{n}^1$, 
as well $\|\gamma_1(\zeta_{{\mathbf{h}^{i*}_{\mathbf{u}}}})\|_{L^2(0,T;L^2(\Gamma))} \leq C_3 \|{\mathbf{h}^{i*}_{\mathbf{u}}}\|_{[L^2(0,T)]^{N_{CT}}}$, 
$\forall i=1,\ldots,5$. 

In the other hand, for the coefficients $k_1$ and $k_2$, we need to study the particular case for each one of the five species. 
We have the following lemma.
\end{remark}

\begin{lemma} We have the following estimates for the coefficients $ k_1 $ and $ k_2 $ associated to each species:
\begin{itemize}

\item Species $u^3$:
\begin{equation}
\begin{array}{l}
\displaystyle
\|k_1\|_{L^{\infty}(0,T;L^{\infty}(\Omega))} \leq \displaystyle C_1, \\ 
\displaystyle
\|k_2\|_{L^{2}(0,T;L^{2}(\Omega))} \leq
\displaystyle C_2(\|{\mathbf{h}_{\mathbf{u}}^{3*}}\|_{[L^2(0,T)]^{N_{CT}}}\\
\displaystyle \qquad +\|\mathbf{v}\|_{L^2(0,T;[L^3(\Omega)]^3)}\|{\mathbf{h}_{\mathbf{u}}^{3*}}\|_{[L^2(0,T)]^{N_{CT}}} ).
\end{array}
\end{equation}

\item Species $u^2$:
\begin{equation}
\begin{array}{l}
\displaystyle
\|k_1\|_{L^4(0,T;L^3(\Omega))} \leq C_1(1+\|\theta\|_{L^4(0,T;L^3(\Omega))}), \\ 
\displaystyle
\|k_2\|_{L^2(0,T;L^{2}(\Omega))} \leq C_2 (\|u^3\|_{L^2(0,T;L^{2}(\Omega))}\\ 
\displaystyle \qquad +  \|\theta\|_{L^2(0,T;L^2(\Omega))} \|{\mathbf{h}_{\mathbf{u}}^{2*}}\|_{[L^2(0,T)]^{N_{CT}}}
 \\
\displaystyle
\qquad +
\|\mathbf{v}\|_{L^2(0,T;[L^3(\Omega)]^3)}\|{\mathbf{h}_{\mathbf{u}}^{2*}}\|_{[L^2(0,T)]^{N_{CT}}}
+\|{\mathbf{h}_{\mathbf{u}}^{2*}}\|_{[L^2(0,T)]^{N_{CT}}}).
\end{array}
\end{equation}

\item Species $u^4$:
\begin{equation}
\begin{array}{l}
\displaystyle
\|k_1\|_{L^4(0,T;L^3(\Omega))} \leq C_1(1+\|\theta\|_{L^4(0,T;L^3(\Omega))}), \\ 
\displaystyle
\|k_2\|_{L^2(0,T;L^{2}(\Omega))} \leq C_2 (\|u^2\|_{L^2(0,T;L^{2}(\Omega))}+ \|u^3\|_{L^2(0,T;L^{2}(\Omega))} \\ 
\displaystyle \qquad  + \|\theta\|_{L^2(0,T;L^2(\Omega))} \|{\mathbf{h}_{\mathbf{u}}^{4*}}\|_{[L^2(0,T)]^{N_{CT}}} \\
\displaystyle \qquad +\|\mathbf{v}\|_{L^2(0,T;[L^3(\Omega)]^3)}\|{\mathbf{h}_{\mathbf{u}}^{4*}}\|_{[L^2(0,T)]^{N_{CT}}} 
+\|{\mathbf{h}_{\mathbf{u}}^{4*}}\|_{[L^2(0,T)]^{N_{CT}}}).
\end{array}
\end{equation}

\item Species $u^1$, $k_1=0$ and:
\begin{equation}
\begin{array}{l}
\displaystyle 
\|k_2 \|_{L^2(0,T;L^{3/2}(\Omega)} \leq C_2 (\|u^2\|_{L^2(0,T;L^{3/2}(\Omega))}+ \|u^4\|_{L^2(0,T;L^{3/2}(\Omega))}\\ 
\displaystyle
\qquad  +\|u^4\|_{L^4(0,T;L^{3}(\Omega))} \|\theta\|_{L^4(0,T;L^{3}(\Omega))} \\ 
\displaystyle
\qquad +\|u^2\|_{L^4(0,T;L^{3}(\Omega))} \|\theta\|_{L^4(0,T;L^{3}(\Omega))}\\
\displaystyle \qquad + 
\|\mathbf{v}\|_{L^2(0,T;[L^2(\Omega)]^3)}\|{\mathbf{h}_{\mathbf{u}}^{1*}}\|_{[L^2(0,T)]^{N_{CT}}} 
+\|{\mathbf{h}_{\mathbf{u}}^{1*}}\|_{[L^2(0,T)]^{N_{CT}}}) .
\end{array}
\end{equation}

\item Species $u^5$, $k_1=0$ and:
\begin{equation}
\begin{array}{l}
\displaystyle 
\|k_2 \|_{L^2(0,T;L^{3/2}(\Omega)} \leq C_2 ( \|u^2\|_{L^2(0,T;L^{3/2}(\Omega))} + \|u^4\|_{L^2(0,T;L^{3/2}(\Omega))} \\
\displaystyle 
\qquad  +\|u^4\|_{L^4(0,T;L^{3}(\Omega))} \|\theta\|_{L^4(0,T;L^{3}(\Omega))} \\ 
\displaystyle
\qquad +\|u^2\|_{L^4(0,T;L^{3}(\Omega))} \|\theta\|_{L^4(0,T;L^{3}(\Omega))}\\
\displaystyle
\qquad +
\|\mathbf{v}\|_{L^2(0,T;[L^2(\Omega)]^3)}\|{\mathbf{h}_{\mathbf{u}}^{1*}}\|_{[L^2(0,T)]^{N_{CT}}} 
+ \|{\mathbf{h}_{\mathbf{u}}^{5*}}\|_{[L^2(0,T)]^{N_{CT}}}).
\end{array}
\end{equation}
Where $C_1$ and $C_2$ are positive constants that depend on coefficients and data associated to problem 
(\ref{eq:system3}).
\end{itemize}

\end{lemma}

\begin{proof} We will follow the same order of resolution of the decoupled problem:
\begin{itemize}
\item Equation for $u^3$:
\begin{equation} \nonumber \hspace{-.3cm}
\begin{array}{l}
\displaystyle 
k_1 = C_{fz} K_z \frac{{u^{2}}^*}{K_F+|{u^{2}}^*|} - K_{mz}, \\ 
\displaystyle 
k_2 = -\frac{\partial \zeta_{{\mathbf{h}_{\mathbf{u}}^{3*}}}}{\partial t}
- \mathbf{v} \cdot \nabla \zeta_{{\mathbf{h}_{\mathbf{u}}^{3*}}}+\nabla \cdot (\mu^3 \nabla \zeta_{{\mathbf{h}_{\mathbf{u}}^{3*}}}) 
-K_{mz} \zeta_{{\mathbf{h}_{\mathbf{u}}^{3*}}}+
C_{fz} K_z \frac{{u^{2}}^*}{K_F+|{u^{2}}^*|} \zeta_{{\mathbf{h}_{\mathbf{u}}^{3*}}}.
\end{array}
\end{equation}
So, we have that $k_1 \in L^{\infty}(0,T;L^{\infty}(\Omega))$ and $k_2 \in L^{2}(0,T;L^{2}(\Omega))$. 
We also have the following estimates:
\begin{equation}
\begin{array}{l}
\displaystyle
\|k_1\|_{L^{\infty}(0,T;L^{\infty}(\Omega))} \leq \displaystyle C_1, \\ 
\displaystyle
\|k_2\|_{L^{2}(0,T;L^{2}(\Omega))} \leq
\displaystyle C_2(\|{\mathbf{h}_{\mathbf{u}}^{3*}}\|_{[L^2(0,T)]^{N_{CT}}}\\
\displaystyle \qquad +\|\mathbf{v}\|_{L^2(0,T;[L^3(\Omega)]^3)}\|{\mathbf{h}_{\mathbf{u}}^{3*}}\|_{[L^2(0,T)]^{N_{CT}}} ),
\end{array}
\end{equation}
where we have used $\| \zeta_{{\mathbf{h}_{\mathbf{u}}^{3*}}}\|_{W^{1,2,2}(0,T;H^2(\Omega),H^2(\Omega))} 
\leq C(u^{0,3}) \|{\mathbf{h}_{\mathbf{u}}^{3*}}\|_{[L^2(0,T)]^{N_{CT}}}$, in particular, 
$\| \nabla \zeta_{{\mathbf{h}_{\mathbf{u}}^{3*}}}\|_{L^{\infty}(0,T; L^{6}(\Omega))} \leq C(u^{0,3})  
\|{\mathbf{h}_{\mathbf{u}}^{3*}}\|_{[L^2(0,T)]^{N_{CT}}}$
\item Equation for $u^2$:
\begin{equation} \nonumber
\begin{array}{l}
\displaystyle 
k_1 = L(\mathbf{x},t,\theta) \frac{{u^1}^*}{K_N+|{u^1}^*|} - (K_{mf}+K_r), \\ 
\displaystyle
k_2 = - K_z \frac{ {u^2}^*}{K_F+|{u^2}^*|} u^3 -\frac{\partial \zeta_{{\mathbf{h}_{\mathbf{u}}^{2*}}}}{\partial t} 
- \mathbf{v} \cdot \nabla \zeta_{{\mathbf{h}_{\mathbf{u}}^{2*}}} +\nabla \cdot (\mu^2 \nabla \zeta_{{\mathbf{h}_{\mathbf{u}}^{2*}}}) \\ 
\displaystyle
\qquad -(K_{mf}+Kr) \zeta_{{\mathbf{h}_{\mathbf{u}}^{2*}}}
+L(\mathbf{x},t,\theta) \frac{{u^1}^*}{K_N+|{u^1}^*|} \zeta_{{\mathbf{h}_{\mathbf{u}}^{2*}}}.
\end{array}
\end{equation}
In this case, the regularity of the term $k_1$ is imposed by the regularity of the term $L(\mathbf{x},t,\theta)$. 
In particular, we have $W_2 \subset L^2(0,T;L^6(\Omega)) \cap L^{\infty}(0,T; L^2(\Omega))$ $ \subset L^{4}(0,T;L^3(\Omega))$ 
and then, $k_1 \in L^4(0,T;L^3(\Omega))$. 
In the other hand, it is clear that $k_2 \in L^2(0,T;L^{2}(\Omega))$. Finally, we have the following estimates:
\begin{equation}
\begin{array}{l}
\displaystyle
\|k_1\|_{L^4(0,T;L^3(\Omega))} \leq C_1(1+\|\theta\|_{L^4(0,T;L^3(\Omega))}), \\ 
\displaystyle
\|k_2\|_{L^2(0,T;L^{2}(\Omega))} \leq C_2 (\|u^3\|_{L^2(0,T;L^{2}(\Omega))}\\ 
\displaystyle \qquad +  \|\theta\|_{L^2(0,T;L^2(\Omega))} \|{\mathbf{h}_{\mathbf{u}}^{2*}}\|_{[L^2(0,T)]^{N_{CT}}}
 \\
\displaystyle
\qquad +
\|\mathbf{v}\|_{L^2(0,T;[L^3(\Omega)]^3)}\|{\mathbf{h}_{\mathbf{u}}^{2*}}\|_{[L^2(0,T)]^{N_{CT}}}
+\|{\mathbf{h}_{\mathbf{u}}^{2*}}\|_{[L^2(0,T)]^{N_{CT}}}),
\end{array}
\end{equation}
where we have used $\| \zeta_{{\mathbf{h}_{\mathbf{u}}^{2*}}}\|_{W^{1,2,2}(0,T;H^2(\Omega),H^2(\Omega))} 
\leq C(u^{0,2}) \|{\mathbf{h}_{\mathbf{u}}^{2*}}\|_{[L^2(0,T)]^{N_{CT}}}$, 
in particular, $\| \zeta_{{\mathbf{h}_{\mathbf{u}}^{2*}}}\|_{L^{\infty}(0,T; L^{\infty}(\Omega))},\,
\| \nabla \zeta_{{\mathbf{h}_{\mathbf{u}}^{2*}}}\|_{L^{\infty}(0,T; L^{6}(\Omega))}
 \leq C(u^{0,2})  
\|{\mathbf{h}_{\mathbf{u}}^{2*}}\|_{[L^2(0,T)]^{N_{CT}}}$.
\item Equation for $u^4$:
\begin{equation} \nonumber
\begin{array}{l}
\displaystyle 
k_1 = - K_{rd} D(\theta), \\ 
\displaystyle
k_2 = K_{mf} u^2 + K_{mz}u^3- K_{rd} D(\theta)\zeta_{{\mathbf{h}_{\mathbf{u}}^{4*}}} 
-\frac{\partial \zeta_{{\mathbf{h}_{\mathbf{u}}^{4*}}}}{\partial t} - \mathbf{v} \cdot \nabla \zeta_{{\mathbf{h}_{\mathbf{u}}^{4*}}}
+\nabla \cdot (\mu^4 \nabla \zeta_{{\mathbf{h}_{\mathbf{u}}^{4*}}}).
\end{array}
\end{equation}
The regularity of the term $k_1$ is given by the regularity of the water temperature. So, $k_1 \in L^4(0,T;L^3(\Omega))$. 
In the case of $k_2$ we are in the same situation as above 
and then $k_2 \in L^2(0,T;L^2(\Omega))$. We also have the 
following estimates:
\begin{equation}
\begin{array}{l}
\displaystyle
\|k_1\|_{L^4(0,T;L^3(\Omega))} \leq C_1(1+\|\theta\|_{L^4(0,T;L^3(\Omega))}), \\ 
\displaystyle
\|k_2\|_{L^2(0,T;L^{2}(\Omega))} \leq C_2 (\|u^2\|_{L^2(0,T;L^{2}(\Omega))}+ \|u^3\|_{L^2(0,T;L^{2}(\Omega))} \\ 
\displaystyle \qquad  + \|\theta\|_{L^2(0,T;L^2(\Omega))} \|{\mathbf{h}_{\mathbf{u}}^{4*}}\|_{[L^2(0,T)]^{N_{CT}}} \\
\displaystyle \qquad +\|\mathbf{v}\|_{L^2(0,T;[L^3(\Omega)]^3)}\|{\mathbf{h}_{\mathbf{u}}^{4*}}\|_{[L^2(0,T)]^{N_{CT}}} 
+\|{\mathbf{h}_{\mathbf{u}}^{4*}}\|_{[L^2(0,T)]^{N_{CT}}}).
\end{array}
\end{equation}
\item Equation for $u_1$:
\begin{equation} \nonumber
\begin{array}{l}
\displaystyle 
k_1 = 0, \\ 
\displaystyle
k_2 = C_{nc} K_r u^2 + C_{nc} K_{rd} D(\theta) u^4 - C_{nc} L(\mathbf{x},t,\theta)  \frac{{u^1}^*}{K_N+|{u^1}^*|} u^2 \\
\displaystyle  
\qquad -\frac{\partial \zeta_{{\mathbf{h}_{\mathbf{u}}^{1*}}}}{\partial t} - \mathbf{v} \cdot \nabla \zeta_{{\mathbf{h}_{\mathbf{u}}^{1*}}}
+\nabla \cdot (\mu^1 \nabla \zeta_{{\mathbf{h}_{\mathbf{u}}^{1*}}}).
\end{array}
\end{equation}
In the term $k_2$, the most restrictive regularity is determined by the product of two functions ($u^4$ and $\theta$) lying in the space $L^4(0,T;L^3(\Omega))$. So, $k_2 \in L^2(0,T;L^{3/2}(\Omega))$. Besides, we have the following estimate:
\begin{equation}
\begin{array}{l}
\displaystyle 
\|k_2 \|_{L^2(0,T;L^{3/2}(\Omega)} \leq C_2 (\|u^2\|_{L^2(0,T;L^{3/2}(\Omega))}+ \|u^4\|_{L^2(0,T;L^{3/2}(\Omega))}\\ 
\displaystyle
\qquad  +\|u^4\|_{L^4(0,T;L^{3}(\Omega))} \|\theta\|_{L^4(0,T;L^{3}(\Omega))} \\ 
\displaystyle
\qquad +\|u^2\|_{L^4(0,T;L^{3}(\Omega))} \|\theta\|_{L^4(0,T;L^{3}(\Omega))}\\
\displaystyle \qquad + 
\|\mathbf{v}\|_{L^2(0,T;[L^2(\Omega)]^3)}\|{\mathbf{h}_{\mathbf{u}}^{1*}}\|_{[L^2(0,T)]^{N_{CT}}} 
+\|{\mathbf{h}_{\mathbf{u}}^{1*}}\|_{[L^2(0,T)]^{N_{CT}}}) .
\end{array}
\end{equation}
\item Equation for $u_5$:
\begin{equation}
\begin{array}{l} \nonumber
\displaystyle 
k_1 = 0, \\ 
\displaystyle 
k_2= -C_{oc} K_r u^2 - C_{oc}K_{rd}D(\theta)u^4 +C_{oc}  L(\mathbf{x},t,\theta) \frac{ {u^1}^*}{K_N+|{u^1}^*|} u^2 \\ 
\displaystyle 
\qquad -\frac{\partial \zeta_{{\mathbf{h}_{\mathbf{u}}^{5*}}}}{\partial t} - \mathbf{v} \cdot \nabla \zeta_{{\mathbf{h}_{\mathbf{u}}^{5*}}}
+\nabla \cdot (\mu^5 \nabla \zeta_{{\mathbf{h}_{\mathbf{u}}^{5*}}}).
\end{array}
\end{equation}
We are in the same situation as above, in particular, $k_2 \in L^2(0,T;L^{3/2}(\Omega))$, 
and we also have the following estimate:
\begin{equation}
\begin{array}{l}
\displaystyle 
\|k_2 \|_{L^2(0,T;L^{3/2}(\Omega)} \leq C_2 ( \|u^2\|_{L^2(0,T;L^{3/2}(\Omega))} + \|u^4\|_{L^2(0,T;L^{3/2}(\Omega))} \\
\displaystyle 
\qquad  +\|u^4\|_{L^4(0,T;L^{3}(\Omega))} \|\theta\|_{L^4(0,T;L^{3}(\Omega))} \\ 
\displaystyle
\qquad +\|u^2\|_{L^4(0,T;L^{3}(\Omega))} \|\theta\|_{L^4(0,T;L^{3}(\Omega))}\\
\displaystyle
\qquad +
\|\mathbf{v}\|_{L^2(0,T;[L^2(\Omega)]^3)}\|{\mathbf{h}_{\mathbf{u}}^{1*}}\|_{[L^2(0,T)]^{N_{CT}}} 
+ \|{\mathbf{h}_{\mathbf{u}}^{5*}}\|_{[L^2(0,T)]^{N_{CT}}}).
\end{array}
\end{equation} \qed
\end{itemize}
\end{proof}}

Now we can state the following existence result for the generic equation (\ref{eq:tipo}). 
The proof of this result can be done using techniques analogous to those presented in \cite{fran1}:

\begin{theorem} \label{teo:pbtipo} For any given elements $\mathbf{v}\in \mathbf{W}_3$, 
$w_0 \in H^1(\Omega)$, $k_1 \in L^4(0,T;L^3(\Omega))$, 
$k_2 \in L^2(0,T;L^{3/2}(\Omega))$ and
$k_3 \in L^2(0,T;L^2(\Gamma_1))$, there exists a unique 
element $w \in W^{1,2,2}(0,T;H^1_{0,\Gamma_2}(\Omega),
{H^1_{0,\Gamma_2}(\Omega)}') 
\cap L^{\infty}(0,T;L^2(\Omega))$, with $w(0)=w^0$ a.e. 
$x \in \Omega$, that satisfies the following variational formulation:
\begin{equation} \label{eq:soltipo}
\begin{array}{r}
\displaystyle 
\int_{\Omega} \frac{\partial w}{\partial t} \eta \,d \mathbf{x} +\int_{\Omega} \mathbf{v}\cdot \nabla w \eta \, dx+
 \int_{\Omega} \mu \nabla w \cdot \nabla \eta \, d \mathbf{x} =
\int_{\Omega} k_1 w \eta \,d \mathbf{x}   \\ 
\displaystyle +\int_{\Omega} k_2 \eta \, d \mathbf{x}
+\int_{\Gamma_1} 
k_3 \eta \, d \gamma, \quad {\rm a.e.} \, t \in (0,T),\quad \forall \eta \in H^1_{0,\Gamma_2}(\Omega),
\end{array}
\end{equation}
where $ H^1_{0,\Gamma_2}(\Omega)=\{\eta \in H^1(\Omega):\; \eta_{|_{\Gamma_2}}=0\}$. The 
solution also verifies the following estimate:
\begin{equation} \label{eq:acotatipo} \hspace{-.1cm}
\begin{array}{r}
\|w\|_{W^{1,2,2}(0,T;H^1_{0,\Gamma_0}(\Omega),
{H^1_{0,\Gamma_0}(\Omega)}')} \leq C(\|k_1\|_{L^4(0,T;L^3(\Omega))},\|w^0\|_{H^1(\Omega)},
\|\mathbf{v}\|_{\mathbf{W}_1}) \\ 
\displaystyle
\times \bigg[ 1+ \|k_2\|_{L^2(0,T;L^{3/2}(\Omega)}+ \|k_3\|_{L^2(0,T;L^{2}(\Gamma_1))} \bigg],
\end{array}
\end{equation}
where $C$ is a positive constant depending on $k_1$, $w^0$ and $\mathbf{v}$.   \qed
\end{theorem}

Thus, thanks to Remark \ref{obs:pbtipo} and Theorem \ref{teo:pbtipo}, we have the following estimates:

\begin{lemma} \label{remark4} A solution $\mathbf{u}$ of the uncoupled system (\ref{eq:system3d1}) verifies the following:
\begin{itemize}
\item Estimates for $u^3$ and $\mathbf{h}^3_{\mathbf{u}}$:
\begin{eqnarray}
&& \displaystyle
\|u^3\|_{W^{1,2,2}(0,T;H^1(\Omega), H^1(\Omega)')} \leq C_3(\|\mathbf{v}\|_{\mathbf{W}_1}) 
\bigg[1+\|{\mathbf{h}^3_{\mathbf{u}}}^*\|_{[L^2(0,T)]^{N_{CT}}} \bigg], \label{eq:estimate3a} \\
&& \displaystyle
\|\mathbf{h}^3_{\mathbf{u}}\|_{[L^2(0,T)]^{N_{CT}}} \leq C_8(\|\mathbf{v}\|_{\mathbf{W}_1}) 
\bigg[1+\|{\mathbf{h}^3_{\mathbf{u}}}^*\|_{[L^2(0,T)]^{N_{CT}}} \bigg], \label{eq:estimate3b}
\end{eqnarray}
where the constants $C_3$ and $C_8$ also depend on the initial condition $\|u^{0,3}\|_{H^1(\Omega)}$.
\item Estimates for $u^2$ and $\mathbf{h}^2_{\mathbf{u}}$:
\begin{eqnarray} 
&& \displaystyle 
\|u^2\|_{W^{1,2,2}(0,T;H^1(\Omega),H^1(\Omega)')} \leq C_2(\|\mathbf{v}\|_{\mathbf{W}_1}, \|\theta\|_{W_2}) \nonumber \\
&&\displaystyle
\qquad \times \bigg[1+\|{\mathbf{h}^2_{\mathbf{u}}}^*\|_{[L^2(0,T)]^{N_{CT}}} + \|{\mathbf{h}^3_{\mathbf{u}}}^*\|_{[L^2(0,T)]^{N_{CT}}}\bigg], \label{eq:estimate2a} \\
&&\displaystyle 
\|\mathbf{h}^2_{\mathbf{u}}\|_{[L^2(0,T)]^{N_{CT}}} \leq C_7(\|\mathbf{v}\|_{\mathbf{W}_1}, \|\theta\|_{W_2}) \nonumber \\
&&\displaystyle
\qquad \times \bigg[1+\|{\mathbf{h}^2_{\mathbf{u}}}^*\|_{[L^2(0,T)]^{N_{CT}}} + \|{\mathbf{h}^3_{\mathbf{u}}}^*\|_{[L^2(0,T)]^{N_{CT}}}\bigg], \label{eq:estimate2b}
\end{eqnarray}
where constants $C_2$ and $C_7$ also depend on the initial conditions $\|u^{0,2}\|_{H^1(\Omega)}$ and $\|u^{0,3}\|_{H^1(\Omega)}$.
\item Estimates for $u^4$ and $\mathbf{h}^4_{\mathbf{u}}$:
\begin{eqnarray}
&& \hspace{-.3cm} \displaystyle
\|u^4\|_{W^{1,2,2}(0,T;H^1(\Omega),H^1(\Omega)')} \leq C_4(\|\mathbf{v}\|_{\mathbf{W}_1}, \|\theta\|_{W_2}) \nonumber \\
&&\displaystyle
\times \bigg[1+\|{\mathbf{h}^2_{\mathbf{u}}}^*\|_{[L^2(0,T)]^{N_{CT}}} +
\|{\mathbf{h}^3_{\mathbf{u}}}^*\|_{[L^2(0,T)]^{N_{CT}}} + \|{\mathbf{h}^4_{\mathbf{u}}}^*\|_{[L^2(0,T)]^{N_{CT}}} \bigg], \label{eq:estimate4a} \\
&& \hspace{-.3cm} \displaystyle
\|\mathbf{h}^4_{\mathbf{u}}\|_{[L^2(0,T)]^{N_{CT}}} \leq C_9(\|\mathbf{v}\|_{\mathbf{W}_1}, \|\theta\|_{W_2}) \nonumber \\
&&\displaystyle
\times \bigg[1+\|{\mathbf{h}^2_{\mathbf{u}}}^*\|_{[L^2(0,T)]^{N_{CT}}} +
\|{\mathbf{h}^3_{\mathbf{u}}}^*\|_{[L^2(0,T)]^{N_{CT}}}+\|{\mathbf{h}^4_{\mathbf{u}}}^*\|_{[L^2(0,T)]^{N_{CT}}} \bigg], \label{eq:estimate4b}
\end{eqnarray}
where the constants $C_4$ and $C_9$ also depend on the initial conditions $\|u^{0,2}\|_{H^1(\Omega)}$, $\|u^{0,3}\|_{H^1(\Omega)}$ 
and $\|u^{0,4}\|_{H^1(\Omega)}$.
\item Estimates for $u^1$ and $\mathbf{h}^1_{\mathbf{u}}$:
\begin{eqnarray}
&& \displaystyle
\|u^1\|_{W^{1,2,2}(0,T;H^1(\Omega),H^1(\Omega)')} \leq C_1(\|\mathbf{v}\|_{\mathbf{W}_1}, \|\theta\|_{W_2}) 
\bigg[1+\|{\mathbf{h}^1_{\mathbf{u}}}^*\|_{[L^2(0,T)]^{N_{CT}}} \nonumber \\
&& \displaystyle 
\quad + \|{\mathbf{h}^2_{\mathbf{u}}}^*\|_{[L^2(0,T)]^{N_{CT}}} + \|{\mathbf{h}^3_{\mathbf{u}}}^*\|_{[L^2(0,T)]^{N_{CT}}} 
+\|{\mathbf{h}^4_{\mathbf{u}}}^*\|_{[L^2(0,T)]^{N_{CT}}} \bigg], \label{eq:estimate1a} \\
&& \displaystyle 
\|\mathbf{h}^1_{\mathbf{u}}\|_{[L^2(0,T)]^{N_{CT}}} \leq C_6(\|\mathbf{v}\|_{\mathbf{W}_1}, \|\theta\|_{W_2})
\bigg[1+\|{\mathbf{h}^1_{\mathbf{u}}}^*\|_{[L^2(0,T)]^{N_{CT}}}\nonumber \\
&& \displaystyle 
\quad + \|{\mathbf{h}^2_{\mathbf{u}}}^*\|_{[L^2(0,T)]^{N_{CT}}} + \|{\mathbf{h}^3_{\mathbf{u}}}^*\|_{[L^2(0,T)]^{N_{CT}}} 
+\|{\mathbf{h}^4_{\mathbf{u}}}^*\|_{[L^2(0,T)]^{N_{CT}}} \bigg], \label{eq:estimate1b}
\end{eqnarray}
where constants $C_1$ and $C_6$ also depend on the initial conditions $\|u^{0,1}\|_{H^1(\Omega)}$, 
$\|u^{0,2}\|_{H^1(\Omega)}$, $\|u^{0,3}\|_{H^1(\Omega)}$ and $\|u^{0,4}\|_{H^1(\Omega)}$. 
\item Estimates for $u^5$ and $\mathbf{h}^5_{\mathbf{u}}$:
\begin{eqnarray}
&& \displaystyle
\|u^5\|_{W^{1,2,2}(0,T;H^1(\Omega),H^1(\Omega)')} \leq C_5(\|\mathbf{v}\|_{\mathbf{W}_1}, \|\theta\|_{W_2})
\bigg[1+\|{\mathbf{h}^2_{\mathbf{u}}}^*\|_{[L^2(0,T)]^{N_{CT}}} \nonumber \\
&& \displaystyle 
\quad + \|{\mathbf{h}^3_{\mathbf{u}}}^*\|_{[L^2(0,T)]^{N_{CT}}} + \|{\mathbf{h}^4_{\mathbf{u}}}^*\|_{[L^2(0,T)]^{N_{CT}}} 
+\|{\mathbf{h}^5_{\mathbf{u}}}^*\|_{[L^2(0,T)]^{N_{CT}}} \bigg],  \label{eq:estimate5a} \\
&& \displaystyle 
\|\mathbf{h}^5_{\mathbf{u}}\|_{[L^2(0,T)]^{N_{CT}}} \leq C_{10}(\|\mathbf{v}\|_{\mathbf{W}_1}, \|\theta\|_{W_2})
\bigg[1+\|{\mathbf{h}^2_{\mathbf{u}}}^*\|_{[L^2(0,T)]^{N_{CT}}} \nonumber \\
&& \displaystyle 
\quad + \|{\mathbf{h}^3_{\mathbf{u}}}^*\|_{[L^2(0,T)]^{N_{CT}}} + \|{\mathbf{h}^4_{\mathbf{u}}}^*\|_{[L^2(0,T)]^{N_{CT}}} 
+\|{\mathbf{h}^5_{\mathbf{u}}}^*\|_{[L^2(0,T)]^{N_{CT}}} \bigg], \label{eq:estimate5b}
\end{eqnarray}
where the constants $C_5$ and $C_{10}$ also depend on the initial conditions $\|u^{0,2}\|_{H^1(\Omega)}$, 
$\|u^{0,3}\|_{H^1(\Omega)}$, $\|u^{0,4}\|_{H^1(\Omega)}$ and $\|u^{0,5}\|_{H^1(\Omega)}$.     \qed
\end{itemize}
\end{lemma}

\begin{remark}
We must to note here that all above estimates do not depend on the variable $\mathbf{u}^*$, since the dependence on 
$\mathbf{u}^*$ appears within terms of the form:
\begin{displaymath}
\frac{{u^k}^*}{K+|{u^k}^*|},
\end{displaymath}
with $K>0$, and those terms are bounded a.e. $(\mathbf{x},t)\in \Omega \times (0,T)$ by a constant independent on $\mathbf{u}^*$. \qed
\end{remark}

Now, we will prove the main result of this Section:

\begin{theorem}[Existence of solution for the eutrophication system] \label{existencesystem3} 
If there exist coefficients and data such that:
\begin{equation}
\begin{array}{rcl}
C_6(\|\mathbf{v}\|_{\mathbf{W}_1}, \|\theta\|_{W_2})& < & 1, \\
C_7(\|\mathbf{v}\|_{\mathbf{W}_1}, \|\theta\|_{W_2})& < & 1, \\
C_8(\|\mathbf{v}\|_{\mathbf{W}_1})  & < & 1, \\
C_9(\|\mathbf{v}\|_{\mathbf{W}_1}, \|\theta\|_{W_2})& < & 1, \\
C_{10}(\|\mathbf{v}\|_{\mathbf{W}_1}, \|\theta\|_{W_2})& < & 1,\\
\end{array}
\end{equation} 
for all $\mathbf{g} \in \mathcal{U}_{ad}$, then there will exist positive constants $\widetilde{C}_i$, $i=1,\ldots,10$, such that the operator 
$M_{\mathbf{u}}:\mathbf{B}_{\mathbf{u}} \rightarrow \mathbf{B}_{\mathbf{u}}$ defined by (\ref{puntofijo1}) 
has a fixed point, which is solution of the state system (\ref{eq:system1}), (\ref{eq:system2}) and (\ref{eq:system3}) in the sense of 
Definition \ref{definition1}, where:
\begin{equation} \label{eq:defBpfu}
\begin{array}{rcl}
\mathbf{B}_{\mathbf{u}}&=&\Big\{ (\mathbf{u},\mathbf{h}_{\mathbf{u}}) \in [L^{2}(0,T;L^2(\Omega))]^5\times [L^2(0,T)]^{5 \times N_{CT}} \, : \\ 
&& \|u^i\|_{L^2(0,T;L^2(\Omega))} \leq \widetilde{C}_i, \; \forall i=1,\ldots,5, \\ 
&& \|\mathbf{h}^i_{\mathbf{u}}\|_{[L^2(0,T)]^{N_{CT}}} \leq \widetilde{C}_{5+i},\; \forall i=1,\ldots,5 \Big\}.
\end{array}
\end{equation} 
\end{theorem}

\begin{proof} In order to apply the Schauder fixed point Theorem (see, for instance, Theorem 9.5 of \cite{conway1}), we will 
prove that the operator $M_{\mathbf{u}}$ is compact and that there exist positive constants $\widetilde{C}_i$, $i=1,\ldots,10$, such 
that the operator $M_{\mathbf{u}}$ maps elements from the set $\mathbf{B}_{\mathbf{u}}$ (which is closed and convex) into itself.

\begin{itemize}
\item The operator $M_{\mathbf{u}}$ is compact in the sense that 
it is continuous and $\overline{M_{\mathbf{u}}(A)}$ is compact whenever 
$A$ is a bounded subset of $[L^2(0,T;L^2(\Omega)]^5 \times [L^2(0,T)]^{5 \times N_{CT}}$:

\

In fact, given a convergent sequence $\{(\mathbf{u}^*_n,\mathbf{h}_n^*)\}_{n \in \mathbb{N}}\subset [L^2(0,T;L^2(\Omega)]^5 \times [L^2(0,T)]^{5 \times N_{CT}}$ 
such that $\mathbf{u}^*_n \rightarrow \mathbf{u}^*$ in $[L^2(0,T;L^2(\Omega))]^5$ and $\mathbf{h}_n^* \rightarrow \mathbf{h}_{\mathbf{u}}^*$ in $[L^2(0,T)]^{5 \times N_{CT}}$, 
we have that $(\mathbf{u}_n,\mathbf{h}_n) = M_{\mathbf{u}}(\mathbf{u}^*_n,\mathbf{h}_n^*)
\in \mathbf{W}_3 \times [H^1(0,T)]^{5 \times N_{CT}}$ is such that $\mathbf{u}_n=\mathbf{w}_n+\boldsymbol{\zeta}_{
\mathbf{h}_{n}^*}$, with $\mathbf{w}_n \in \widetilde{\mathbf{W}}_3$ the solution of the following variational formulation:
\begin{equation} \label{eq:system3d}
\begin{array}{r}
\displaystyle 
\int_{\Omega} \frac{\partial \mathbf{w}_n}{\partial t} \cdot \boldsymbol{\eta} \, d\mathbf{x} + 
\int_{\Omega} \nabla \mathbf{w}_n \mathbf{v} \cdot  \boldsymbol{\eta} \, d\mathbf{x} +
\Lambda_{\mu} \int_{\Omega}  \nabla \mathbf{w}_n:\nabla  \boldsymbol{\eta} \, d\mathbf{x} \\
\displaystyle 
= \int_{\Omega} \widehat{\mathbf{A}}(\theta,\mathbf{u}^*_n,\boldsymbol{\zeta}_{\mathbf{h}_n^*}+
\mathbf{w}_n) \cdot \boldsymbol{\eta} \, d \mathbf{x} +\int_{\Omega} \mathbf{H}_{n} \cdot  \boldsymbol{\eta} \, d \mathbf{x} \\ 
\displaystyle 
+\int_{\Gamma_S \cup \Gamma_N \cup \Gamma_C} \mathbf{g}_{n} \cdot \boldsymbol{\eta} \, d \gamma, 
\quad \mbox{a.e.} \ t \in (0,T), \quad \forall \boldsymbol{\eta} \in \widetilde{\mathbf{X}}_3,
\end{array}
\end{equation}
where:
\begin{equation}
\begin{array}{l}
\displaystyle H_{n}^i = \displaystyle -\frac{\partial \zeta_{{\mathbf{h}^*_n}^i}}{\partial t}
- \mathbf{v} \cdot \nabla \zeta_{{\mathbf{h}^*_n}^i}+\nabla \cdot (\mu^i \nabla \zeta_{{\mathbf{h}^*_n}^i}) \in L^2(0,T;L^2(\Omega)), \\
\displaystyle 
g_{n}^i = -\mu^i \frac{\partial \zeta_{{\mathbf{h}^*_{n}}^i}}{\partial \mathbf{n}} \in 
L^2(0,T;L^2(\Gamma_S \cup \Gamma_N \cup \Gamma_C)), \\ 
\displaystyle
h_n^{i,k} = \frac{1}{\mu(C^k)} \int_{C^k} u_n^i \, d \gamma \in H^1(0,T), \quad i=1,\ldots,5, \quad k=1,\ldots,N_{CT} .
\end{array}
\end{equation}
Using Lemma \ref{traza2}, we know that 
$\boldsymbol{\zeta}_{\mathbf{h}_{n}^*} \rightarrow \boldsymbol{\zeta}_{\mathbf{h}^*}$ in $W^{1,2,2}(0,T;[H^2(\Omega)]^5,[H^2(\Omega)]^5)$, 
$\mathbf{H}_{n} \rightarrow \mathbf{H}_{\mathbf{u}}$ in $[L^2(0,T;[L^2(\Omega))]^5$, 
and $\mathbf{g}_n \rightarrow \mathbf{g}_{\mathbf{u}}$ in $[L^2(0,T)]^{5\times N_{CT}}$. 
Thus, taking subsequences if necessary, we have that $\mathbf{u}_n^* \rightarrow \mathbf{u}^*$ a.e. $(x,t) \in \Omega \times (0,T)$, 
$\frac{\partial \mathbf{w}_n}{\partial t} \rightharpoonup \frac{\partial\mathbf{w}}{\partial t}$ weakly in $L^2(0,T;\widetilde{\mathbf{X}}_3')$, 
$\mathbf{w}_n \rightharpoonup \mathbf{w}$ weakly in $L^2(0,T;\widetilde{\mathbf{X}}_3)$, 
and $\mathbf{w}_n \rightarrow \mathbf{w}$ strongly in $[L^{10/3-\epsilon}(0,T,L^{10/3-\epsilon}(\Omega))]^5$, for all $\epsilon>0$. 
It is straightforward to prove, using previous convergences, that we can pass to the limit in variational formulation (\ref{eq:system3d}) 
obtaining that $M_{\mathbf{u}} (\mathbf{u}_n^*,\mathbf{h}^*_n) \rightarrow M_{\mathbf{u}}(\mathbf{u}^*,\mathbf{h}^*)$ in $[L^2(0,T;L^2(\Omega))]^5 
\times [L^2(0,T)]^{5 \times N_{CT}}$, with $(\mathbf{u},\mathbf{h}_{\mathbf{u}})=M_{\mathbf{u}}(\mathbf{u}^*,\mathbf{h}^*)$. 
We also derive that $\mathbf{u}=\mathbf{w}+\boldsymbol{\zeta}_{\mathbf{h}_{\mathbf{u}}^*} \in \mathbf{W}_3$, with $\mathbf{w} \in 
\widetilde{\mathbf{W}}_3$ the solution of the following variational formulation:
\begin{equation} \label{eq:system3e}
\begin{array}{r}
\displaystyle 
\int_{\Omega} \frac{\partial \mathbf{w}}{\partial t} \cdot \boldsymbol{\eta} \, d\mathbf{x} + 
\int_{\Omega} \nabla \mathbf{w} \mathbf{v} \cdot  \boldsymbol{\eta} \, d\mathbf{x} +
\Lambda_{\mu} \int_{\Omega}  \nabla \mathbf{w}:\nabla  \boldsymbol{\eta} \, d\mathbf{x} \\
\displaystyle 
= \int_{\Omega} \widehat{\mathbf{A}}(\theta,\mathbf{u}^*,\boldsymbol{\zeta}_{\mathbf{h}_{\mathbf{u}}^*}+
\mathbf{w}) \cdot \boldsymbol{\eta} \, d \mathbf{x} +\int_{\Omega} \mathbf{H} \cdot  \boldsymbol{\eta} \, d \mathbf{x} \\
\displaystyle 
+\int_{\Gamma_S \cup \Gamma_N \cup \Gamma_C} \mathbf{g}\cdot \boldsymbol{\eta} \, d \gamma, 
\quad \mbox{a.e.} \ t \in (0,T), \quad \forall \boldsymbol{\eta} \in \widetilde{\mathbf{X}}_3,
\end{array}
\end{equation}
where:
\begin{equation}
\begin{array}{l}
\displaystyle H^i = -\frac{\partial \zeta_{{\mathbf{h}^*_{\mathbf{u}}}^i}}{\partial t} - \mathbf{v} \cdot \nabla \zeta_{{\mathbf{h}^*_{\mathbf{u}}}^i}+
\nabla \cdot (\mu^i \nabla \zeta_{{\mathbf{h}^*_{\mathbf{u}}}^i}) \in L^2(0,T;L^2(\Omega)), \\
\displaystyle 
g_{n}^i = -\mu^i \frac{\partial \zeta_{{\mathbf{h}^*_{\mathbf{u}}}^i}}{\partial \mathbf{n}} \in 
L^2(0,T;L^2(\Gamma_S \cup \Gamma_N \cup \Gamma_C)), \quad i=1,\ldots,5.
\end{array} 
\end{equation}
Finally, the compactness of operator $M_{\mathbf{u}}$ is a direct consequence of the compact embedding of 
space $\mathbf{W}_3 \times [H^1(0,T)]^{5 \times N_{CT}}$ into space $[L^2(0,T;L^2(\Omega))]^5 \times [L^2(0,T)]^{5 \times N_{CT}}$.

\item There exist positive constants $\widetilde{C}_i$, $i=1,\ldots,10$, such that the operator $M_{\mathbf{u}}$ applies elements from the set 
$\mathbf{B}_{\mathbf{u}}$ into itself:

\

We only need to prove that, if we define the following constants:
\begin{equation} \hspace{-.6cm} 
\begin{array}{ll}
\widetilde{C}_{1} = \displaystyle \frac{C_1}{(1-C_6)(1-C_7)(1-C_8)(1-C_9)},
& \widetilde{C}_{6} = \displaystyle \frac{C_6}{(1-C_6)(1-C_7)(1-C_8)(1-C_9)}, \\
\widetilde{C}_{2} = \displaystyle \frac{C_2}{(1-C_7)(1-C_8)},
& \widetilde{C}_{7} = \displaystyle \frac{C_7}{(1-C_7)(1-C_8)}, \\
\widetilde{C}_{3} = \displaystyle \frac{C_3}{(1-C_8)},
& \widetilde{C}_{8} = \displaystyle \frac{C_8}{(1-C_8)}, \\
\widetilde{C}_{4} =  \displaystyle \frac{C_4}{(1-C_7)(1-C_8)(1-C_9)},
& \widetilde{C}_{9} = \displaystyle \frac{C_9}{(1-C_7)(1-C_8)(1-C_9)}, \\
\widetilde{C}_{5} = \displaystyle \frac{C_5}{(1-C_7)(1-C_8)(1-C_9)(1-C_{10})},
& \widetilde{C}_{10} = \displaystyle \frac{C_{10}}{(1-C_7)(1-C_8)(1-C_9)(1-C_{10})},
\end{array} \nonumber
\end{equation}
then $(\mathbf{u},\mathbf{h}_{\mathbf{u}})=M_{\mathbf{u}}(\mathbf{u}^*,\mathbf{h}_{\mathbf{u}}^*) \in \mathbf{B}_{\mathbf{u}}$, 
for all $(\mathbf{u}^*,\mathbf{h}_{\mathbf{u}}^*) \in \mathbf{B}_{\mathbf{u}}$. Indeed, given an element 
$(\mathbf{u}^*,\mathbf{h}_{\mathbf{u}}^*) \in \mathbf{B}_{\mathbf{u}}$, we have that, thanks to the estimates 
(\ref{eq:estimate3a})-(\ref{eq:estimate5b}), we can easily obtain by simple (but tedious) algebraic computations that 
$(\mathbf{u},\mathbf{h}_{\mathbf{u}})=M_{\mathbf{u}}(\mathbf{u}^*,\mathbf{h}_{\mathbf{u}}^*)$ satisfies the following estimates:
\begin{equation}
\begin{array}{l} \nonumber
\displaystyle
\|u^1\|_{W^{1,2,2}(0,T;H^1(\Omega),H^1(\Omega)')} \leq C_1 \bigg[1+ \frac{C_6}{(1-C_6)(1-C_7)(1-C_8)(1-C_9)} \\
\displaystyle
\quad +\frac{C_7}{(1-C_7)(1-C_8)} + \frac{C_8}{(1-C_8)}+ \frac{C_9}{(1-C_7)(1-C_8)(1-C_9)} \bigg] = \widetilde{C}_1, \\ 
\displaystyle
\|\mathbf{h}_{\mathbf{u}}^1\|_{[L^5(0,T)]^{N_{CT}}} \leq C_6 \bigg[1+\frac{C_6}{(1-C_6)(1-C_7)(1-C_8)(1-C_9)} \\
\displaystyle
\quad + \frac{C_7}{(1-C_7)(1-C_8)} + \frac{C_8}{(1-C_8)}+ \frac{C_9}{(1-C_7)(1-C_8)(1-C_9)}\bigg] = \widetilde{C}_6, \\ 
\displaystyle
\|u^2\|_{W^{1,2,2}(0,T;H^1(\Omega),H^1(\Omega)')} \leq C_2 \bigg[1+
\frac{C_7}{(1-C_7)(1-C_8)}+ \frac{C_8}{(1-C_8)} \bigg] = \widetilde{C}_2,  \\ 
\displaystyle
\|\mathbf{h}^2_{\mathbf{u}}\|_{[L^2(0,T)]^{N_{CT}}} \leq C_7 \bigg[1+ \frac{C_7}{(1-C_7)(1-C_8)}+\frac{C_8}{(1-C_8)} \bigg] = \widetilde{C}_7, \\
\displaystyle
\|u^3\|_{W^{1,2,2}(0,T;H^1(\Omega),H^1(\Omega)')} \leq C_3 \bigg[1+\frac{C_8}{(1-C_8)} \bigg] = \widetilde{C}_3, \\ 
\displaystyle
\|\mathbf{h}^3_{\mathbf{u}}\|_{[L^2(0,T)]^{N_{CT}}} \leq C_8 \bigg[1+\frac{C_8}{(1-C_8)} \bigg] = \widetilde{C}_8, \\
\displaystyle
\|u^4\|_{W^{1,2,2}(0,T;H^1(\Omega),H^1(\Omega)')} \leq C_4 \bigg[1+ \frac{C_7}{(1-C_7)(1-C_8)} \\
\displaystyle
\quad + \frac{C_8}{(1-C_8)} +\frac{C_9}{(1-C_7)(1-C_8)(1-C_9)} \bigg]=\widetilde{C}_4,  \\ 
\displaystyle
\|\mathbf{h}^4_{\mathbf{u}}\|_{[L^2(0,T)]^{N_{CT}}} \leq C_9 \bigg[1+ \frac{C_7}{(1-C_7)(1-C_8)} \\
\displaystyle
\quad + \frac{C_8}{(1-C_8)} + \frac{C_9}{(1-C_7)(1-C_8)(1-C_9)} \bigg]=\widetilde{C}_9, \\
\displaystyle
\|u^5\|_{W^{1,2,2}(0,T;H^1(\Omega),H^1(\Omega)')} \leq C_5 \bigg[
\frac{C_7}{(1-C_7)(1-C_8)}+ \frac{C_8}{(1-C_8)} \\
\displaystyle
\quad + \frac{C_9}{(1-C_7)(1-C_8)(1-C_9)} + \frac{C_{10}}{(1-C_7)(1-C_8)(1-C_9)(1-C_{10})} \bigg] = \widetilde{C}_5, 
\end{array}
\end{equation}
\begin{equation} \nonumber
\begin{array}{l}
\displaystyle
\|\mathbf{h}^5_{\mathbf{u}}\|_{[L^2(0,T)]^{N_{CT}}} \leq C_{10} \bigg[
\frac{C_7}{(1-C_7)(1-C_8)}+\frac{C_8}{(1-C_8)} \\
\displaystyle
\quad + \frac{C_9}{(1-C_7)(1-C_8)(1-C_9)} + \frac{C_{10}}{(1-C_7)(1-C_8)(1-C_9)(1-C_{10})} \bigg]=\widetilde{C}_{10}.
\end{array}
\end{equation}
Thus, $M_{\mathbf{u}}(\mathbf{u}^*,\mathbf{h}_{\mathbf{u}}^*) \in \mathbf{B}_{\mathbf{u}}$
\end{itemize}
Then, as a direct consequence of the Schauder Theorem, we obtain the existence of a fixed point 
$(\mathbf{u},\mathbf{h}_{\mathbf{u}}) \in \mathbf{B}_{\mathbf{u}}$, which is, from the construction of operator $M_{\mathbf{u}}$,
a solution of problem (\ref{eq:system3}).   \qed
\end{proof}


\section{Mathematical analysis of the optimal control problem}

In this section we will prove the existence of solution for the optimal control problem $(\mathcal{P})$. 
It is important to remark here that, since we have not demonstrated the uniqueness of solution for 
the state systems (\ref{eq:system1}), (\ref{eq:system2}) and (\ref{eq:system3}), we will treat the problem as a 
multistate control problem (cf. \cite{fran9}). Thus, we define the set:
\begin{equation}
\begin{array}{rcl}
\mathcal{U}&=&\displaystyle
\big\{ (\mathbf{v},\theta,\mathbf{u},\mathbf{g}) \in  L^3(0,T;\mathbf{X}_1) \times 
L^2(0,T;X_2) \times L^2(0,T;\mathbf{X}_3) \times \mathcal{U}_{ad} \, : \\ 
&& \displaystyle
(\mathbf{v},\theta,\mathbf{u}) \; \text{is a solution of (\ref{eq:system1}), (\ref{eq:system2}) and (\ref{eq:system3}) associated to} \; \mathbf{g} , \\ 
&& \displaystyle 
{\rm verifying} \; \frac{1}{\mu(\Omega_C)} \int_{\Omega_C} u^5(t) \, d\mathbf{x} \in [\lambda^m,\lambda^M], \ \forall t \in [0,T] \big\},
\end{array}
\end{equation}
where the set of admissible controls $\mathcal{U}_{ad}$ is bounded, convex and closed (in particular, $\mathcal{U}_{ad}$ is weakly closed). 
We observe that the constraints in the sets $\mathcal{U}$ and $\mathcal{U}_{ad}$ are 
well defined since $g^k \in H^1(0,T) \subset \subset\mathcal{C}([0,T])$, $k=1,\ldots,N_{CT}$,
and ${u}^i \in \mathcal{C}([0,T]; L^2(\Omega))$, $i=1,\ldots,5$. Then, we prove the following property for the set $\mathcal{U}$.

\begin{lemma} The set $\mathcal{U}$ is weakly closed. 
\end{lemma}

\begin{proof} Let us consider a sequence of elements
$\{(\mathbf{v}_n,\theta_n,\mathbf{u}_n,\mathbf{g}_n)\}_{n \in \mathbb{N}}\subset \mathcal{U}$ such that
$(\mathbf{v}_n,\theta_n,\mathbf{u}_n,\mathbf{g}_n) \rightharpoonup (\mathbf{v},\theta,\mathbf{u}, \mathbf{g})$ in
$L^3(0,T;\mathbf{X}_1) \times L^2(0,T;X_2) \times L^2(0,T;\mathbf{X}_3) \times \mathcal{U}_{ad}.$
In particular, the sequence $\{\mathbf{g}_n\}_{n \in \mathbb{N}}$ is bounded in $[H^1(0,T)]^{N_{CT}}$ and then, 
thanks to the estimates obtained in Lemma 7 of \cite{fran7}, in Theorem 9 of \cite{fran8}, and 
in above Lemmas \ref{traza1}, \ref{traza2} and \ref{remark4}, we have that the sequence 
$\{(\mathbf{z}_n, \boldsymbol{\zeta}_{\mathbf{g}_n}, \xi_n,\zeta_{\mathbf{h}_{\theta_n}} \mathbf{w}_n,
\boldsymbol{\zeta}_{\mathbf{h}_{\mathbf{u}_n}} )\}_{n \in \mathbb{N}} \subset
\widetilde{\mathbf{W}}_1 \times W^{1,2,2}(0,T;[H_{\sigma}^2(\Omega)]^3, [H_{\sigma}^2(\Omega)]^3) \times 
\widetilde{W}_2 \times W^{1,2,2}(0,T;H^2(\Omega), H^2(\Omega))\times 
\widetilde{\mathbf{W}}_3 \times W^{1,2,2}(0,T;[H^2(\Omega)]^5,[H^2(\Omega)]^5) $
induced by Definition \ref{definition1} is bounded, where, 
for all $n \in \mathbb{N}$, 
$\mathbf{v}_n=\mathbf{w}_n + \boldsymbol{\zeta}_{\mathbf{g}_n}$, 
$\theta_n=\xi_n +\zeta_{\mathbf{h}_{\theta_n}}$ and 
$\mathbf{u}_n=\mathbf{w}_n+\boldsymbol{\zeta}_{
\mathbf{h}_{\mathbf{u}_n}}.$ 

Now, if we denote by 
$\mathbf{w}=\mathbf{v}-\boldsymbol{\zeta}_{\mathbf{g}}$, 
$\xi=\theta-\zeta_{\mathbf{h}_{\theta}}$ and 
$\mathbf{w}=\mathbf{u}-\boldsymbol{\zeta}_{\mathbf{h}_{\mathbf{u}}}$, 
we have (taking subsequences if necessary) the following convergences for the elements associated to the sequence of controls:
\begin{itemize}
\item $g_n(t) \rightarrow g(t)$ strongly for all $t \in [0,T]$ (so, in particular, $g_n(0)\rightarrow g(0)$ and, consequently, $g(0)=0$),
\item $\displaystyle \boldsymbol{\zeta}_{\mathbf{g}_n} \rightharpoonup \boldsymbol{\zeta}_{\mathbf{g}}$ weakly in 
$W^{1,2,2}(0,T;[H_{\sigma}^2(\Omega)]^3, [H_{\sigma}^2(\Omega)]^3)$,
\item $\displaystyle \zeta_{\mathbf{h}_{\theta_n}} \rightarrow \zeta_{\mathbf{h}_{\theta}}$ strongly in $W^{1,2,2}(0,T;H^2(\Omega), H^2(\Omega))$, 
\item $\displaystyle \boldsymbol{\zeta}_{\mathbf{h}_{\mathbf{u}_n}} \rightarrow \boldsymbol{\zeta}_{\mathbf{h}_{\mathbf{u}}}$ strongly 
in $W^{1,2,2}(0,T;[H^2(\Omega)]^5,[H^2(\Omega)]^5)$,
\end{itemize}
where the first convergence is a direct consequence of compactness of $H^1(0,T)$ in $\mathcal{C}([0,T])$ and the two last convergences 
are consequence of Lemma \ref{traza2}. 
In a similar way we also have, for the sequence $\{\mathbf{z}_n\}_{n \in \mathbb{N}}$, the following convergences:
\begin{itemize}
\item $\mathbf{z}_n \rightarrow \mathbf{z}$ strongly in $L^p(0,T;[L^q(\Omega)]^3)$ for all $1<p<\infty$ and $2\leq q <\infty$,
\item $\displaystyle \mathbf{z}_n \rightharpoonup \mathbf{z}$ weakly in $L^3(0,T;\widetilde{\mathbf{X}}_1)$,
\item $\displaystyle \frac{d \mathbf{z}_n}{dt} \rightharpoonup \frac{d \mathbf{z}}{dt}$ weakly in $L^2(0,T;[L^2(\Omega)]^3)$,
\item $\displaystyle \nabla \mathbf{z}_n \rightharpoonup^* \nabla \mathbf{z}$ weakly-$*$ in $L^{\infty}(0,T;[L^3(\Omega)]^3)$,
\item $\displaystyle \beta(e(\boldsymbol{\zeta}_{\mathbf{g}_n} +\mathbf{z}_n)) \, e(\boldsymbol{\zeta}_{\mathbf{g}_n}+
\mathbf{z}_n)\rightharpoonup \widehat{\beta}$, weakly in $L^{3/2}(0,T;\widetilde{\mathbf{X}}_1')$.
\end{itemize}
Moreover, for the sequence $\{\xi_n\}_{n \in \mathbb{N}}$ we have:
\begin{itemize}
\item $\displaystyle \xi_n \rightharpoonup \xi$ in $L^2(0,T;\widetilde{X}_2)$, 
\item $\displaystyle \xi_n \rightharpoonup^* \xi$ in $L^{\infty}(0,T;L^2(\Omega))$, 
\item $\displaystyle \xi_n \rightarrow \xi$ in $L^{10/3-\epsilon}(0,T;L^{10/3-\epsilon}(\Omega))$, for all $\epsilon >0$ small enough,
\item $\displaystyle \xi_n \rightarrow \xi$ in $L^2(0,T;L^2(\Gamma_C))$,
\item $\displaystyle \xi_n \rightarrow \xi$ in $L^4(0,T;L^4(\Gamma_S))$.
\end{itemize}
Finally, for the sequence $\{\mathbf{w}_n\}_{n \in 
\mathbb{N}}$, we have the following convergences:
\begin{itemize}
\item $\displaystyle \mathbf{w}_n \rightharpoonup \mathbf{w}$ weakly in $L^2(0,T;\widetilde{\mathbf{X}}_3)$, 
\item $\displaystyle \frac{d \mathbf{w}_n}{dt} \rightharpoonup \frac{d \mathbf{w}}{dt}$ weakly in $L^2(0,T;\widetilde{\mathbf{X}}_3')$, 
\item $\displaystyle \mathbf{w}_n \rightarrow \mathbf{w}$ strongly in $[L^{10/3-\epsilon}(0,T;L^{10/3-\epsilon}(\Omega))]^5$, for all $\epsilon >0$ small enough,
\item $\displaystyle \mathbf{w}_n \rightarrow \mathbf{w}$ strongly in $[L^2(0,T;L^2(\Gamma_C))]^5$.
\end{itemize}

So, we are able to pass to the limit in the corresponding variational formulations using the same arguments that we have employed 
for proving the compactness of operator $M_{\mathbf{u}}$, and in the Galerkin approximations for systems 
(\ref{eq:system1}) and (\ref{eq:system2}) (cf. \cite{fran7} and \cite{fran8}). 
The only difficulty here is to prove that 
$\widehat{\beta}=\beta(e(\boldsymbol{\zeta}_{\mathbf{g}} +\mathbf{z})) \, e(\boldsymbol{\zeta}_{\mathbf{g}}+\mathbf{z})$. 
However, by Lemma 4.2 of \cite{fran7} we have that
\begin{eqnarray*} \hspace{-.1cm}
\int_0^T \int_{\Omega}\left[ \beta(e(\boldsymbol{\zeta}_{\mathbf{g}_n}+\mathbf{z}_n))  
e(\boldsymbol{\zeta}_{\mathbf{g}_n}+\mathbf{z}_n)- \beta(e(\boldsymbol{\zeta}_{\mathbf{g}_n}+ \boldsymbol{\eta})) 
e(\boldsymbol{\zeta}_{\mathbf{g}_n}+\boldsymbol{\eta})\right] :  e(\mathbf{z}_n-\boldsymbol{\eta}) \, d \mathbf{x}\, dt \geq 0
\end{eqnarray*}
for all $\boldsymbol{\eta} \in L^3(0,T;\widetilde{\mathbf{X}}_1)$, and then, using similar 
techniques that we can find in the proof of Theorem 4.3 of \cite{fran7}, we can prove that
\begin{equation} \nonumber 
\int_0^T \int_{\Omega} \left[ \widehat{\beta}- \beta(e(\boldsymbol{\zeta}_{\mathbf{g}}+ \boldsymbol{\eta})) 
e(\boldsymbol{\zeta}_{\mathbf{g}}+\boldsymbol{\eta})\right] :  e(\mathbf{z}-\boldsymbol{\eta})  \, d \mathbf{x}\, dt  \geq 0,
\end{equation}
for all $\boldsymbol{\eta}\in L^3(0,T;\widetilde{\mathbf{X}}_1)$.  Finally, choosing $\boldsymbol{\eta}=\mathbf{z} \pm \lambda \boldsymbol{\zeta}$, 
with $\boldsymbol{\zeta} \in  L^3(0,T;\widetilde{\mathbf{X}}_1)$ and $\lambda$ an arbitrary positive number, 
and multiplying both sides of the inequality by $\lambda^{-1}$, we obtain
\begin{equation} \nonumber
\begin{array}{l}
\displaystyle
\int_0^T \int_{\Omega} \left[ \widehat{\beta}- \beta(e(\boldsymbol{\zeta}_{\mathbf{g}}+ \mathbf{z}+\lambda \boldsymbol{\zeta})) 
e(\boldsymbol{\zeta}_{\mathbf{g}}+ \mathbf{z}+\lambda \boldsymbol{\zeta})\right] :  e(\boldsymbol{\zeta})  \, d \mathbf{x}\, dt  \leq  0 \\
\displaystyle
\int_0^T \int_{\Omega} \left[ \widehat{\beta}- \beta(e(\boldsymbol{\zeta}_{\mathbf{g}}+ \mathbf{z}-\lambda \boldsymbol{\zeta})) 
e(\boldsymbol{\zeta}_{\mathbf{g}}+ \mathbf{z}-\lambda \boldsymbol{\zeta})\right] :  e(\boldsymbol{\zeta})  \, d \mathbf{x}\, dt  \geq  0.
\end{array}
\end{equation}
Now, letting $\lambda$ tend to zero, we deduce that, for all $\boldsymbol{\zeta} \in  L^3(0,T;\widetilde{\mathbf{X}}_1)$:
\begin{equation}
\int_0^T \int_{\Omega} 
\left[ \widehat{\beta}-
 \beta(e(\boldsymbol{\zeta}_{\mathbf{g}}+
 \mathbf{z})) 
e(\boldsymbol{\zeta}_{\mathbf{g}}+
\mathbf{z})\right] :  
e(\boldsymbol{\zeta})  \, d \mathbf{x}\, dt  =0.
\end{equation}
Thus, $\widehat{\beta}= \beta(e(\boldsymbol{\zeta}_{\mathbf{g}}+ \mathbf{z})) e(\boldsymbol{\zeta}_{\mathbf{g}}+ \mathbf{z})$ 
a.e. $(\mathbf{x},t) \in \Omega \times ]0,T[$, and then, $(\mathbf{v},\theta,\mathbf{u})$ is a solution of the systems 
(\ref{eq:system1}), (\ref{eq:system2}) and (\ref{eq:system3}) associated to the control $\mathbf{g}$. 

Finally, by the strong convergence of $\{\mathbf{u}_n\}_{n \in \mathbb{N}}$ in $[L^2(0,T;L^2(\Gamma_C))]^5$, we have
\begin{equation}
\frac{1}{\mu(\Omega_C)} \int_{\Omega_C} u^5(t) \, 
d\mathbf{x} \in [\lambda^m,\lambda^M], \quad \forall t \in [0,T],
\end{equation}
and, consequently, the element $(\mathbf{v},\theta,\mathbf{u},\mathbf{g})\in \mathcal{U}$.      \qed
\end{proof}

\begin{theorem} [Existence of optimal solution]
The optimal control problem $(\mathcal{P})$ has, at least, a solution.
\end{theorem}

\begin{proof} Let us consider a minimizing sequence $\{(\mathbf{v}_n,\theta_n,\mathbf{u}_n,\mathbf{g}_n)\}_{n \in \mathbb{N}}\subset \mathcal{U}$. 
Then, $\{\mathbf{g}_n\}_{n \in \mathbb{N}}$ is bounded in $[H^1(0,T)]^{N_{CT}}$, which implies, thanks to the estimates 
(\ref{eq:estimate1a}), (\ref{eq:estimate2a}), (\ref{eq:estimate3a})and  (\ref{eq:estimate4a}), and to the Hypotheses of Theorem \ref{existencesystem3}, 
that the sequence $\{\mathbf{u}_n\}_{n \in \mathbb{N}}$ is bounded in $\mathbf{W}_3$. 
We also have, thanks to estimates obtained in \cite{fran7} and \cite{fran8} that the sequence $\{(\mathbf{v}_n,\theta_n)\}_{n \in \mathbb{N}}$
is also bounded in $\mathbf{W}_1 \times W_2$. 
Thus, we can take a subsequence of $\{(\mathbf{v}_n,\theta_n,\mathbf{u}_n,\mathbf{g}_n)\}_{n \in \mathbb{N}}\subset \mathcal{U}$, 
still denoted in the same way, such that $(\mathbf{v}_n,\theta_n,\mathbf{u}_n,\mathbf{g}_n) \rightharpoonup 
(\widetilde{\mathbf{v}},\widetilde{\theta}, \widetilde{\mathbf{u}}, \widetilde{\mathbf{g}})$ in 
$L^3(0,T;\mathbf{X}_1) \times L^2(0,T;X_2) \times L^2(0,T;\mathbf{X}_3) \times \mathcal{U}_{ad}.$
Moreover, from previous Lemma, we have that $(\widetilde{\mathbf{v}},\widetilde{\theta}, \widetilde{\mathbf{u}}, \widetilde{\mathbf{g}}) \in \mathcal{U}$. 

Finally, due to the continuity and the convexity of the cost functional $J$ (in particular, $J$ is weakly lower semicontinuous), we deduce that:
\begin{equation} \nonumber
J(\widetilde{\mathbf{g}}) \leq \liminf_{n \to \infty} J(\mathbf{g}_n)=
\inf_{(\mathbf{v},\theta,\mathbf{u},\mathbf{g}) \in \mathcal{U}} J(\mathbf{g}) \leq J(\widetilde{\mathbf{g}}).
\end{equation}
Thus, $\widetilde{\mathbf{g}} \in \mathcal{U}_{ad}$ is a solution of the optimal control problem $(\mathcal{P})$.   \qed
\end{proof}

\begin{remark} It is worthwhile remarking here that, using standard techniques in the spirit of those presented in below section,
it is possible to obtain a formal optimality system for the characterization of the optimal solutions of the control problem $(\mathcal{P})$.
However, since this is not the main aim of this paper, and for the sake of brevity, we will not present here this optimality system,
focusing our attention on the numerical computation of these optimal solutions. \qed
\end{remark}


\section{Numerical resolution of the control problem}

In this section we will present a numerical approximation for the optimal control problem $(\mathcal{P})$. 
So, we will discretize the state systems (\ref{eq:system1}), (\ref{eq:system2}) and (\ref{eq:system3}) using a standard finite element method, 
and we will compute the numerical approximation of the resulting nonlinear optimization problem (that appears after 
the full space-time discretization of the control problem) using an interior point algorithm. 
In this particular case, due to the specific relations between the dimensions of the control and the constraint variables, 
the numerical approximation of the Jacobian matrix of the constraints will be performed using the discretized adjoint system 
(row by row) instead of the linearized systems (column by column). In addition, the computation of each row of the Jacobian matrix 
will be parallelized.

\subsection{Space-time discretization}

Let us consider a regular partition $0=t_0<t_1<\ldots<t_N=T$ of the time interval $[0,T]$ such that
$t_{n+1}-t_n=\Delta t=\frac{1}{\alpha}$, $\forall n =0,\ldots,N-1$, and recall the material derivative of a generic scalar field 
$\phi$ defined as:
\begin{equation}
\frac{D \phi}{Dt}(\mathbf{x},t)= \frac{\partial}{\partial t} \phi(\mathbf{X}(\mathbf{x},t),t)=
\frac{\partial \phi}{\partial t}(\mathbf{x},t)+ \mathbf{v}(\mathbf{x},t) \cdot \nabla \phi(\mathbf{x},t),
\end{equation} 
where $\mathbf{X}$ represents the characteristic line, that is, verifies the equation:
\begin{equation}
\frac{\partial \mathbf{X}}{\partial t}(\mathbf{x},t)= \mathbf{v}(\mathbf{x},t).
\end{equation}
So, we can approximate the material derivative in the following way:
\begin{equation} \label{eq:materialderivative}
\frac{D \phi}{Dt}(t_{n+1}) \simeq \alpha \left(\phi^{n+1}-\phi^n \circ \mathbf{X}^n_{-}\right),
\end{equation}
where $\phi^n$ represents an approximation to $\phi(t_n)$, and $\mathbf{X}^n_{-}(\mathbf{x})=\mathbf{X}(\mathbf{x},t_{n+1};t_{n})$ 
(i.e., the position at time $t_{n}$ of a particle that at time $t_{n+1}$ was located at point $\mathbf{x}$) 
is the solution of the following trajectory equation:
\begin{equation}
\left\{
\begin{array}{l}
\displaystyle \frac{d \mathbf{X}}{d \tau}=\mathbf{v}(\mathbf{X}(\mathbf{x},t;\tau),\tau),  \\ 
\displaystyle \mathbf{X}(\mathbf{x},t;t)=\mathbf{x},
\end{array}
\right.
\end{equation}
approached by the Euler scheme, that is, we consider the following approximation 
(see further details, for instance, in \cite{cita1,cita2}):
\begin{equation}
(\phi^n \circ \mathbf{X}^n_{-})(\mathbf{x}) \simeq \phi^n(\mathbf{x}-\Delta t\,  \mathbf{v}^n(\mathbf{x}) ).
\end{equation}

For the space discretization, we take a family of meshes $\tau_h$ for the domain $\Omega$ with characteristic size $h$ 
and, associated to this family of meshes, we define the following finite element spaces (cf. Section 4.1 of \cite{girault1}):
\begin{itemize}
\item $\mathbf{V}_h$ ($\mathbb{P}_{1b}$ FEM space) for the water velocity $\mathbf{v}$:
\begin{equation}
\mathbf{V}_h=\{\mathbf{v} \in [\mathcal{C}(\overline{\Omega})]^3 \; : \; \mathbf{v}_{_{\tau}}\in [\mathbb{P}_{1b}(\tau)]^3,\; \forall \tau \in \tau_h,
\ \mathbf{v}_{|_{\partial \Omega \setminus (\Gamma_T \cup \Gamma_C)}}=\mathbf{0}\},
\end{equation}
and, for the test functions and the adjoint state, the subspace:
\begin{equation}
\mathbf{W}_h=\{\mathbf{w} \in \mathbf{V}_h \; : \; \mathbf{w}_{|_{\Gamma_T}}=\mathbf{0}\}.
\end{equation}
\item $M_h$ ($\mathbb{P}_1$ FEM space) for the water pressure $p$:
\begin{equation}
M_h=\{p \in \mathcal{C}(\overline{\Omega})\; :\; p_{|_{\tau}} \in \mathbb{P}_1(\tau),\; \forall \tau \in \tau_h\}.
\end{equation}
\item $K_h$ ($\mathbb{P}_1$ FEM space) for the water temperature $\theta$:
\begin{equation}
K_h=\{\theta \in \mathcal{C}(\overline{\Omega})\; :\; \theta_{|_{\tau}} \in \mathbb{P}_1(\tau),\; \forall \tau \in \tau_h\},
\end{equation}
and, for the test functions and the adjoint state, the subspace:
\begin{equation}
H_h=\{\theta \in K_h\; :\; \theta_{|_{\Gamma_T}}=0\}.
\end{equation}
\item $\mathbf{X}_h$ ($\mathbb{P}_1$ FEM space) for the concentration $\mathbf{u}$ of the species involved in eutrophication process:
\begin{equation}
\mathbf{X}_h=\{\mathbf{u} \in [\mathcal{C}(\overline{\Omega})]^5 \; :\; \mathbf{u}_{|_{\tau}} \in [\mathbb{P}_1(\tau)]^5,\; \forall \tau \in \tau_h\},
\end{equation}
and, for the test functions and the adjoint state, the subspace:
\begin{equation}
\mathbf{Z}_h=\{\mathbf{u} \in \mathbf{X}_h \; :\; \mathbf{u}_{|_{\Gamma_T}}=\mathbf{0}\}.
\end{equation}
\end{itemize}

With respect to the computational treatment of the problem, 
we have used the open code FreeFem++ \cite{HECHT1} for the space-time discretizations of the problem. 
We have also employed a penalty method (cf. Section 4.3 of \cite{girault1}) for computing the solution of the Stokes 
problems that appear after discretization. 
Finally, in order to reduce the CPU time necessary for computing the solution of the state systems, we have applied an explicit scheme 
(evaluations in previous time step) for the nonlinearities and the coupled terms of the discretized problem. 

So, we consider the following space-time discretization for the optimal control problem $(\mathcal{P})$ where,
for the sake of simplicity, we will use the same notations for the discrete problem as in the case of the continuous one:
\begin{enumerate}
\item \textit{Coupling of temperature/species in collectors and injectors}: 

We denote by $\theta^n\in K_h$ 
and $\mathbf{u}^n\in \mathbf{X}_h$, respectively, the water temperature and the species concentration at time step $n=0,\ldots,N$. 
Then, we consider the following approximation for functions $\gamma_{\theta}^k$, $k=1,\ldots,N_{CT}$, defined in (\ref{eq:gammatheta}),
(analogously for functions $\gamma_{u^i}^k$, $k=1,\ldots,N_{CT}$, $i=1,\ldots,5$, defined in (\ref{eq:gammaeutro})):
\begin{eqnarray*} \hspace{-.8cm}
\gamma_{\theta}^k(t)=\frac{1}{\mu(C^k)} \left[ \chi_{(-\infty,t_0)} \int_{C^k} \theta^0 d \gamma +
\sum_{n=1}^{N} \chi_{[t_{n-1},t_n)} \int_{C^k} \theta^{n-1} d \gamma + \chi_{[t_N,\infty)} \int_{C^k} \theta^N d \gamma \right]
\end{eqnarray*}
Moreover, if we assume the value $\epsilon=\frac{\Delta t}{2}$ in the definition (\ref{eq:rhoeps}) of function $\rho_{\epsilon}$ we have 
that the support of $\rho_{\Delta t/2}(t^n-\frac{\Delta t}{2}-s)$ is contained in $(t^n-\Delta t, t^n)=(t^{n-1},t^n)$, for all $n=1,\ldots, N$, and then:
\begin{equation} \nonumber
\begin{array}{rcl}
\displaystyle 
\phi_{\theta}^n(\mathbf{x})&=& \displaystyle \sum_{k=1}^{N_{CT}} \varphi^k(\mathbf{x}) 
\int_{-T}^T \rho_{\Delta t/2}(t^n-\frac{\Delta t}{2} -s) \, \gamma_{\theta}^k(s) \, ds  \\
&=&\displaystyle \sum_{k=1}^{N_{CT}} \varphi^k(\mathbf{x}) \int_{t_{n-1}}^{t_n}  \rho_{\Delta t/2}(t^n-\frac{\Delta t}{2} -s) \left[
\frac{1}{\mu(C^k)} \int_{C^k} \theta^{n-1} \, d \gamma\right] ds  \\
&=&\displaystyle \sum_{k=1}^{N_{CT}} \varphi^k(\mathbf{x}) \, \frac{1}{\mu(C^k)} \int_{C^k} \theta^{n-1} \, d \gamma.
\end{array}
\end{equation}
Finally, we approximate each element $\varphi^k$ by the indicator function of the injector $T^k$, $k=1,\ldots, 
N_{CT}$, and each element $\tilde{\varphi}^k$ by the indicator function of the collector $C^k$, $k=1,\ldots,N_{CT}$. 
Thus, the temperature in each injector at time step $t_n$ is the mean temperature in the corresponding collector at time step $t_{n-1}$ 
(analogously for the species of the eutrophication model). \\

\item \textit{Discretized control}:

{\color{red} We consider the following discretization of the admisible set (\ref{uad}) (we will also denote 
by ${\mathcal{U}}_{ad}$ the set of admissible discrete controls):
\begin{equation*}
\begin{array}{l}
\mathcal{U}_{ad}=\{ \mathbf{g} \in [\mathcal{C}([0,T])]^{N_{CT}}:\; \mathbf{g}(0)=\boldsymbol{0}, \\
\qquad \mathbf{g}_{|_{[t_n,t_{n+1}]}}\in [\mathbb{P}_1([t_n,t_{n+1}])]^{N_{CT}},\; \forall n=0,\ldots,N-1, \, \text{and}\\
\qquad g^{n,k}=g^k(t_n) \in [c_1,c_2], \; \forall k=1,\ldots,N_{CT},\; \forall n=1,\ldots,N\},
\end{array}
\end{equation*}
where $c_1\,,c_2>0$ technological bounds related to 
mechanical characteristics of pumps and they are chosen so that 
$\|g^k\|_{H^1(0,T)} \leq c$,  $\forall k=1,\ldots,N_{CT}$. So, if we consider 
the standard basis for the previous finite element space, we can consider 
the following discrete control:
\begin{equation} 
\mathbf{g}=(\underbrace{g^{1,1},g^{1,2},\ldots,g^{1,N_{CT}}}_{\mathbf{g}^1}, \, \ldots \, ,
\underbrace{g^{N,1},g^{N,2},\ldots,g^{N,N_{CT}}}_{\mathbf{g}^{N}}) \in \mathbb{R}^{N\times N_{CT}},
\end{equation}

}

\item \textit{Discretized cost functional:} 

{\color{red} In order to simplify the numerical resolution of the control problem, we will consider the following 
modification of the cost functional restriction to the previous admissible set:} 
\begin{equation}
{J}(\mathbf{g})=\frac{\sigma_1}{2}\sum_{n=1}^N \sum_{k=1}^{N_{CT}} (g^{n,k})^2 
+\frac{\sigma_2}{2} \sum_{n=1}^{N-1} \sum_{k=1}^{N_{CT}} (g^{n+1,k}-g^{n,k})^2,
\end{equation}
where $\sigma_1$ and $\sigma_2$ are positive weights that we will take into account in the numerical tests. \\

\item \textit{Discretized state constraints:} 

We consider the function:
\begin{equation} 
\mathbf{G} : \mathbf{g} \in {\mathcal{U}}_{ad} \longrightarrow \mathbf{G}(\mathbf{g})
= ({G}^1(\mathbf{g}),\ldots,{G}^N(\mathbf{g})) \in \mathbb{R}^{N} ,
\end{equation}
where, for each $n=1,\ldots,N$, 
\begin{equation} 
{G}^n(\mathbf{g})=\frac{1}{\mu(\Omega_C)} \int_{\Omega_C} {u}^{n+1,5}  \, d \mathbf{x},
\end{equation}
with $\mathbf{u}^{n+1} \in \mathbf{Z}_h$ the solution of the discretized eutrophication model.
Thus, we can express:
\begin{equation} 
{\mathcal{U}}=\{\mathbf{g}\in {\mathcal{U}}_{ad}\; :\; {G}^n(\mathbf{g}) \in [\lambda^m,\lambda^M], \ \forall n=1,\ldots,N\}.
\end{equation}
It is worthwhile remarking here that, due the type of time discretization considered for the material derivatives (\ref{eq:materialderivative}), 
the control $\mathbf{g}^N$ acts over the species and the temperature at time $t_{N+1}$. So, it will be necessary to compute 
one additional time step in the case of temperature and species in order to take into account this control. 
This fact can be more clearly noticed in the dependence scheme shown in Figure~\ref{dscheme}. \\

\begin{figure}[!ht]
\[
\xymatrix{  
& \mathbf{v}^0 \ar[dr] \ar[drr] \ar[d] & 
\theta^0 \ar[d] \ar[dr] \ar[dl] & 
\mathbf{u}^0 \ar[d] \\ 
\mathbf{g}^1 \ar[r] & 
\mathbf{v}^1 \ar[dr] \ar[drr] \ar[d] & 
\theta^1 \ar[d] \ar[dr] \ar[dl] & 
\mathbf{u}^1 \ar[d] \\ 
\mathbf{g}^2 \ar[r] & 
\mathbf{v}^2 \ar[dr] \ar[drr] \ar[d] & 
\theta^2 \ar[d] \ar[dr] \ar[dl] & 
\mathbf{u}^2  \ar[d] \ar[r]& 
G^1(\mathbf{g})=G^1(\mathbf{g}^1) \\ 
\mathbf{g}^3 \ar[r]
& \mathbf{v}^3 \ar@{.>}[d]& 
\theta^3 \ar@{.>}[d]& 
\mathbf{u}^3 \ar[r] \ar@{.>}[d]& 
G^2(\mathbf{g})=G^2 (\mathbf{g}^1,\mathbf{g}^2) \\
\mathbf{g}^N  \ar[r] & 
\mathbf{v}^N \ar[dr] \ar[drr]& 
\theta^N \ar[d] \ar[dr] & 
\mathbf{u}^N \ar[d] \ar[r]& 
G^{N-1}(\mathbf{g})=G^{N-1} (\mathbf{g}^1,\mathbf{g}^2,\ldots,\mathbf{g}^{N-1})\\ 
& &  \theta^{N+1} &\mathbf{u}^{N+1} \ar[r] & G^{N}(\mathbf{g})=
G^{N} (\mathbf{g}^1,\mathbf{g}^2,\ldots,\mathbf{g}^{N-1},\mathbf{g}^{N})}
\]
\caption{Dependence scheme for the discretized variables.}
\label{dscheme}
\end{figure}
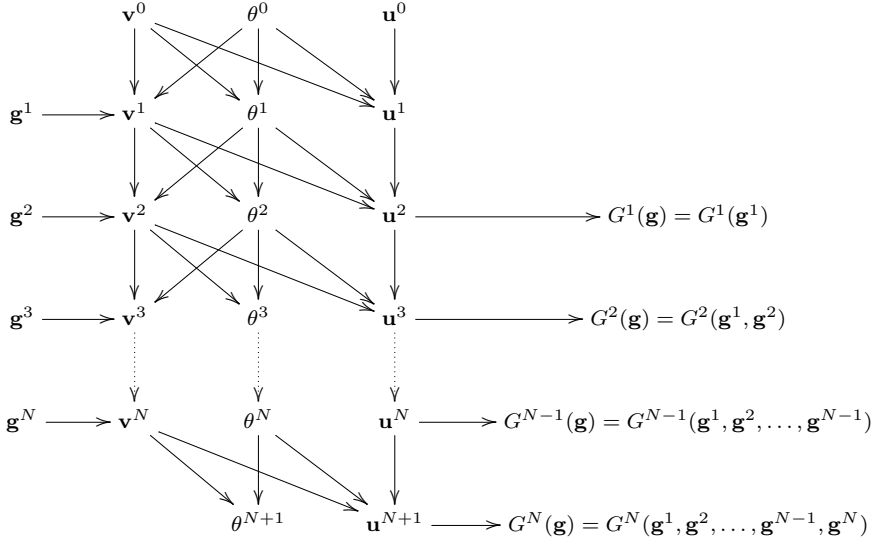

\item \textit{Water velocity and pressure:} 

Given $\mathbf{v}^0 \in \mathbf{W}_h$, 
the pair velocity/pressure $(\mathbf{v}^{n+1},p^{n+1}) \in \mathbf{V}_h \times M_h$, for each $n=0,1,\ldots,N-1$, with:
\begin{equation} 
\displaystyle
\mathbf{v}^{n+1}_{|{T^k}}=-\frac{g^{n+1,k}}{\mu(T^k)} \,\mathbf{n}, \quad
\mathbf{v}^{n+1}_{|{C^k}}=\frac{g^{n+1,k}}{\mu(C^k)} \,\mathbf{n}, \quad \forall k=1,\ldots,N_{CT},
\end{equation}
is the solution of the fully discretized system:
\begin{equation} \label{eq:system1f}
\begin{array}{r}
\displaystyle
\alpha \int_{\Omega} \mathbf{v}^{n+1} \cdot \boldsymbol{\eta} \, d \mathbf{x} +
\int_{\Omega} \beta(\mathbf{v}^n)  e(\mathbf{v}^{n+1}):e(\boldsymbol{\eta}) \, d \mathbf{x}  
-\int_{\Omega} p^{n+1} \nabla \cdot\boldsymbol{\eta}\, d \mathbf{x} \\
\displaystyle
-\int_{\Omega} \nabla \cdot \mathbf{v}^{n+1} q \, d \mathbf{x}
-\lambda \int_{\Omega} p^{n+1} q \, d \mathbf{x}
=\alpha \int_{\Omega} (\mathbf{v}^n \circ X^n_{-}) \cdot \boldsymbol{\eta}\, d \mathbf{x} \\
\displaystyle
+\int_{\Omega} \alpha_0(\theta^{n}-\theta^0) \mathbf{a}_g \cdot \boldsymbol{\eta}\, d \mathbf{x},
\quad \forall \boldsymbol{\eta} \in \mathbf{W}_h, \ \forall q \in M_h,
\end{array}
\end{equation}
where $\lambda>0$ is the penalty parameter and 
$\beta(\mathbf{v}^n)=2 \nu +2\nu_{tur}[e(\mathbf{v}^n):
e(\mathbf{v}^n)]^{1/2}$. \\

\item \textit{Water temperature:} 
Given $\theta^0 \in K_h$, the temperature $\theta^{n+1} \in K_h$, for each $n=0,\ldots,N$, with
\begin{equation}  
\theta^{n+1}_{|_{T^k}}=\frac{1}{\mu(C^k)} \int_{C^k} \theta^{n} \, d \gamma,\quad \forall k=1,\ldots, N_{CT},
\end{equation}
is the solution of the fully discretized system:
\begin{equation} \label{eq:system2f}
\begin{array}{r}
\displaystyle
\alpha \int_{\Omega} \theta^{n+1} \eta \, d \mathbf{x} + 
K \int_{\Omega} \nabla \theta^{n+1} \cdot \nabla \eta \, d \mathbf{x} +
b_1^N \int_{\Gamma_N} \theta^{n+1} \eta \, d \gamma \\ 
\displaystyle
+ b_1^S \int_{\Gamma_S} \theta^{n+1} \eta \, d \gamma
=\alpha \int_{\Omega} (\theta^n \circ X^n_{-}) \eta \, d\mathbf{x} +
b_1^N\int_{\Gamma_N} \theta_N^{n+1} \eta \, d \gamma \\ 
\displaystyle
+b_1^S \int_{\Gamma_S} \theta_S^{n+1} \eta\, d \gamma
+b_2^S \int_{\Gamma_S} ({T_r^{n}}^4 -|\theta^{n}|^3 \theta^{n}) \eta \, d \gamma, 
\quad \forall \eta \in H_h.
\end{array}
\end{equation}

\item \textit{Eutrophication species concentration:} 

Given $\mathbf{u}^0 \in 
\mathbf{X}_h$, the species concentration $\mathbf{u}^{n+1} \in \mathbf{X}_h$, for each $n=0,\ldots,N$, with:
\begin{equation}
\mathbf{u}^{n+1}_{|_{T^k}}= \frac{1}{\mu(C^k)} \int_{C^k} \mathbf{u}^n \,d\gamma, \quad \forall k=1,\ldots,N_{CT},
\end{equation}
is the solution of the fully discretized system:
\begin{equation} \label{eq:system3f}
\begin{array}{r}
\displaystyle 
\alpha \int_{\Omega} \mathbf{u}^{n+1} \cdot \boldsymbol{\eta} \, d \mathbf{x}
+ \int_{\Omega} \Lambda_{\mu} \nabla \mathbf{u}^{n+1} : \nabla \boldsymbol{\eta} \, d \mathbf{x} 
+\int_{\Omega} \mathbf{A}^n(\theta^n,\mathbf{u}^n) \mathbf{u}^{n+1} \cdot \boldsymbol{\eta} \, d \mathbf{x} \\ 
\displaystyle
=\alpha \int_{\Omega} \left(\mathbf{u}^n \circ X^n_- \right) \cdot 
\boldsymbol{\eta}\, d \mathbf{x} , \quad \forall \boldsymbol{\eta} \in \mathbf{Z}_h,
\end{array}
\end{equation}
where $\mathbf{A}^n(\theta^n,\mathbf{u}^n)\in \mathbb{R}^{5 \times 5}$ is the following matrix:
\begin{equation} 
\left[\begin{array}{ccccc}
0 & \displaystyle \frac{C_{nc} L^n(\theta^n) u^{n,1}}{K_N+|u^{n,1}|} - C_{nc} K_r & 0 & 0 & -C_{nc} K_{rd} D(\theta^n) \\ 
0 & \displaystyle K_r - \frac{L^n(\theta^n)u^{n,1}}{K_N+|u^{n,1}|}+K_{mf} & \displaystyle \frac{K_z u^{n,2}}{K_F+|u^{n,2}|} & 0 & 0 \\ 
0 & 0 & \displaystyle -\frac{C_{fz}K_z u^{n,2}}{K_F+|u^{n,2}|}+K_{mz} & 0 & 0 \\
0 & -K_{mf} & -K_{mz} & 0 & K_{rd}D(\theta^n) \\
0 & \displaystyle C_{oc}K_r- \frac{ C_{oc}  L(\theta^n)u^{n,1}}{K_N+|u^{n,1}|} & 0 & 0 & C_{oc}K_{rd}D(\theta^n) 
\end{array}\right]. \nonumber
\end{equation} 

\end{enumerate}

\subsection{Numerical resolution of the optimization problem}

Once developed above space-time discretization, as introduced in previous section, 
we obtain the following discrete optimization problem:
\begin{equation} \label{eq:discretecontrol}
({\mathcal{P}}) \qquad \min\{J(\mathbf{g}):\; \mathbf{g} \in \mathcal{U}\} \nonumber 
\end{equation}
In order to solve this nonlinear optimization problem, we will use the interior point algorithm IPOPT \cite{Biegler1} interfaced with 
the FreeFem++ code that we have developed. One of the requirements for using previous algorithm is the knowledge of functions that compute 
the gradient of the cost functional and the Jacobian matrix of the constraints. 

In the case of the cost functional, we have that its differential 
$\delta_{\mathbf{g}}J(\mathbf{g}) \in \mathcal{L}(\mathbb{R}^{N \times N_{CT}},\mathbb{R})$ is such that, 
for any $\delta \mathbf{g} = (\delta \mathbf{g}^1,\ldots, \delta \mathbf{g}^N) \in \mathbb{R}^{N \times N_{CT}}$:
\begin{equation} \label{eq:costediscreto}
\begin{array}{l}
\displaystyle
\delta_{\mathbf{g}}J(\mathbf{g}) (\delta \mathbf{g})= \sigma_1 \sum_{n=1}^N \sum_{k=1}^{N_{CT}} g^{n,k} \delta g^{n,k} \\
\displaystyle
\quad +\sigma_2 \sum_{n=1}^{N-1} \sum_{k=1}^{N_{CT}} (g^{n+1,n}-g^{n,k}) (\delta g^{n+1,n}-\delta g^{n,k}).
\end{array}
\end{equation}
Therefore, $[\nabla_{\mathbf{g}} J(\mathbf{g})]_i=
\delta_{\mathbf{g}} J(\mathbf{g}) (\mathbf{e}_i)$, where $\mathbf{e}_i$, $i=1,\ldots,N\times N_{CT}$, is the 
$i$-th vector of the canonical basis in $\mathbb{R}^{N \times N_{CT}}$. 

In the case of the Jacobian matrix of the constraints, we know that the differential associated to the application 
$\mathbf{G}:\mathcal{U}_{ad} \subset \mathbb{R}^{N \times N_{CT}} \rightarrow \mathbb{R}^N$ is such that 
$\delta_{\mathbf{g}} \mathbf{G}(\mathbf{g}) \in \mathcal{L}(\mathbb{R}^{N \times N_{CT}},\mathbb{R}^{N})$. 
So, given any element $\delta \mathbf{g} \in \mathbb{R}^{N \times N_{CT}}$, we have that
$\delta_{\mathbf{g}} \mathbf{G}(\mathbf{g}) (\delta \mathbf{g}) \in \mathbb{R}^N$, and the Jacobian matrix 
$J_{\mathbf{g}} \mathbf{G}(\mathbf{g}) \in  \mathcal{M}_{N \times (N_{CT} \times N)} $ is such 
that $[J_{\mathbf{g}} \mathbf{G}(\mathbf{g})]_{j,i}=\langle \delta_{\mathbf{g}} \mathbf{G}(\mathbf{g})(\mathbf{e}_i), \widetilde{\mathbf{e}}_j \rangle$, 
where $\widetilde{\mathbf{e}}_j$, $j=1,\ldots,N$, is the $j$-th vector of the canonical basis in $\mathbb{R}^{N}$. 
As above commented, for computing previous matrix we can use either the linearized state systems or the adjoint state systems. 
The choice of one method or another depends on the relation between the dimension of the space of controls 
($N \times N_{CT}$) and the dimension of the space where the application $\mathbf{G}$ takes values ($N$).
\begin{itemize}
\item When using the linearized systems, we would have to solve $N_{CT} \times N$ times these systems 
(in this case, we would compute the Jacobian matrix column by column):
\begin{equation} \nonumber
J_{\mathbf{g}} \mathbf{G}(\mathbf{g})=\left(\begin{array}{c|c|c|c|c}
\delta_{\mathbf{g}} \mathbf{G}(\mathbf{g})(\mathbf{e}_1) & \delta_{\mathbf{g}} \mathbf{G}(\mathbf{g})(\mathbf{e}_2) &\cdots & 
\delta_{\mathbf{g}} \mathbf{G}(\mathbf{g})(\mathbf{e}_{N_{CT}-1}) & \delta_{\mathbf{g}} \mathbf{G}(\mathbf{g})(\mathbf{e}_{N_{CT}}) 
\end{array} \right),
\end{equation}
where $\delta_{\mathbf{g}} \mathbf{G}(\mathbf{g})(\mathbf{e}_k) \in \mathcal{M}_{N \times 1}(\mathbb{R})$, for $k=1,\ldots,N_{CT}$. 
\item When employing the adjoint state systems, we would have to solve $N $ times these systems 
(now, we would compute the Jacobian row by row): 
\begin{equation} \nonumber
J_{\mathbf{g}} \mathbf{G}(\mathbf{g})(\delta \mathbf{g})=\left(\begin{array}{c}
\nabla_{\mathbf{g}} G^1(\mathbf{g}) \\ \hline 
\nabla_{\mathbf{g}} G^2(\mathbf{g}) \\ \hline
\vdots \\ \hline
\nabla_{\mathbf{g}} G^{N-1}(\mathbf{g})\\ \hline 
\nabla_{\mathbf{g}} G^N(\mathbf{g})
\end{array} \right),
\end{equation}
where $\nabla_{\mathbf{g}} G^i (\mathbf{g}) \in \mathbb{R}^{N_{CT} }$ is such that 
$[\nabla_{\mathbf{g}} G^j (\mathbf{g})]_i=\delta_{\mathbf{g}} G^j(\mathbf{g}) (\mathbf{e}_i)$, 
$j=1,\ldots,N$, $i=1,\ldots,N\times N_{CT}$. 
\end{itemize}
In our case, $N_{CT}>1$. So, it is more advantageous to employ the adjoint state systems and computing the Jacobian matrix 
row by row ($j=1,\ldots,N$). 
However, in order to obtain a computational expression for the Jacobian matrix using the adjoint systems it will be 
necessary deriving first a theoretical expression using the linearized systems and then applying a transposition procedure.

\begin{lemma}[Computing the Jacobian matrix using linearized systems]
Within the framework introduced in this Section, we have the following expression for the Jacobian matrix of the 
constraints using the linearized equations: Given an element $\delta \mathbf{g} \in \mathbb{R}^{N \times N_{CT}}$, then
\begin{equation} \nonumber
\delta_{\mathbf{g}} \mathbf{G}(\mathbf{g}) (\delta \mathbf{g})= 
\left(\begin{array}{c}
\displaystyle \frac{1}{\mu(\Omega_C)} \int_{\Omega_C} \delta u^{2,5} \, d\mathbf{x}  \\
\vdots \\
\displaystyle \frac{1}{\mu(\Omega_C)} \int_{\Omega_C} \delta u^{N+1,5} \, d\mathbf{x} \vspace{0.1cm} \\
\end{array}\right),
\end{equation} 
where $\{(\delta \mathbf{v}^n,\delta p^n)\}_{n=0}^{N} \subset \mathbf{V}_h \times M_h$, 
$\{\delta \theta^n\}_{n=0}^{N+1} \subset K_h$ and $\{\delta \mathbf{u}^n\}_{n=0}^{N+1} \subset \mathbf{X}_h$ 
are, respectively, the solutions of the linearized hydrodynamic model, the linearized thermic model and the linearized eutrophication model,
defined as:
\begin{itemize}
\item \textit{Linearized system for water velocity and pressure:} Given $(\delta \mathbf{v}^0,\delta p^0)=(\mathbf{0},0)$, 
for each $n=0,\ldots,N-1$, $(\delta \mathbf{v}^n,\delta p^n) \in \mathbf{V}_h \times M_h$, with
\begin{equation} 
\begin{array}{r}
\displaystyle
\delta \mathbf{v}^{n+1}_{|{T^k}}=-\frac{\delta g^{n+1,k}}{\mu(T^k)} \,\mathbf{n}, \quad
\delta \mathbf{v}^{n+1}_{|{C^k}}=\frac{\delta g^{n+1,k}}{\mu(C^k)} \,\mathbf{n}, \quad \forall k=1,\ldots,N_{CT},
\end{array}
\end{equation}
is the solution of:
\begin{equation} \label{eq:system1g}
\begin{array}{r}
\displaystyle 
\alpha \int_{\Omega} \delta \mathbf{v}^{n+1} \cdot \boldsymbol{\eta}\, d \mathbf{x}
+ \int_{\Omega} \beta(\mathbf{v}^n) e(\delta \mathbf{v}^{n+1}):e(\boldsymbol{\eta}) \, d \mathbf{x} 
- \int_{\Omega} \delta p^{n+1} \nabla \cdot \boldsymbol{\eta} \, d \mathbf{x} \\ 
\displaystyle
- \int_{\Omega} \nabla \cdot \delta \mathbf{v}^{n+1} q \, d \mathbf{x}
- \lambda \int_{\Omega} \delta p^{n+1} q \, d \mathbf{x}
= \alpha \int_{\Omega} (\delta \mathbf{v}^n \circ X^n_-) \cdot \boldsymbol{\eta}\, d \mathbf{x} \\ 
\displaystyle
- \int_{\Omega} (\nabla \mathbf{v}^n \circ X^n_-) \delta \mathbf{v}^{n} \cdot \boldsymbol{\eta}\, d \mathbf{x}
+\int_{\Omega} \alpha_0 \, \delta\theta^{n}\mathbf{a}_g \cdot \boldsymbol{\eta}\, d \mathbf{x} \\ 
\displaystyle
- \int_{\Omega} \gamma(\mathbf{v}^n) \, e(\mathbf{v}^n):e(\delta \mathbf{v}^n) \, e(\mathbf{v}^{n+1}):
e(\boldsymbol{\eta}) \, d \mathbf{x}, \quad \forall \boldsymbol{\eta} \in \mathbf{W}_h, \ \forall q \in M_h,
\end{array}
\end{equation}
where $\gamma(\mathbf{v}^n)=2\nu_{tur}[e(\mathbf{v}^n): e(\mathbf{v}^n)]^{-1/2}$.

\item \textit{Linearized system for water temperature:} Given $\delta \theta^0=0$, for 
each $n=0,\ldots,N$, $\delta \theta^{n+1} \in K_h$, with:
\begin{equation}  
\delta \theta^{n+1}_{|_{T^k}}=\frac{1}{\mu(C^k)} \int_{C^k} \delta\theta^{n} \, d \gamma, \quad \forall k=1,\ldots, N_{CT},
\end{equation}
is the solution of:
\begin{equation} \label{eq:system2g}
\begin{array}{r}
\displaystyle
\alpha \int_{\Omega} \delta \theta^{n+1} \eta \, d \mathbf{x} 
+ K \int_{\Omega} \nabla \delta \theta^{n+1} \cdot \nabla \eta \, d \mathbf{x} 
+ b_1^N \int_{\Gamma_N} \delta \theta^{n+1} \eta \, d \gamma \\ 
\displaystyle
+ b_1^S \int_{\Gamma_S} \delta \theta^{n+1} \eta \, d \gamma
= \alpha \int_{\Omega} (\delta \theta^n \circ X^n_-) \eta \, d\mathbf{x} \\ 
\displaystyle
- \int_{\Omega} (\nabla \theta^n \circ X^n_-) \cdot \delta \mathbf{v}^n \eta \, d \mathbf{x}
- 4 b_2^S\int_{\Gamma_S} |\theta^n|^3 \delta \theta^n \eta \, d \gamma, \quad \forall \eta \in H_h.
\end{array}
\end{equation}

\item \textit{Linearized system for eutrophication model:} Given $\delta \mathbf{u}^0=\mathbf{0}$, 
for each $n=0,\ldots,N$, $\delta \mathbf{u}^{n+1} \in \mathbf{X}_h$, with:
\begin{equation}
\delta \mathbf{u}^{n+1}_{|_{T^k}}= \frac{1}{\mu(C^k)} \int_{C^k} \delta \mathbf{u}^n \,d\gamma,  \quad \forall k=1,\ldots,N_{CT},
\end{equation}
is the solution of:
\begin{equation} \label{eq:system3g}
\begin{array}{r}
\displaystyle  
\alpha \int_{\Omega} \delta \mathbf{u}^{n+1} \cdot \boldsymbol{\eta} \, d \mathbf{x}
+ \int_{\Omega} \Lambda_{\mu}\nabla \delta \mathbf{u}^{n+1} : \nabla \boldsymbol{\eta} \, d \mathbf{x} \\ 
\displaystyle
+ \int_{\Omega} \mathbf{A}^n(\theta^n,\mathbf{u}^n) \delta \mathbf{u}^{n+1} \cdot \boldsymbol{\eta} \, d \mathbf{x} 
= \alpha \int_{\Omega} (\delta \mathbf{u}^n \circ X^n_-) \cdot \boldsymbol{\eta} \, d \mathbf{x} \\ 
\displaystyle
- \int_{\Omega}(\nabla \mathbf{u}^n \circ X^n_-) \delta \mathbf{v}^n \cdot \boldsymbol{\eta}\, d \mathbf{x} 
- \int_{\Omega} \delta_{\theta} \mathbf{A}^n(\theta^n,\mathbf{u}^n) (\delta \theta^n) \mathbf{u}^{n+1} \cdot \boldsymbol{\eta}\, d \mathbf{x} \\ 
\displaystyle
- \int_{\Omega} \delta_{\mathbf{u}} \mathbf{A}^n(\theta^n,\mathbf{u}^n)(\delta \mathbf{u}^n) 
\mathbf{u}^{n+1} \cdot \boldsymbol{\eta}\, d \mathbf{x},\quad \forall \boldsymbol{\eta} \in \mathbf{Z}_h.
\end{array}
\end{equation}
\end{itemize}
\end{lemma}

\begin{proof}
The proof is straightforward, where the only drawback is related to the computation of terms of the type
$\delta_{\mathbf{g}}\left( \boldsymbol{\varphi}(\mathbf{g},\mathbf{x}-\Delta t\, \mathbf{v}(\mathbf{g},\mathbf{x}))\right)(\delta g)$,
where $\boldsymbol{\varphi}(\mathbf{g},\mathbf{x})$ and $\mathbf{v}(\mathbf{g},\mathbf{x})$
are vector functions smooth enough (the scalar case would be analogous). Nevertheless, using the chain rule, we can easily obtain that:
\begin{equation} \nonumber
\begin{array}{l}
\displaystyle
\delta_{\mathbf{g}}\left( \boldsymbol{\varphi}(\mathbf{g},\mathbf{x}-\Delta t\, \mathbf{v}(\mathbf{g},\mathbf{x}))\right)(\delta g) \\
\displaystyle 
\qquad = \delta_{\mathbf{g}}  \boldsymbol{\varphi}(\mathbf{g},\mathbf{x}-\Delta t\, \mathbf{v}(\mathbf{g},\mathbf{x}))(\delta g) 
- \Delta t\, \delta_{\mathbf{x}} \boldsymbol{\varphi}(\mathbf{g},\mathbf{x}-\Delta t\, \mathbf{v}(\mathbf{g},\mathbf{x})) 
(\delta_{\mathbf{g}}\mathbf{v}(\mathbf{g},\mathbf{x}) (\delta \mathbf{g}) ) \\ 
\displaystyle
\qquad \equiv (\delta \boldsymbol{\varphi} \circ X_ -) - \Delta t\, (\nabla \boldsymbol{\varphi} \circ X_-) \delta \mathbf{v}.
\end{array}
\end{equation}
(We must note here that, in our specific formulation, we deal with the function $b:x \in \mathbb{R} \rightarrow b(x)=x \, |x|^3$, 
that is differentiable in $\mathbb{R}$, with $b'(x)=4\, |x|^3$).   \qed
\end{proof}

\begin{lemma}[Computing the Jacobian matrix using the adjoint equations] 
Within the framework introduced in this Section, we have the following expression for the Jacobian matrix of the 
constraints using the adjoint systems: 
For each row $k=1,\ldots,N$, the matrices $\{\nabla_{\mathbf{g}^n} G^k (\mathbf{g})\}_{n=1}^k \subset 
\mathcal{M}_{1 \times N_{CT}}(\mathbb{R})$ can be computed using the following expressions:
\begin{itemize}
\item If $n \in \{1,\ldots,k\}\setminus\{N\}$, 
\begin{equation} \nonumber
\begin{array}{rcl}
\displaystyle 
\delta_{\mathbf{g}^n} G^k(\mathbf{g}) (\delta \mathbf{g}^n) &=&
\displaystyle \sum_{i=1}^{N_{CT}} \frac{\delta g^{n,i}}{\mu(T^i)} \int_{T^i} \beta(\mathbf{v}^{n-1}) e(\mathbf{w}^{n-1}) 
\mathbf{n} \cdot \mathbf{n} - q^{n-1}\, d \gamma \\ 
&+& \displaystyle \sum_{i=1}^{N_{CT}} \frac{\delta g^{n,i}}{\mu(C^i)} \int_{C^i} q^{n-1}-
\beta(\mathbf{v}^{n-1})e(\mathbf{w}^{n-1}) \mathbf{n} \cdot \mathbf{n}  \, d \gamma \\ 
&+& \displaystyle \sum_{i=1}^{N_{CT}} \frac{\delta g^{n,i}}{\mu(T^i)} \int_{T^i} \gamma(\mathbf{v}^n)
e(\mathbf{v}^{n+1}): e(\mathbf{w}^{n}) e(\mathbf{v}^n)\mathbf{n} \cdot \mathbf{n} \, d \gamma \\ 
&-& \displaystyle \sum_{i=1}^{N_{CT}} \frac{\delta g^{n,i}}{\mu(C^i)} \int_{C^i} \gamma(\mathbf{v}^n)
e(\mathbf{v}^{n+1}): e(\mathbf{w}^{n}) e(\mathbf{v}^n)\mathbf{n} \cdot \mathbf{n} \, d \gamma.
\end{array}
\end{equation}
\item If $n=N$,
\begin{equation} \nonumber
\begin{array}{rcl}
\displaystyle 
\delta_{\mathbf{g}^n} G^k(\mathbf{g}) (\delta \mathbf{g}^n) &=& \displaystyle \sum_{i=1}^{N_{CT}} \frac{\delta g^{n,i}}{\mu(T^i)}
\int_{T^i} \beta(\mathbf{v}^{n-1}) e(\mathbf{w}^{n-1}) \mathbf{n} \cdot \mathbf{n} - q^{n-1}\, d \gamma \\ 
&+& \displaystyle \sum_{i=1}^{N_{CT}} \frac{\delta g^{n,i}}{\mu(C^i)} \int_{C^i} q^{n-1}-
\beta(\mathbf{v}^{n-1})e(\mathbf{w}^{n-1}) \mathbf{n} \cdot \mathbf{n}  \, d \gamma,
\end{array}
\end{equation} 
\end{itemize}
where if we introduce, for each row $k=1,\ldots,N$, the following vector (defined from the usual Kronecker delta $\delta_{ij}$
and the indicator function of subset $\Omega_C$): 
\begin{equation} \nonumber
\mathbf{H}^{n+1}_k=\left( 0 , 0 , 0 , 0 ,  \frac{1}{\mu(\Omega_C)} \, \chi_{\Omega_C} \, \delta_{kn} \right) 
\in \mathcal{M}_{1\times 5}(\mathbb{R}), \quad n=0,\ldots,N,
\end{equation} 
then the adjoint states associated to the eutrophication system $\{\mathbf{z}^n\}_{n=0}^{N+1} \subset \mathbf{Z}_h$, 
to the hydrodynamic system $\{(\mathbf{w}^n,q^n)\}_{n=0}^N \subset \mathbf{W}_h \times M_h$, and to the 
temperature system $\{\xi^n\}_{n=0}^{N+1}\subset H_h$ are, respectively, the solution of the following systems: 

\begin{itemize}
\item \textit{Adjoint system for eutrophication model}:
\begin{itemize}
\item For $n=N+1$, $\mathbf{z}^n=\mathbf{0}$.
\item For $n=N$, $\mathbf{z}^n \in \mathbf{Z}_h$ is such that:
\begin{equation} \label{eq:system3i-1} \hspace{-1.1cm}
\left\{\begin{array}{l}
\displaystyle
\alpha \mathbf{z}^n - \nabla \cdot (\Lambda_{\mu} \nabla \mathbf{z}^n)+ \mathbf{A}^n(\theta^n,\mathbf{u}^n)^T \mathbf{z}^n =
\alpha (\mathbf{z}^{n+1} \circ  \mathbf{X}_+^{n+1}) + \mathbf{H}^{n+1}_k \ \mbox{in} \; \Omega, \\ 
\displaystyle
\mathbf{z}^n=\mathbf{0} \quad \mbox{on} \; \Gamma_T,  \\ 
\displaystyle
\Lambda_{\mu} \nabla \mathbf{z}^n\mathbf{n}=\mathbf{0} \quad \mbox{on} \; \partial \Omega \setminus(\Gamma_T \cup \Gamma_C) \\ 
\displaystyle
\Lambda_{\mu} \nabla \mathbf{z}^n\mathbf{n} = -\frac{1}{\mu(C^k)} \int_{T^k} \Lambda_{\mu} \nabla 
\mathbf{z}^{n+1} \mathbf{n} \, d \gamma' \quad \mbox{on} \; C^k, \ k=1,\ldots,N_{CT},
\end{array} \right.
\end{equation}
where $\mathbf{X}_+^{n+1}(\mathbf{x})=\mathbf{x}+\Delta t \, \mathbf{v}^{n+1}$.
\item For $n=N-1,\ldots,0$, $\mathbf{z}^n \in \mathbf{Z}_h$ is such that:
\begin{equation} \label{eq:system3i-2} \hspace{-.6cm}
\left\{\begin{array}{l}
\displaystyle
\alpha \mathbf{z}^n - \nabla \cdot (\Lambda_{\mu} \nabla \mathbf{z}^n) + \mathbf{A}^n(\theta^n,\mathbf{u}^n)^T \mathbf{z}^n =
\alpha (\mathbf{z}^{n+1} \circ  \mathbf{X}_+^{n+1}) + \mathbf{H}^{n+1}_k \\ 
\quad \displaystyle
-\sum_{l=1}^{5} [\nabla_{\mathbf{u}} A_l^{n+1}(\theta^{n+1},\mathbf{u}^{n+1})]^T \mathbf{u}^{n+2} z^{n+1,l} \quad \mbox{in} \;  \Omega, \\ 
\displaystyle
\mathbf{z}^n=\mathbf{0} \quad \mbox{on} \; \Gamma_T, \\ 
\displaystyle
\Lambda_{\mu} \nabla \mathbf{z}^n\mathbf{n}=\mathbf{0} \quad \mbox{on} \; \partial \Omega \setminus(\Gamma_T \cup \Gamma_C), \\ 
\displaystyle
\Lambda_{\mu} \nabla \mathbf{z}^n\mathbf{n} = -\frac{1}{\mu(C^k)} \int_{T^k} \Lambda_{\mu} \nabla 
\mathbf{z}^{n+1} \mathbf{n} \, d \gamma' \quad \mbox{on} \; C^k, \ k=1,\ldots,N_{CT},
\end{array} \right.
\end{equation}
\end{itemize}

\item \textit{Adjoint system for water temperature:} 
\begin{itemize}
\item For $n=N+1$, $\xi^n=0$.
\item For $n=N$, $\xi^n \in H_h$ is such that:
\begin{equation} \label{eq:system2i-1} \hspace{-.3cm}
\left\{\begin{array}{l}
\displaystyle 
\alpha \xi^n -\nabla \cdot (K \nabla \xi^n) = \alpha (\xi^{n+1}\circ X_+^{n+1}) \quad \mbox{in} \; \Omega, \\
\displaystyle
\xi^n=0 \quad \mbox{on} \; \Gamma_T, \\
\displaystyle
K \nabla \xi^n \cdot \mathbf{n}= -b_1^S \xi^n-4b_2^S|\theta^{n+1}|^3 \xi^{n+1} \quad \mbox{on} \; \Gamma_S, \\
\displaystyle
K \nabla \xi^n \cdot \mathbf{n}= -b_1^N \xi^n \quad \mbox{on} \; \Gamma_N, \\
\displaystyle
K \nabla \xi^n \cdot \mathbf{n}=-\frac{K}{\mu(C^k)} \int_{T^k} \nabla \xi^{n+1}\cdot \mathbf{n}\, d \gamma' \quad \mbox{on} \; C^k,\ k=1,\ldots,N_{CT}.
\end{array}\right.
\end{equation}
\item For $n=N-1,\ldots,0$, $\xi^n \in H_h$ is such that:
\begin{equation} \label{eq:system2i-2} \hspace{-.3cm}
\left\{\begin{array}{l}
\displaystyle 
\alpha \xi^n -\nabla \cdot (K \nabla \xi^n) = \alpha (\xi^{n+1}\circ X_+^{n+1}) \\
\displaystyle
\quad -\frac{d}{d\theta} \mathbf{A}^{n+1} (\theta^{n+1},\mathbf{u}^{n+1}) \mathbf{u}^{n+2} \cdot \mathbf{z}^{n+1}
+\alpha_0 \mathbf{a}_{g} \cdot \mathbf{w}^{n+1} \quad \mbox{in} \; \Omega, \\
\displaystyle
\xi^n=0 \quad \mbox{on} \; \Gamma_T, \\
\displaystyle
K \nabla \xi^n \cdot \mathbf{n}= -b_1^S \xi^n-4b_2^S|\theta^{n+1}|^3 \xi^{n+1} \quad \mbox{on} \; \Gamma_S, \\
\displaystyle
K \nabla \xi^n \cdot \mathbf{n}= -b_1^N \xi^n \quad \mbox{on} \; \Gamma_N, \\
\displaystyle
K \nabla \xi^n \cdot \mathbf{n}=-\frac{K}{\mu(C^k)} \int_{T^k} \nabla \xi^{n+1}\cdot \mathbf{n}\, d \gamma' \quad \mbox{on} \; C^k,\ k=1,\ldots,N_{CT}.
\end{array} \right.
\end{equation}
\end{itemize}

\item \textit{Adjoint system for water velocity and pressure:}
\begin{itemize}
\item For $n=N$, $(\mathbf{w}^n,q^0)=(\mathbf{0},0)$.
\item For $n=N-1$, $(\mathbf{w}^n,q^n)\in \mathbf{W}_h\times M_h$ 
is such that:
\begin{equation}\label{eq:system1i-1}
\left\{\begin{array}{l}
\displaystyle 
\alpha \mathbf{w}^n - {\rm div} (\beta(\mathbf{v}^n)e(\mathbf{w}^n))+\nabla q^n=\alpha(\mathbf{w}^{n+1} \circ X_+^{n+1}) \\
\displaystyle
\quad - (\nabla \mathbf{v}^{n+1} \circ X_-^{n+1})^T\mathbf{w}^{n+1} -(\nabla \mathbf{u}^{n+1} \circ X_-^{n+1})^T \mathbf{z}^{n+1} \\
\displaystyle
\quad -(\nabla \theta^{n+1} \circ X_-^{n+1})^T \xi^{n+1} \quad \mbox{in} \; \Omega, \\ 
\displaystyle
\nabla \cdot \mathbf{w}^n+\lambda q^n=0 \quad \mbox{in} \; \Omega, \\ 
\displaystyle
\mathbf{w}^n=0 \quad \mbox{on} \; \partial \Omega.
\end{array} \right.
\end{equation}
\item For $n=N-2,\ldots,0$, $(\mathbf{w}^n,q^n)\in \mathbf{W}_h\times M_h$ is such that:
\begin{equation}\label{eq:system1i-2}
\left\{\begin{array}{l}
\displaystyle 
\alpha \mathbf{w}^n - {\rm div} (\beta(\mathbf{v}^n)e(\mathbf{w}^n))+\nabla q^n \\
\displaystyle
\quad =\alpha(\mathbf{w}^{n+1} \circ X_+^{n+1}) - (\nabla \mathbf{v}^{n+1} \circ X_-^{n+1})^T\mathbf{w}^{n+1}  \\
\displaystyle
\quad -(\nabla \mathbf{u}^{n+1} \circ X_-^{n+1})^T \mathbf{z}^{n+1} - (\nabla \theta^{n+1} \circ X_-^{n+1})^T \xi^{n+1} \\
\displaystyle
\quad +\nabla \cdot (\gamma(\mathbf{v}^{n+1}) e(\mathbf{v}^{n+2}):
e(\mathbf{w}^{n+1}) e(\mathbf{v}^{n+1})) \quad \mbox{in} \; \Omega, \\ 
\displaystyle
\nabla \cdot \mathbf{w}^n+\lambda q^n=0 \quad \mbox{in} \; \Omega, \\ 
\displaystyle
\mathbf{w}^n=0 \quad \mbox{on} \; \partial \Omega.
\end{array} \right.
\end{equation}
\end{itemize}
\end{itemize}
\end{lemma}

\begin{remark} 
In order to simplify the proof of above Lemma, we have established the adjoint systems 
(\ref{eq:system1i-1})-(\ref{eq:system1i-2}), (\ref{eq:system2i-1})-(\ref{eq:system2i-2}) and (\ref{eq:system3i-1})-(\ref{eq:system3i-2}) 
in a strong formulation (contrary to the case of the linearized systems (\ref{eq:system1g}), (\ref{eq:system2g}) and (\ref{eq:system3g}), 
where we have proposed a variational formulation). It is also clear that these adjoint systems easily admits a variational formulation, 
but we have chosen to formulate them in a strong form for a better understanding of the demonstration.  \qed
\end{remark}

\begin{proof} 
Let us consider as a test functions in the linearized systems (\ref{eq:system1g}), (\ref{eq:system2g}) and (\ref{eq:system3g}), respectively, 
the $n$-th component of the sequences $\{(\mathbf{w}^n,q^n)\}_{n=0}^N \subset \mathbf{W}_h\times M_h$, 
$\{\xi^n\}_{n=0}^{N+1} \subset H_h$ and $\{\mathbf{z}^n\}_{n=0}^N \subset \mathbf{Z}_h$, 
such that $\mathbf{w}^N=\mathbf{0}$, $q^N=0$, $\xi^{N+1}=0$ and $\mathbf{z}^{N+1}=\mathbf{0}$, and let us sum in $n$ from $0$ to $N$. 
Then, after some straightforward computations, taking into account the final conditions for the adjoint systems and the initial 
conditions for the linearized ones, we have:
\begin{itemize}
\item For eutrophication model:
\begin{equation} \label{eq:system3i} \hspace{-.7cm}
\begin{array}{l}
\displaystyle
\sum_{n=0}^{N} \bigg[ \alpha \int_{\Omega} \mathbf{z}^n \cdot \delta \mathbf{u}^{n+1} \, d \mathbf{x}
- \int_{\Omega} \nabla \cdot (\Lambda_{\mu} \nabla \mathbf{z}^{n}) \cdot \delta \mathbf{u}^{n+1} \, d \mathbf{x} \\ 
\displaystyle 
\quad +\int_{\Omega} \mathbf{A}^n(\theta^n,\mathbf{u}^n)^T \mathbf{z}^{n} \cdot \delta \mathbf{u}^{n+1} \, d \mathbf{x} \bigg] 
= \sum_{n=0}^{N} \bigg[ \alpha \int_{\Omega} (\mathbf{z}^{n+1} \circ X_+^{n+1}) \cdot \delta \mathbf{u}^{n+1} \, d \mathbf{x} \\ 
\displaystyle 
\quad -\int_{\Omega} (\nabla \mathbf{u}^{n+1} \circ X^{n+1}_-)^T \mathbf{z}^{n+1} \cdot \delta \mathbf{v}^{n+1} \, d \mathbf{x} \\ 
\displaystyle 
\quad -\sum_{k=1}^{N_{CT}} \int_{C^k}\left( \Lambda_{\mu} \nabla \mathbf{z}^n \mathbf{n}+\frac{1}{\mu(C^k)} \int_{T^k} \Lambda_{\mu}
\nabla \mathbf{z}^{n+1} \mathbf{n} \, d \gamma' \right) \cdot \delta \mathbf{u}^{n+1} \, d \gamma \bigg] \\ 
\displaystyle 
\quad -\sum_{n=0}^{N-1} \bigg[ \int_{\Omega} \left(\frac{d}{d \theta} \mathbf{A}^{n+1}(\theta^{n+1},\mathbf{u}^{n+1}) \mathbf{u}^{n+2} \cdot 
\mathbf{z}^{n+1}\right) \delta \theta^{n+1}\, d \mathbf{x}  \\ 
\displaystyle 
\quad +\int_{\Omega} \left( \sum_{l=1}^5 [\nabla_{\mathbf{u}} A_l^{n+1}(\theta^{n+1},\mathbf{u}^{n+1})]^T
\mathbf{u}^{n+2} z^{n+1,l}\right) \cdot \delta \mathbf{u}^{n+1} \, d \mathbf{x} \bigg],
\end{array}
\end{equation}
with $\mathbf{X}_+^{n+1}(\mathbf{x})=\mathbf{x}+\Delta t \, \mathbf{v}^{n+1}$, and where we are assuming 
$\delta \mathbf{v}^{N+1}=\mathbf{0}$ in order to simplify the notation.

\item For water temperature:
\begin{equation} \label{eq:system2i} 
\begin{array}{l}
\displaystyle 
\sum_{n=0}^{N} \bigg[\alpha \int_{\Omega} \xi^n \delta \theta^{n+1} \, d \mathbf{x} 
-\int_{\Omega} \nabla \cdot (K \nabla \xi^n) \delta \theta^{n+1} \, d \mathbf{x} \\ 
\displaystyle 
\quad +\int_{\Gamma_N} \left(b_1^N \xi^n+K \nabla \xi^n \cdot \mathbf{n} \right)\delta \theta^{n+1} \, d \gamma \\ 
\displaystyle 
\quad +\int_{\Gamma_S} \left(b_1^S \xi^n + 4b_2^S|\theta^{n+1}|^3 \xi^{n+1} + K \nabla \xi^n \cdot \mathbf{n}
\right) \delta \theta^{n+1} \, d \gamma \bigg] \\ 
\displaystyle 
\quad = \sum_{n=0}^N \bigg[ \alpha \int_{\Omega} (\xi^{n+1}\circ X_+^{n+1}) \delta \theta^{n+1} \, d \mathbf{x} \\ 
\displaystyle 
\quad -\int_{\Omega} (\nabla \theta^{n+1} \circ X^{n+1}_-)^T \xi^{n+1} \cdot \delta \mathbf{v}^{n+1} \, d \mathbf{x} \\ 
\displaystyle 
\quad -\sum_{k=1}^{N_{CT}} \int_{C^k} \left( K \nabla \xi^{n} \cdot \mathbf{n} + \frac{1}{\mu(C^k)} \int_{T^k} K \nabla \xi^{n+1} \cdot 
\mathbf{n} \, d \gamma' \right) \delta \theta^{n+1} \, d \gamma \bigg],
\end{array}
\end{equation}
where, for the sake of simplicity, we have also assumed $\delta \mathbf{v}^{N+1}=\mathbf{0}$.

\item For water velocity:
\begin{equation} \label{eq:system1i} \hspace{-.8cm}
\begin{array}{l}
\displaystyle
\sum_{n=0}^{N-1} \bigg[ \alpha \int_{\Omega} \mathbf{w}^n \cdot \delta \mathbf{v}^{n+1} \, d \mathbf{x}
- \int_{\Omega} \text{div}  \left( \beta(\mathbf{v}^n) e(\mathbf{w}^n) \right) \cdot \delta \mathbf{v}^{n+1} \, d \mathbf{x} \\ 
\displaystyle 
\quad - \int_{\Omega} \nabla \cdot \mathbf{w}^n \delta p^{n+1} \, d \mathbf{x}
+ \int_{\Omega} \nabla q^n \cdot \delta \mathbf{v}^{n+1} \, d \mathbf{x} 
-\lambda \int_{\Omega} q^n \delta p^{n+1} \, d \mathbf{x} \bigg] \\ 
\displaystyle 
\quad = \sum_{n=0}^N \bigg[ \int_{\Omega} \alpha_0 \mathbf{a}_{g} \mathbf{w}^{n+1} \delta \theta^{n+1} \, d \mathbf{x} \bigg] 
+\sum_{n=0}^{N-1} \bigg[ \alpha \int_{\Omega} (\mathbf{w}^{n+1} \circ X^{n+1}_+) \cdot \delta \mathbf{v}^{n+1} \, d \mathbf{x} \\ 
\displaystyle
\quad - \int_{\Omega} (\nabla \mathbf{v}^{n+1} \circ X^{n+1}_-)^T \mathbf{w}^{n+1} \cdot \delta \mathbf{v}^{n+1} \, d \mathbf{x} \bigg] \\ 
\displaystyle 
\quad +\sum_{n=0}^{N-2} \bigg[ \int_{\Omega}\text{div}  (\gamma(\mathbf{v}^{n+1}) e(\mathbf{v}^{n+2}):e(\mathbf{w}^{n+1})
e(\mathbf{v}^{n+1}) ) \cdot \delta \mathbf{v}^{n+1} \, d \mathbf{x} \bigg] \\ 
\displaystyle 
\quad +\sum_{n=0}^{N-1} \sum_{k=1}^{N_{CT}}  \delta g^{n+1,k} \bigg[ \frac{1}{\mu(T^k)}
\int_{T^k} \left( \beta(\mathbf{v}^n)e(\mathbf{w}^n) \mathbf{n} \cdot \mathbf{n} -q^n\right) \, d \gamma \\ 
\displaystyle 
\quad -\frac{1}{\mu(C^k)} \int_{C^k} \left( \beta(\mathbf{v}^n) e(\mathbf{w}^n) 
\mathbf{n} \cdot \mathbf{n} -q^n\right) \, d \gamma \bigg] \\ 
\displaystyle 
\quad +\sum_{n=0}^{N-2} \sum_{k=1}^{N_{CT}}  \delta g^{n+1,k} \bigg[ \frac{1}{\mu(T^k)} \int_{T^k} \gamma(\mathbf{v}^{n+1}) 
e(\mathbf{v}^{n+2}):e(\mathbf{w}^{n+1}) e(\mathbf{v}^{n+1}) \mathbf{n} \cdot \mathbf{n} \, d \gamma \\ 
\displaystyle 
\quad -\frac{1}{\mu(C^k)} \int_{C^k} \gamma(\mathbf{v}^{n+1}) e(\mathbf{v}^{n+2}):e(\mathbf{w}^{n+1})
e(\mathbf{v}^{n+1}) \mathbf{n} \cdot \mathbf{n} \, d \gamma \bigg],
\end{array}
\end{equation}
where we have assumed $\mathbf{w}^{N+1}=\mathbf{0}$. 
\end{itemize}

Thus, if we define $\{(\mathbf{w}^n,q^n)\}_{n=0}^N \subset \mathbf{W}_h \times M_h$, $\{\xi^n\}_{n=0}^{N+1} \subset H_h$ 
and $\{\mathbf{z}^n\}_{n=0}^{N+1} \subset \mathbf{Z}_h$, such that $\mathbf{w}^N=\mathbf{0}$, $q^N=0$, $\xi^{N+1}=0$ and 
$\mathbf{z}^{N+1}=\mathbf{0}$, as the solutions of the adjoint system (\ref{eq:system1i-1})-(\ref{eq:system1i-2}), 
(\ref{eq:system2i-1})-(\ref{eq:system2i-2}) and (\ref{eq:system3i-1})-(\ref{eq:system3i-2}), respectively, 
we obtain, after summing above expressions (\ref{eq:system3i}),  (\ref{eq:system2i}) and (\ref{eq:system1i}), that:
\begin{equation} \nonumber \hspace{-.1cm}
\begin{array}{l}
\displaystyle 
\sum_{n=0}^{N} \int_{\Omega} \mathbf{H}^{n+1}_k \cdot \delta \mathbf{u}^{n+1} \, d \mathbf{x}
= \sum_{n=0}^{N-1} \sum_{k=1}^{N_{CT}}  \delta g^{n+1,k} \bigg[ \frac{1}{\mu(T^k)} \int_{T^k} \left( \beta(\mathbf{v}^n)
e(\mathbf{w}^n) \mathbf{n} \cdot \mathbf{n} -q^n\right) \, d \gamma  \\ 
\displaystyle 
\quad -\frac{1}{\mu(C^k)} \int_{C^k} \left( \beta(\mathbf{v}^n) e(\mathbf{w}^n) \mathbf{n} 
\cdot \mathbf{n} -q^n\right) \, d \gamma \bigg] \\ 
\displaystyle 
\quad +\sum_{n=0}^{N-2} \sum_{k=1}^{N_{CT}}  \delta g^{n+1,k} \bigg[ \frac{1}{\mu(T^k)} \int_{T^k} \gamma(\mathbf{v}^{n+1}) 
e(\mathbf{v}^{n+2}):e(\mathbf{w}^{n+1}) e(\mathbf{v}^{n+1}) \mathbf{n} \cdot \mathbf{n} \, d \gamma \\ 
\displaystyle 
\quad -\frac{1}{\mu(C^k)} \int_{C^k} \gamma(\mathbf{v}^{n+1}) 
e(\mathbf{v}^{n+2}):e(\mathbf{w}^{n+1}) e(\mathbf{v}^{n+1}) \mathbf{n} \cdot \mathbf{n} \, d \gamma \bigg].
\end{array}
\end{equation}
And, finally, from the definition:
\begin{equation}
\sum_{n=0}^{N} \int_{\Omega} \mathbf{H}^{n+1}_k \cdot \delta \mathbf{u}^{n+1} \, d \mathbf{x}
= \frac{1}{\mu(\Omega_C)} \int_{\Omega_C} \delta u^{k+1,5}\, d \mathbf{x}.
\end{equation}
\qed
\end{proof}

\subsection{Numerical results}

In order to simplify the graphical representation of the computational results for the numerical tests developed in this study, 
we will present here only the case of a two dimensional domain $\Omega$. 
So, we consider a space configuration similar to that presented in Figure~\ref{figure1}, with $N_{CT} = 4$ collector/injector pairs, 
in a rectangular domain of $20 \, {\rm m} \times 16 \, {\rm m}$. 
We suppose that the diameter of each collector is $1\, {\rm m}$ and the diameter of each injector is $2 \, {\rm m}$. 
For the coefficients of the eutrophication model (\ref{eq:system3}), we have used the same values as those appearing in \cite{DRAGO200117}, 
and for the thermo-hydrodynamic system (\ref{eq:system1}), (\ref{eq:system2}) we have employed the same values as in \cite{fran8}. 
For the space discretization we have generated a regular mesh of $2989$ vertices, as shown in Figure~\ref{figure4}.

\begin{figure}[!ht]
\centering
\includegraphics[width=0.55\textwidth]{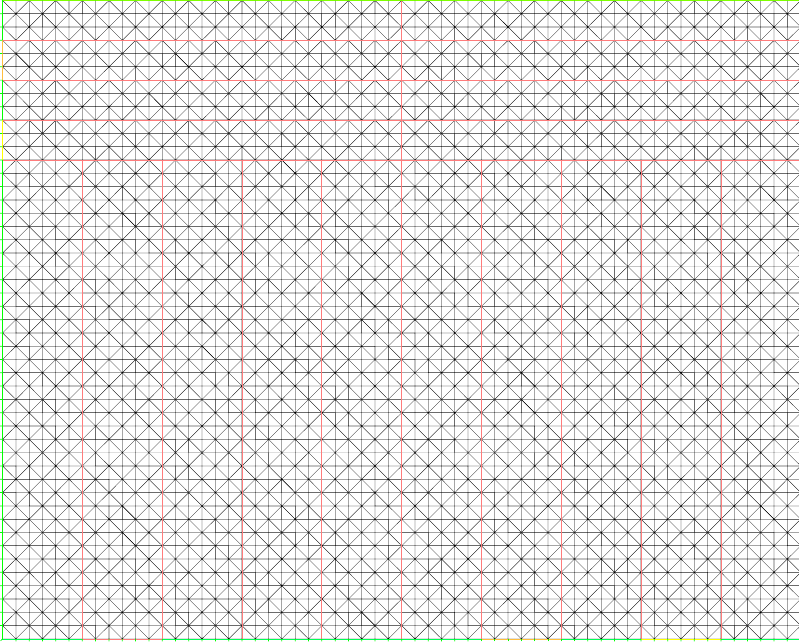}
\caption{Triangular mesh of the domain $\Omega$ for the numerical tests.}
\label{figure4}
\end{figure}

The control domain $\Omega_C$ corresponds to a $3 \, {\rm m}$ strip at the bottom of the domain, 
and all the numerical tests have been performed in a temporal horizon of 12 hours ($T = 43200 \, {\rm s}$). 
Finally, in order to simulate the effects of solar radiation for the heat equation (\ref{eq:system2}), we consider the 
standard function $T_r$ depicted in Figure~\ref{figure4bis}.

\begin{figure}[!ht]
\centering
\includegraphics[width=0.45\textwidth]{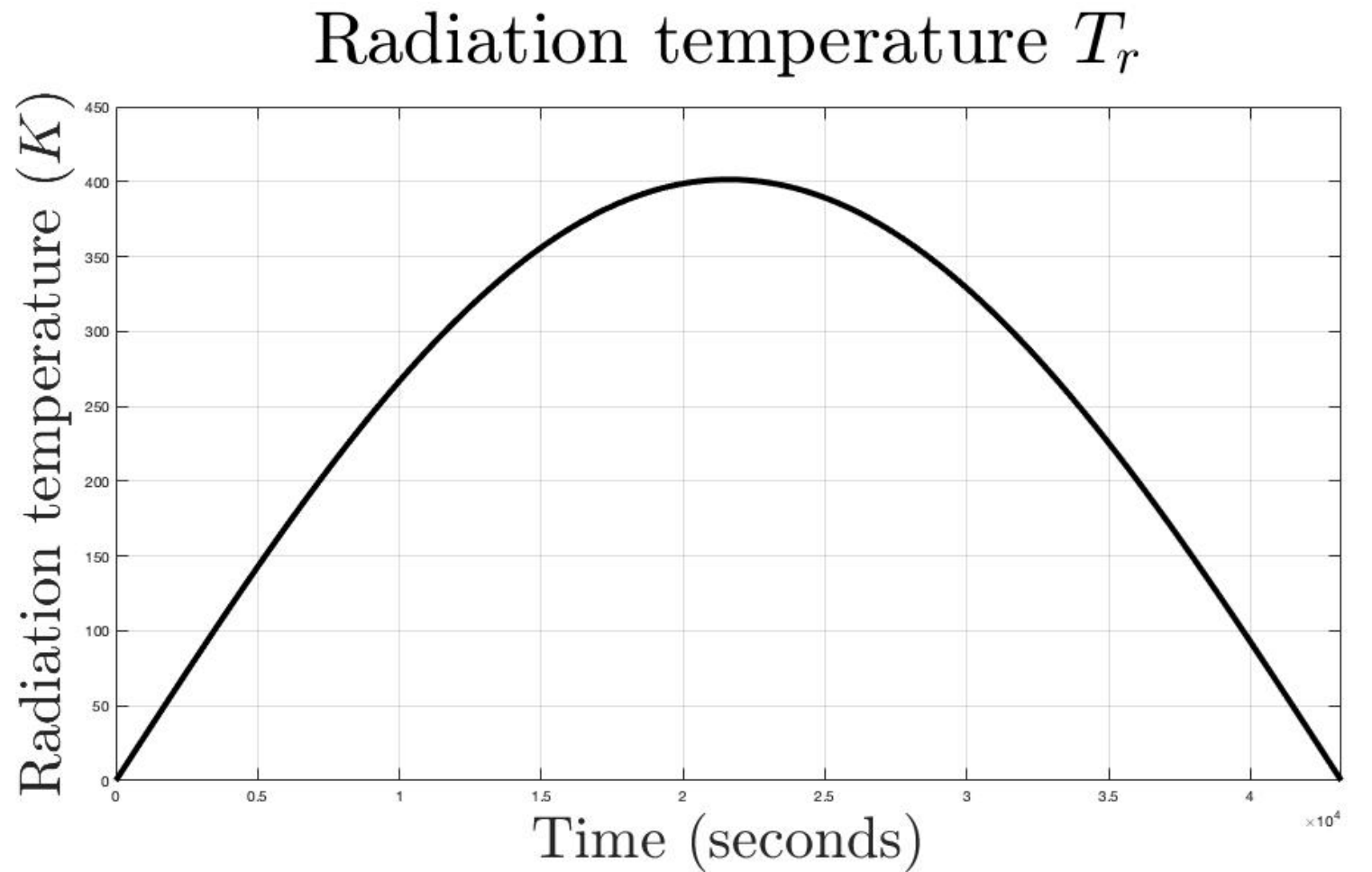}
\caption{Standard profile for radiation temperature $T_r$.}
\label{figure4bis}
\end{figure}

We must remark that our main goal in this first approximation to the numerical resolution of the problem is trying to understand if we 
can improve the management of the pumps with respect to a constant operating regime. So, given a constant reference 
control $\widetilde{\mathbf{g}}$, with $\widetilde{g}^{n,k}=C$ (constant), for $n=1,\ldots,N$, $k=1,\ldots,N_{CT}$, we will 
solve the following modification of the original optimization problem $(\mathcal{P})$:
\begin{equation} \nonumber
(\widehat{\mathcal{P}}) \qquad \min\{J(\mathbf{g}) \, :\; \mathbf{g} \in \mathcal{U}_{ad}, \;
\mathbf{G}(\mathbf{g}) \geq \mathbf{G}(\widetilde{\mathbf{g}})\}.
\end{equation}
In other words, we want to find an optimal control $\widehat{\mathbf{g}}\in \mathcal{U}_{ad}$ that supplies us with a 
higher concentration of dissolved oxygen than that obtained with the constant control $\widetilde{\mathbf{g}}$, 
and that minimizes the energy cost functional $J$. As an illustration to this behaviour, 
in Figure~\ref{figure5} we can see the evolution of the mean concentration of dissolved oxygen in the control domain $\Omega_C$ 
considering a constant reference control $\widetilde{g}^{n,k} = 1.0 \times 10^{-4} \, {\rm m}^3 \, {\rm s}^{-1}$, $\forall n=1,\ldots,N$, 
$\forall k=1,\ldots,N_{CT}$, compared to the mean concentration assuming that all the pumps are out of service 
(that is, $\widetilde{g}^{n,k}= 0.0 \, {\rm m}^3 \, {\rm s}^{-1}$, $\forall n=1,\ldots,N$, $\forall k=1,\ldots,N_{CT}$).
We observe how, if the pumps are out of service, the mean concentration of dissolved oxygen in the control domain decays gradually but, 
nevertheless, if we consider a constant flow rate (not necessarily large), this mean concentration 
of dissolved oxygen increases in a significant way. 

\begin{figure}[!ht]
\centering
\includegraphics[width=0.65\textwidth]{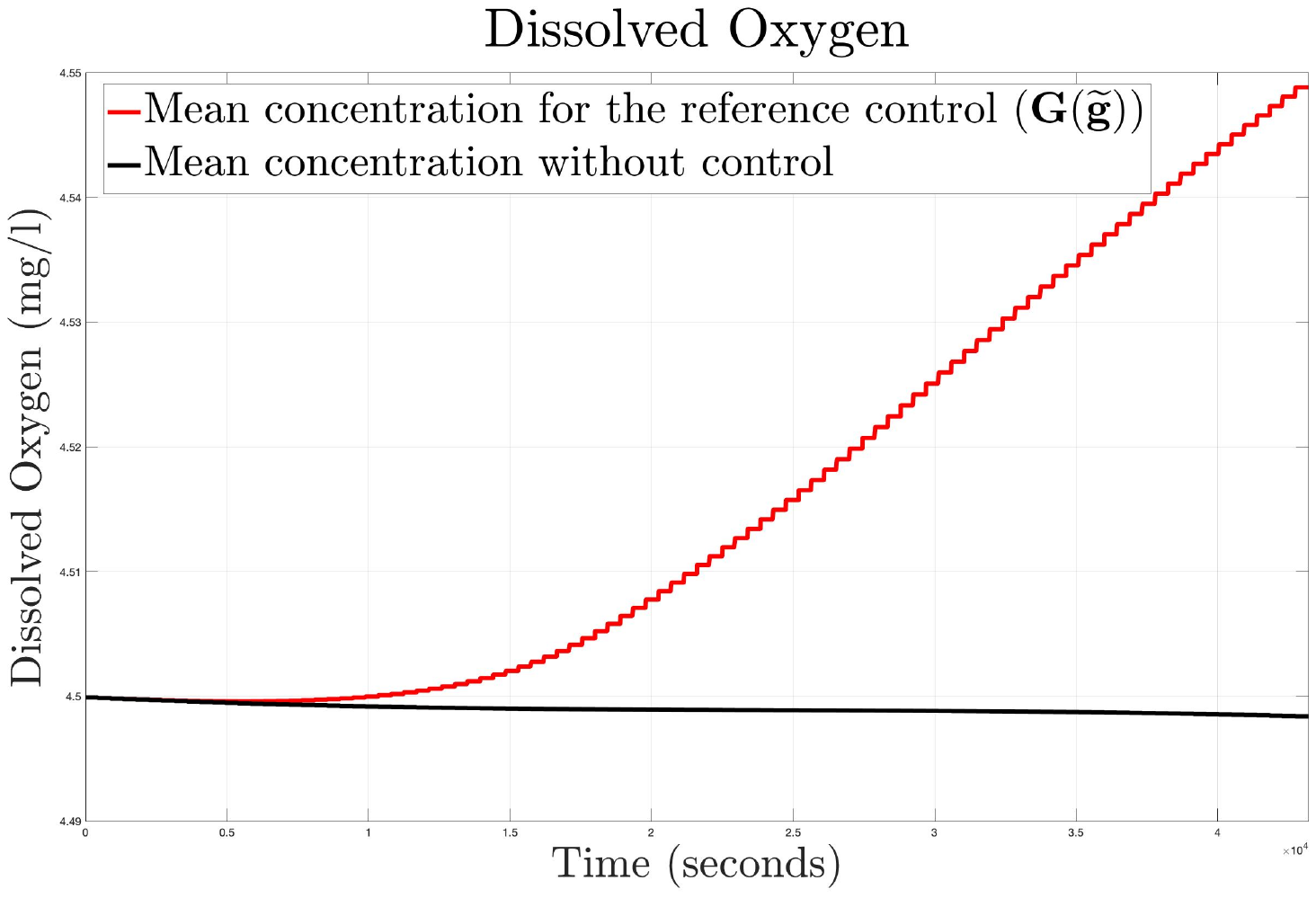}
\caption{Comparison of the mean concentrations of dissolved oxygen in $\Omega_C$, taking a time step length $\Delta t=450\, {\rm s}$, 
for a constant flow rate of $1.0 \times 10^{-4} \, {\rm m}^3 \, {\rm s}^{-1}$ in all the pumps, and for the case without pumping.}
\label{figure5}
\end{figure}

In this final part of the Section we present several numerical results that we have obtained using 
different choices of the time step length $\Delta t$. We must mention that in the numerous numerical tests developed,
we have always obtained that $\mathbf{G}(\widehat{\mathbf{g}})=\mathbf{G}(\widetilde{\mathbf{g}})$, and also a 
reduction in the value of the cost functional $J(\widehat{\mathbf{g}}) < J(\widetilde{\mathbf{g}})$. 
So, in Figure~\ref{figure6} we can see the optimal control that we have obtained taking 
$\sigma_1=0.5$ and $\sigma_2=1-\sigma_1=0.5$, for time steps of $\Delta t=3600\, {\rm s}$ and $\Delta t=1800\, {\rm s}$
(corresponding to $N=12$ and $N=24$, respectively). 
In Figure~\ref{figure7} we can find the optimal control corresponding to time steps of $\Delta t=900 \, {\rm s}$ and $\Delta t=450\, {\rm s}$
($N=48$ and $N=96$, respectively), showing the robustness of our methodology.
 
\begin{figure}[!ht]
\centering
\includegraphics[width=0.49\textwidth]{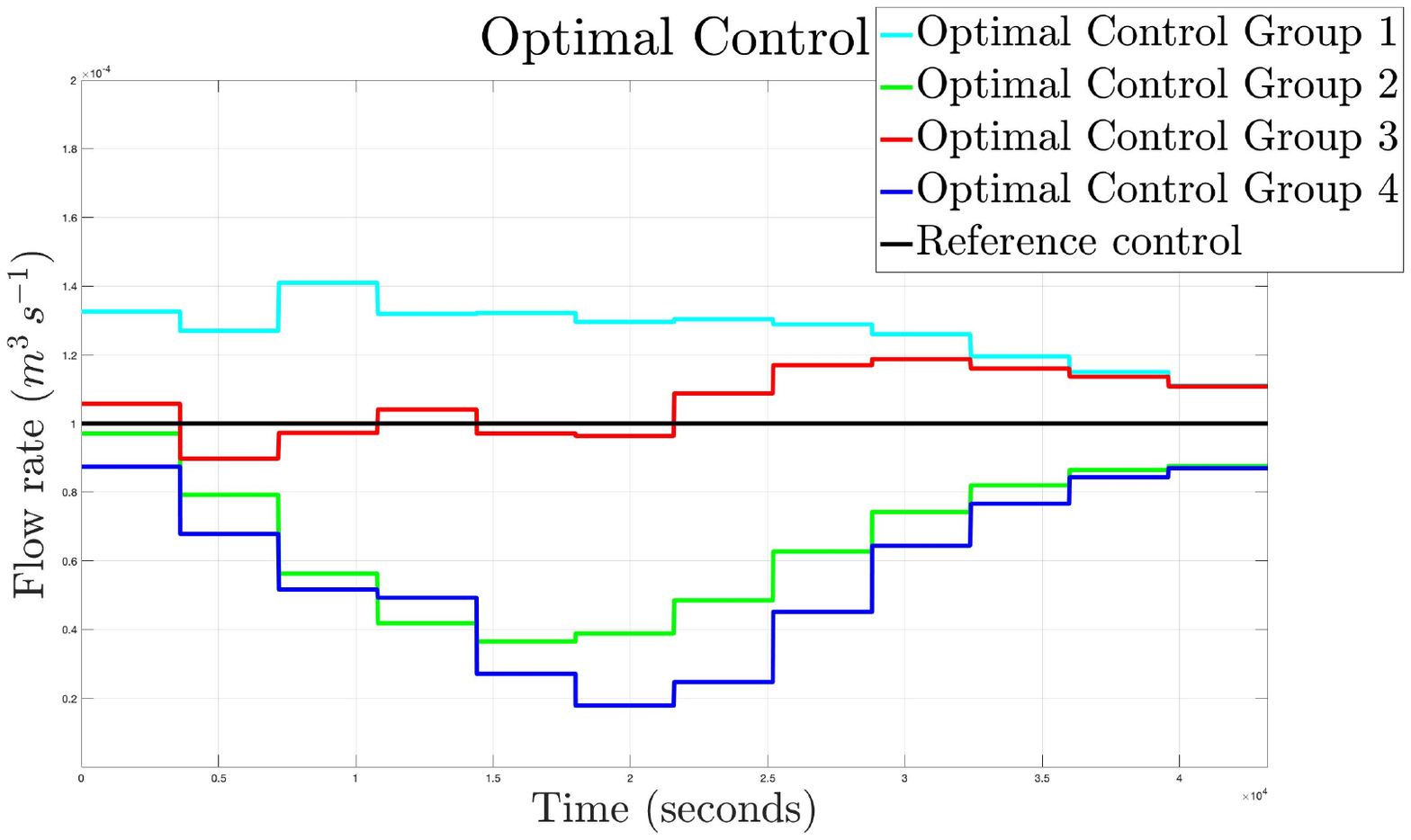}
\includegraphics[width=0.49\textwidth]{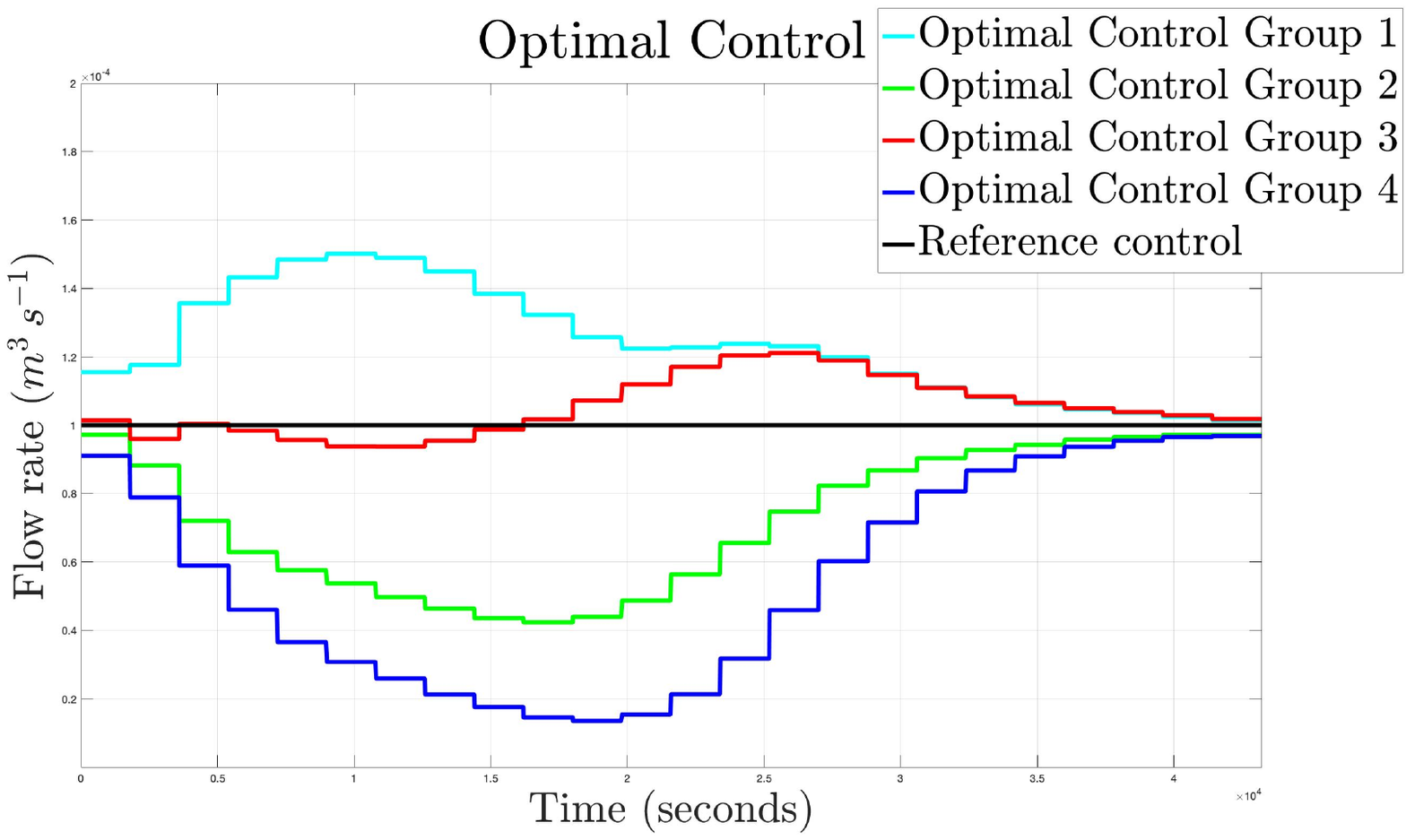}
\caption{Evolution of the optimal flow rates for the four pumps taking weights $\sigma_1 = \sigma_2 = 0.5$, and
$\Delta t=3600 \, {\rm s}$ (left) or $\Delta t=1800 \, {\rm s}$ (right).}
\label{figure6}
\end{figure}

\begin{figure}[!ht]
\centering
\includegraphics[width=0.49\textwidth]{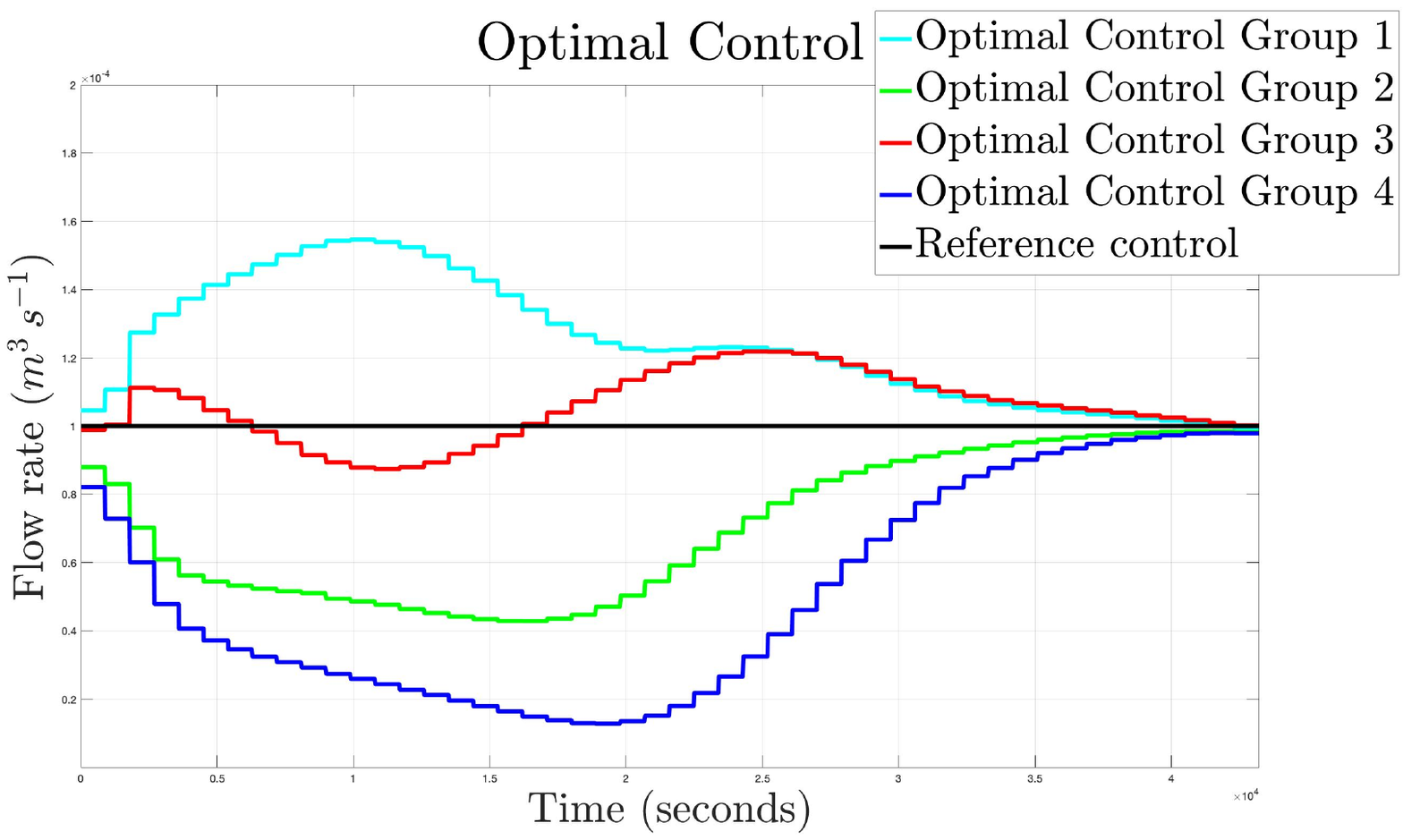}
\includegraphics[width=0.49\textwidth]{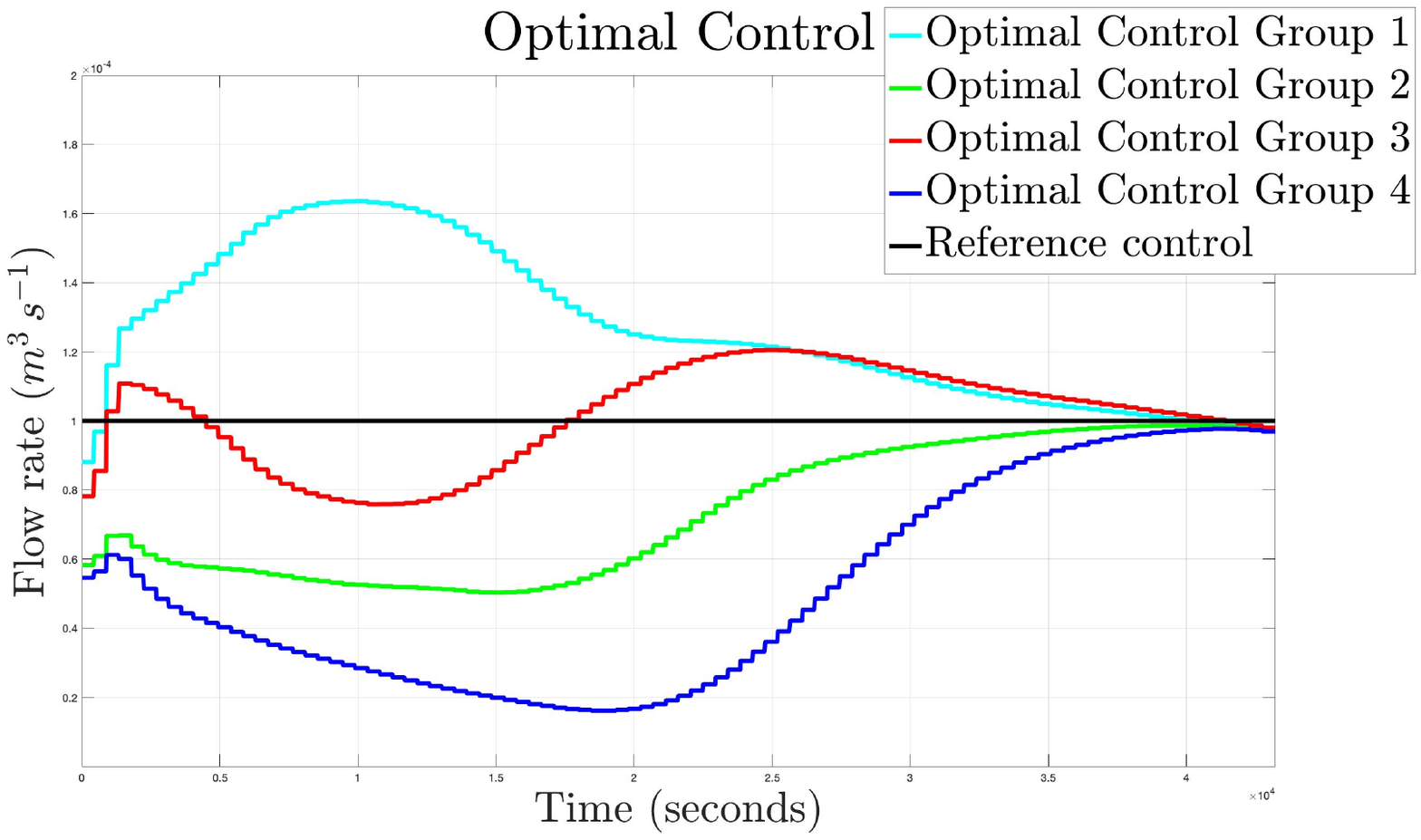}
\caption{Evolution of the optimal flow rates for the four pumps taking $\sigma_1 = \sigma_2 = 0.5$, and
$\Delta t=900 \, {\rm s}$ (left) or $\Delta t=450 \, {\rm s}$ (right).}
\label{figure7}
\end{figure}

We observe that the flow rates associated to the two upper collectors ($g^1$ and $g^3$) are significantly higher 
than the corresponding to lower collectors ($g^2$ and $g^4$). This is caused by the fact that the photosynthesis 
is more intense in the superficial layers and, consequently, the presence of dissolved oxygen is higher there. 

{\color{red} In Table~\ref{table1} we can see the comparison between the functional cost evaluated 
in the reference control and in the optimal control. We can observe that as we decrease the time step, 
the difference between the reference cost and the optimal cost increases. This is because as we decrease 
the time step, we can act more precisely over the system and achieve better results.
\begin{table}[!ht]
\centering
\begin{tabular}{ccccc}
& $\Delta t=3600$ s & $\Delta t=1800$ s & $\Delta t=900$ s & $\Delta t=450$ s \\ \hline
$J(\widetilde{\mathbf{g}})$ & $1.2000e-07$ & $2.4000e-07$ & $4.8000e-07$ & $9.6000e-07$\\
$J(\widehat{\mathbf{g}})$ & $1.0973e-07$ & $2.1865e-07$ & $4.3195e-07$ & $8.7104e-07$
\end{tabular}
\caption{\label{table1} Functional cost evaluated in the Reference Control ($\widetilde{\mathbf{g}}$) 
vs. Optimal Control ($\widehat{\mathbf{g}}$).}
\end{table}}

In Figure~\ref{figure8} we can see the evolution of the constraints for the choice of the time step length $\Delta t=450\, {\rm s}$. 
We can verify there that the optimal constraint $\mathbf{G}(\widetilde{\mathbf{g}})$ and the reference constraint 
$\mathbf{G}(\widehat{\mathbf{g}})$ are virtually indistinguishable, that is, with optimal strategy $\widetilde{\mathbf{g}}$ we obtain 
the same water quality in the control region as with the constant reference flow rate $\widetilde{\mathbf{g}}$, 
but with a significative decrease in energy cost.

\begin{figure}[!ht]
\centering
\includegraphics[width=0.65\textwidth]{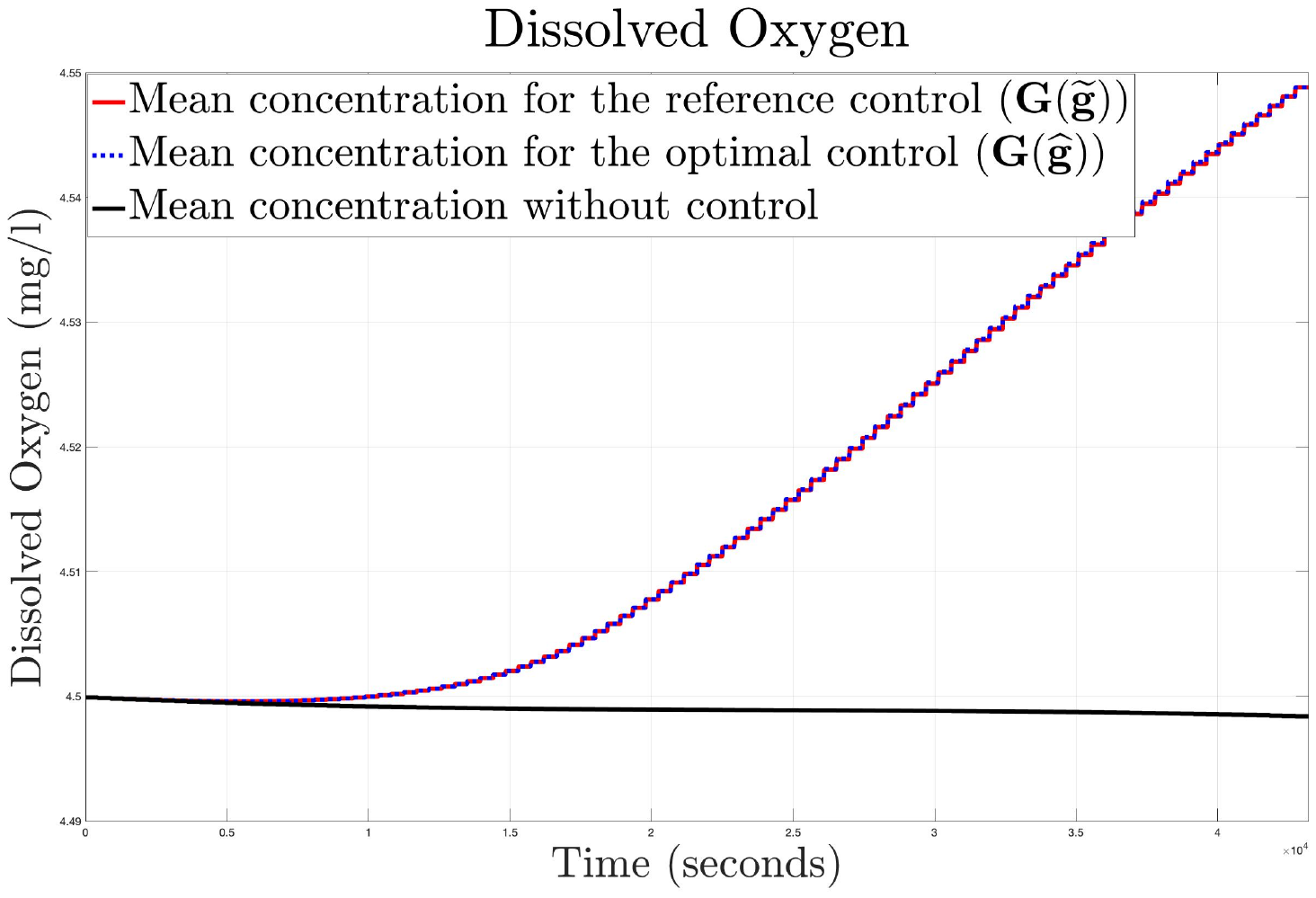}
\caption{Evolution of the constraints, for $\Delta t=450\, {\rm s}$, in the controlled and uncontrolled cases.}
\label{figure8}
\end{figure}

Finally, in Figure~\ref{figure9} we show the concentration of dissolved oxygen in the whole domain $\Omega$
associated to the optimal control solution for $\Delta t=450\, {\rm s}$ (left),
and the concentration of dissolved oxygen when all the pumps are off (right), 
both in the last time step (corresponding to $N=96$). 
We can easily notice here the pumping effects associated to the optimal control in the bottom layer, with an evident
improvement of water quality in the region. 

\begin{figure}[!ht]
\centering
\includegraphics[width=0.49\textwidth]{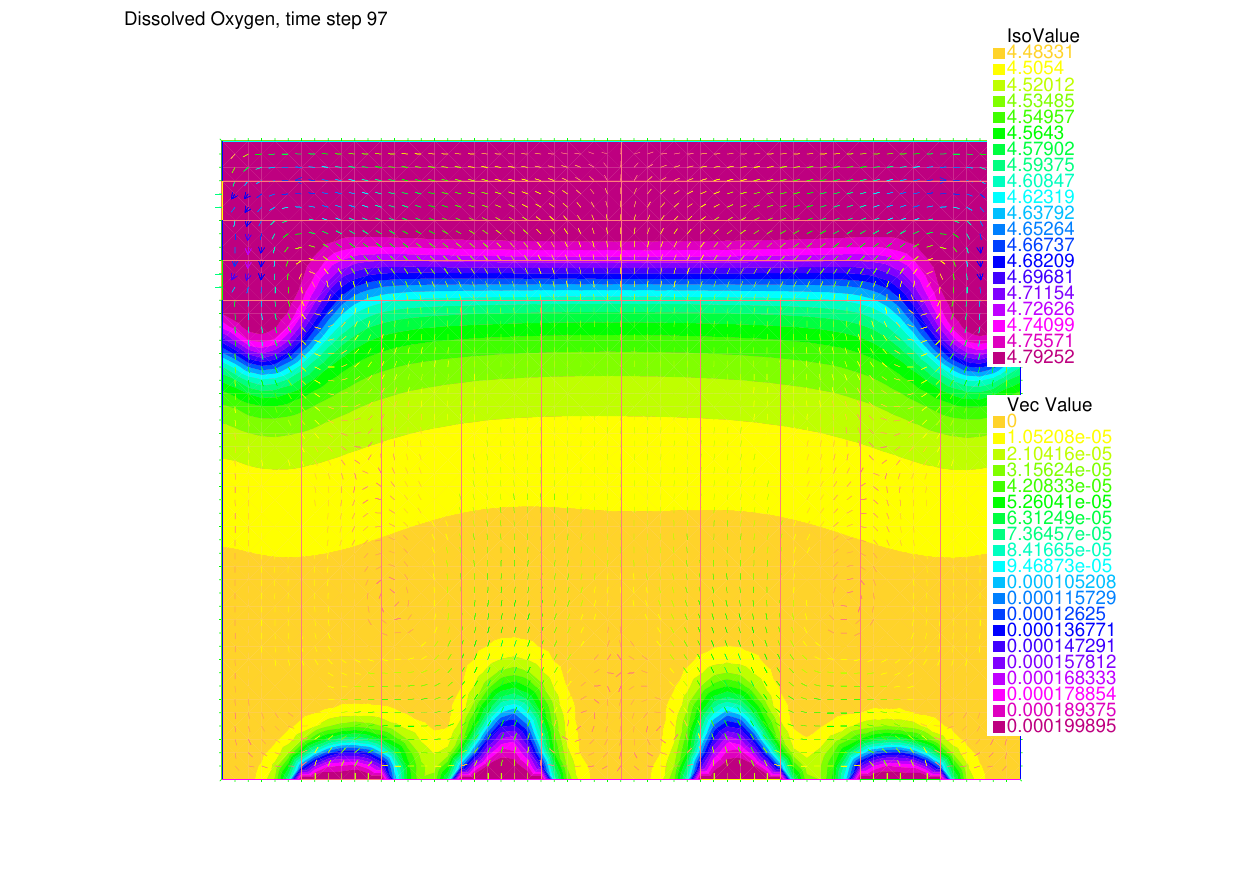}
\includegraphics[width=0.49\textwidth]{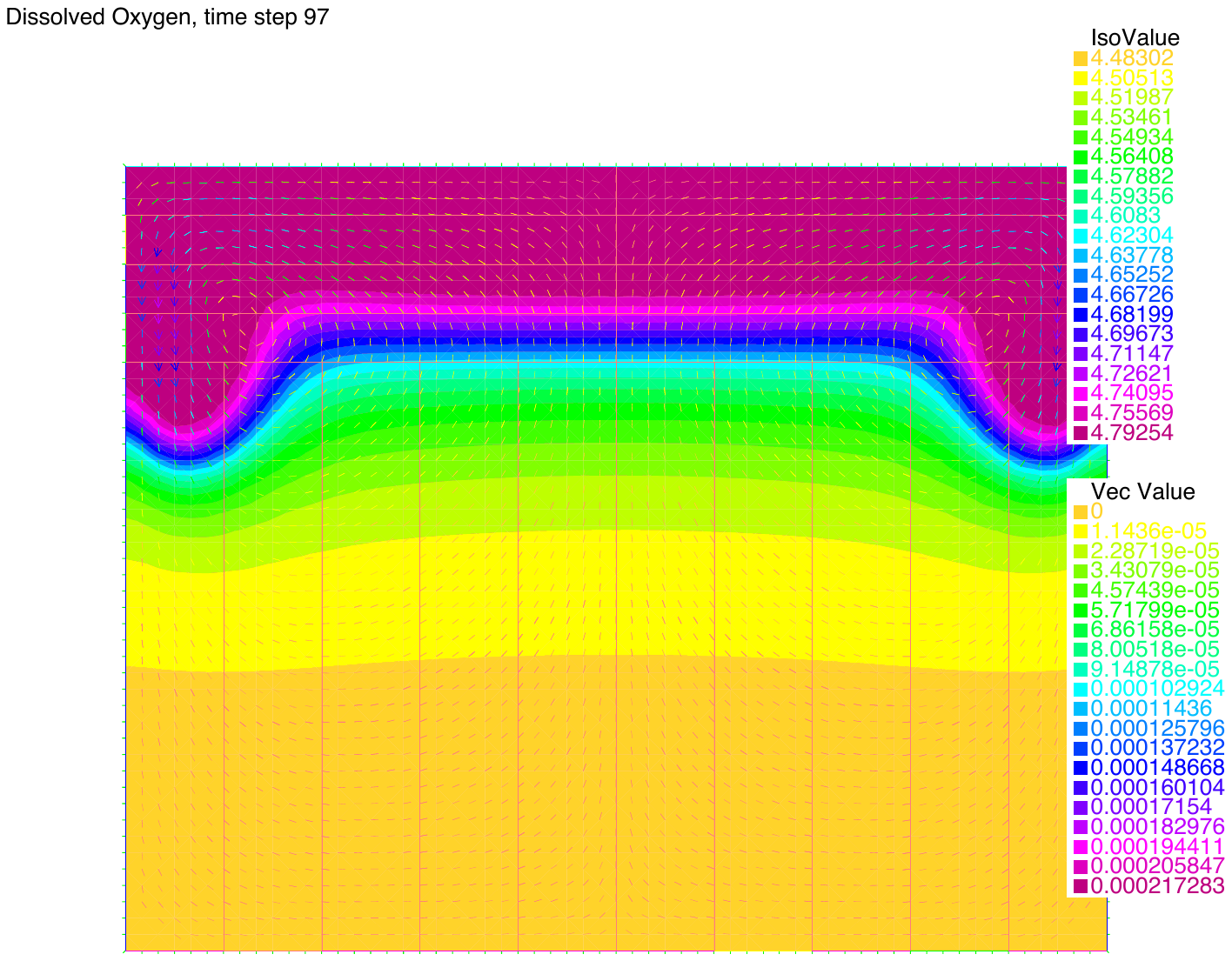}
\caption{Concentration of dissolved oxygen in the last time step  
corresponding to the optimal solution (left), and without control (right).}
\label{figure9}
\end{figure}





\end{document}